\documentclass[11pt, oneside]{article}

\usepackage{url}
\usepackage[title]{appendix}
\usepackage{amsmath,amstext,amssymb,amsfonts,amsthm}
\usepackage{xcolor}
\usepackage{graphicx}
\usepackage{tikz}
\usetikzlibrary{calc,angles,quotes,arrows.meta,decorations.pathmorphing,fadings}
\tikzset{resistor/.style={decorate,decoration={zigzag,segment length=4pt,amplitude=3pt,pre length=3pt,post length=3pt}}}
\usepackage[american]{circuitikz}
\ctikzset{bipoles/length=0.55cm, resistors/zigs=3, resistors/thickness=0.8}
\tikzfading[name=fade right, left color=transparent!0, right color=transparent!100]
\tikzfading[name=fade left,  right color=transparent!0, left color=transparent!100]
\tikzfading[name=fade down, bottom color=transparent!0, top color=transparent!100]
\tikzfading[name=fade up,   top color=transparent!0, bottom color=transparent!100]
\usepackage{booktabs}
\usepackage{multirow}
\usepackage{siunitx}
\usepackage[labelfont=bf]{caption}
\usepackage{subcaption}
\usepackage{comment}

\newtheorem{prop}{Proposition}[section]

\newtheorem{rem}{Remark}[section]
\newtheorem{Def}{Definition}[section]
\numberwithin{equation}{section}
\numberwithin{figure}{section}
\numberwithin{table}{section}
\allowdisplaybreaks[4]
\newcommand{\Om}{\Omega}
\newcommand{\Gm}{\Gamma}
\newcommand{\Km}{\mathbf{K}_m}
\newcommand{\bn}{\mathbf{n}}
\newcommand{\bu}{\mathbf{u}}

\newcommand{\grad}{\nabla}
\newcommand{\jump}[1]{[\![#1]\!]}
\newcommand{\Th}{\mathcal{T}_h}
\newcommand{\ThG}{\mathcal{T}_h^{\Gm}}
\newcommand{\transp}{{\mathsf T}}
\newcommand{\egamma}{\mathbf{e}_\gamma}
\newcommand{\uT}{\mathbf{u}_T}

\title{An unfitted finite element discrete fracture model for low-permeability barriers via local stiffness matrix modification}
\author{Ziyao Xu\footnotemark[1]}
\date{}

\begin{document}

\maketitle
\renewcommand{\thefootnote}{\fnsymbol{footnote}}
\footnotetext[1]{Department of Mathematics and Statistics,
Binghamton University, Binghamton, NY 13902, USA. E-mail: zxu24@binghamton.edu}

\begin{center}
\small
\begin{minipage}{0.9\textwidth}
\textbf{Abstract.}
Finite element methods are among the most widely used discretizations for flow in porous media, and their discrete fracture models (DFMs) for highly conductive fractures are well-established. 
Low-permeability barriers, by contrast, have long resisted this framework because they induce pressure discontinuities that continuous elements cannot represent directly.
In this paper, we propose a simple extension of the linear finite element method for modeling low-permeability barriers through a closed-form modification of the local stiffness matrix of each barrier-cut element.
The method retains exactly the same $H^1$-conforming $P^1$-finite element space, preserves the original sparsity pattern and symmetric positive definiteness, requires no mesh fitting, and coincides with the standard finite element method away from barriers.
Moreover, an inexpensive, purely local post-processing step recovers the discontinuous pressure field with a sharp jump at the barrier interface, thereby removing the one-element-wide smearing present in the continuous solution.
Convergence studies with manufactured solutions and two- and three-dimensional benchmark problems from the literature confirm the effectiveness of the method.


\end{minipage}
\end{center}
\setlength{\parindent}{2em}

\pagenumbering{arabic}

\section{Introduction}\label{sec:intro}
Single-phase flow in porous media is governed by Darcy's law. An accurate prediction of such flows is fundamental to many applications, including groundwater management, oil and gas recovery, and geothermal energy development.
Natural formations, however, are rarely homogeneous. Geological processes such as faulting, as well as human activities such as fracking, can generate extensive fracture systems in subsurface rocks. 
Open fractures and fractures filled with proppants are often much more permeable than the surrounding matrix and therefore act as highly conductive \emph{fractures}. In contrast, fractures sealed by mineral cement can be much less permeable than the matrix and form \emph{barriers} that block flow. Both types of structures are common in subsurface formations and can fundamentally impact the direction and rate of fluid flow. Numerical simulations must therefore represent both effects accurately.

Because fracture apertures are often several orders of magnitude smaller than the field scale, a full-dimensional representation of fractures usually leads to excessive computational cost and severely ill-conditioned linear systems. A common alternative is the method of discrete fracture models (DFMs), which adopt mixed-dimensional approaches to reduce each fracture to a codimension-one interface embedded in the surrounding full-dimensional porous matrix \cite{martin2005modeling}.

For highly conductive fractures, the finite element DFM \cite{noorishad1982upstream, baca1984modelling, kim2000finite, karimi2003numerical, zhang2013accurate} is one of the earliest and primary classes of discretizations.
Its main idea is simple. Assuming that pressure is continuous across a fracture, the reduced fracture model contributes only a tangential diffusion term supported on the interface in the weak formulation. 
A lower-dimensional stiffness term can therefore be added directly to the bulk Galerkin system on a fitted mesh.\footnote{Terminology varies in the literature. Here, a mesh is called \emph{fitted} if its edges or faces align with the fracture, and \emph{unfitted} otherwise. We avoid the alternative term \emph{nonconforming} to prevent confusion with nonconforming finite element spaces.}
The fracture pressure is represented using the nodal unknowns already present in the mesh. As a result, no additional unknowns are introduced, no interface conditions need to be imposed explicitly, and the assembled system remains symmetric positive definite.
Since the method only requires an additional elementwise assembly loop, it can be incorporated into a general finite element framework with little modification. These methods are mainly used on fitted meshes, but an extension to unfitted meshes is also feasible \cite{xu2020hybrid}. 

Barriers are fundamentally different. Their effect is associated with a pressure jump across the interface, which cannot be represented by a continuous function space. Standard finite element methods therefore do not recognize the barrier unless the formulation or approximation space is modified. Benchmark studies show that this is not a minor issue.
In the three-dimensional comparison by Berre et al.\ \cite{berre2021verification}, the barrier test clearly separated the performance of different methods. Methods that assumed continuous hydraulic head across the fracture, including the standard finite element DFM built in COMSOL and the Lagrange multiplier approaches in \cite{koppel2019lagrange, schadle20193d}, failed to capture the pressure jump.
The only continuous finite element-type method that agreed with the reference solution resolved the barrier as an equi-dimensional region with finite thickness \cite{favino2020fully}. Thus, within the finite element framework, the treatment of barriers has not yet become a standard capability.

Outside the finite element framework, barriers have been studied more extensively, but most methods rely on fitted meshes. When element faces are fitted to the barriers, mixed-dimensional methods can impose the interface conditions directly \cite{martin2005modeling}; mortar methods can further extend this treatment to nonmatching meshes \cite{frih2012modeling, boon2018robust}. 
Cell-centered finite volume methods incorporate the barrier resistance into the transmissibility \cite{karimi2004efficient, angot2009asymptotic}, while vertex-centered box methods recover the pressure jump by breaking the trial space across barrier faces \cite{glaser2022comparison, xu2024box}. More recently, high-order finite volume methods \cite{liu2026high} and interior penalty discontinuous Galerkin methods \cite{liu2026interior} have also been developed within the same fitted setting.

When barriers cut through the background mesh, the range of available methods becomes much narrower, and most existing approaches introduce additional local structures to represent the pressure jump. 
XFEM-type methods \cite{d2012mixed,fumagalli2013numerical,schwenck2015dimensionally,zhao2018modeling,cervera2022comparative} enrich the approximation space on cut elements to represent interface-induced discontinuities and to impose the interface conditions weakly, whereas immersed finite element methods incorporate the interface conditions through modified local basis functions \cite{zhao2024discrete,zhao2026petrov}. 
For general elliptic interface problems, interface-penalty methods instead duplicate the approximation space on cut elements and impose Robin jump conditions using Nitsche-type terms \cite{li2022high}.
These methods avoid globally fitted meshes, but rely on geometry-dependent enrichment or coupling constructions. 
The embedded discrete fracture model constitutes another important family of unfitted methods. While the original EDFM primarily targets conductive fractures \cite{li2008efficient,moinfar2014development}, pEDFM extends it to barriers by modifying transmissibility connections according to the fracture projections \cite{tene2017projection}.
Subsequent variants improve the representation of inclined barriers and near-fracture flow through continuous projections, local fine-scale problems, or enriched pressure approximations \cite{rashid2024continuous,losapio2023local,jiao2024enriched}. 
The RDFM instead incorporates barriers into hybrid-dimensional resistance tensors and discretizes the resulting model directly on unfitted meshes \cite{xu2023hybrid, fu2023hybridizable}.

The aforementioned methods are effective within their respective settings, but each departs in some way from the standard $H^1$-conforming finite element method. This leads to the question studied in this paper: can low-permeability barriers be captured by only a \emph{minimal} modification of the standard finite element method, without introducing additional degrees of freedom, enriching or modifying the approximation space, or requiring a barrier-fitted mesh?
We give an affirmative answer in this paper.

The proposed method incorporates barriers through a simple closed-form modification of the local stiffness matrices. Its construction is based on a variational principle. The barrier effect enters the energy functional through a Robin-type pressure-jump term supported on the barrier interface. For each element $T$ cut by a barrier, an auxiliary variable is introduced to represent the pressure jump that is missing from the original continuous approximation space. This variable is then eliminated by a local stationarity condition. 
In single-barrier cases, the resulting method is the following closed-form rank-one correction of the local stiffness matrix:
\[
K_T\ \longmapsto\ K_T-\frac{(K_T\mathbf e_\gamma)(K_T\mathbf e_\gamma)^{\mathrm T}}{(K_T)_{\gamma\gamma}+\ell_T/R},
\]
where $K_T$ is the standard local stiffness matrix on $T$, and the other quantities will be defined precisely later. In implementation, only a few lines of code need to be added to the standard stiffness-matrix assembly loop. 
The computed pressure is continuous and spreads pressure drop over one element width, which converges to the exact jump under mesh refinement.
After a local post-processing step, the interface jump can be recovered, and we obtain a piecewise discontinuous reconstructed pressure in which each pressure drop is placed directly on the barrier.
The method is designed to be integrated directly into the classical finite element discrete fracture model. Highly conductive fractures are incorporated by adding tangential stiffness contributions to the global stiffness matrix, while the barrier treatment proposed here modifies the local stiffness matrices. An existing FEM-DFM code can be extended to barriers by modifying only the local stiffness matrices of barrier-cut elements, while leaving its data structures, linear solvers, and treatment of highly conductive fractures unchanged. 

The remainder of the paper is organized as follows. Section~\ref{sec:model} states the interface model and its
broken energy functional. Section~\ref{sec:method} derives the method by a variational argument and
local condensation, establishes its structural properties, and treats barrier networks.
Section~\ref{sec:recovery} develops the post-processing recovery of the pressure jumps.
Section~\ref{sec:num} reports numerical experiments, including manufactured convergence tests and published benchmark problems. Section~\ref{sec:conclusion} concludes. Appendix~\ref{app:network} presents a resistor-network interpretation of the method,
and Appendix~\ref{app:threed} details the three-dimensional construction.
\section{The interface model and its broken energy functional}\label{sec:model}

\subsection{Model problem}\label{sec:modelproblem}

Let $\Om\subset\mathbb{R}^2$ be the computational domain. 
For clarity, we state the interface model for a single barrier that separates the domain into two sides. This is only a notational simplification. Immersed barriers, terminating barriers, and barrier networks are treated by the same interface condition. 
Let $\Gm$ be a (piecewise) straight barrier interface, separating $\Om$ into $\Om^-$ and $\Om^+$. 
For a function $w$ on $\Om^-\cup\Om^+$, define its jump on $\Gamma$ as $\jump{w}:=w^--w^+$, where $w^\pm$ are the traces of $w$ on $\Om^{\pm}$.
Let the unit normal $\bn_{\Gamma}$ be oriented from $\Om^-$ to $\Om^+$. 
See Figure \ref{fig:model-single-barrier} for an illustration. 
The matrix permeability $\Km(x)$ is symmetric positive definite and may be heterogeneous and anisotropic. 
The barrier resistance is
\[
R:=\frac{a}{k_b}>0,
\]
where $a\ll1$ is the aperture and $k_b\ll1$ is the permeability of the barrier.
 
We consider the interface problem of
Martin--Jaffr\'e-Roberts type \cite{martin2005modeling} in the low-permeability limit,
\begin{subequations}\label{eq:reduced-model}
\begin{align}
    -\grad\cdot(\Km\grad p) &= f, && \text{in }\Om^-\cup\Om^+, \label{eq:bulk}\\
    \jump{u_n} &=0, && \text{on }\Gm, \label{eq:fluxcont}\\
    \jump{p} &= R\, u_n, && \text{on }\Gm, \label{eq:robin}
\end{align}
\end{subequations}
where
\[
    \bu=-\Km\grad p,\qquad
    u_n=\bu\cdot\bn_{\Gamma}=-\Km\grad p\cdot\bn_{\Gamma},
\]
completed by $p=p_D$ on $\Gm_D\subseteq\partial\Om$ and $\bu\cdot\bn=q_N$ on
$\Gm_N=\partial\Om\setminus\Gm_D$. 
The normal flux $u_n$ is continuous across the barrier, and the
pressure $p$ drops by $R$ times that flux. 

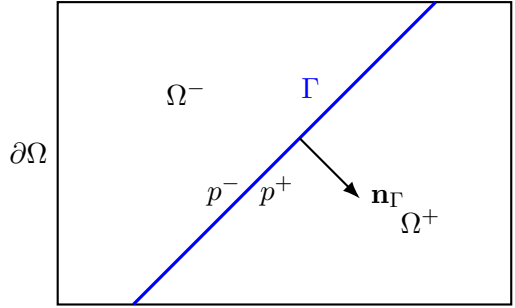
\begin{figure}[htbp!]
\centering
\begin{tikzpicture}[scale=1.0,>=Latex]

\coordinate (A) at (0,0);
\coordinate (B) at (6,0);
\coordinate (C) at (6,4);
\coordinate (D) at (0,4);

\coordinate (P1) at (1.0,0);
\coordinate (P2) at (5.0,4);


\draw[thick] (A)--(B)--(C)--(D)--cycle;

\draw[blue,very thick] (P1)--(P2);

\coordinate (M) at ($(P1)!0.55!(P2)$);
\draw[->,thick] (M) -- ++(0.8,-0.8) node[right] {$\mathbf n_{\Gamma}$};

\node at (1.7,2.8) {$\Omega^{-}$};
\node at (4.8,1.1) {$\Omega^{+}$};
\node[blue,above left] at ($(P1)!0.65!(P2)$) {$\Gamma$};

\node at (2.2,1.5) {$p^{-}$};
\node at (2.9,1.5) {$p^{+}$};


\node[left]  at ($(A)!0.5!(D)$) {$\partial\Omega$};

\end{tikzpicture}
\caption{Model geometry for a single barrier interface.  The barrier $\Gamma$ separates the domain into two sides $\Omega^-$ and $\Omega^+$, with unit normal $\mathbf n_{\Gamma}$ oriented from $\Omega^-$ to $\Omega^+$.  The pressure may jump across $\Gamma$.}
\label{fig:model-single-barrier}
\end{figure}

\subsection{Weak form and broken energy functional}\label{sec:energy}

The natural function space for the interface problem \eqref{eq:reduced-model} is broken across the barrier,
\[
    V=\bigl\{v\in L^2(\Om):\ v|_{\Om^\pm}\in H^1(\Om^\pm)\bigr\},
\]
with $V_D=\{v\in V:\ v=p_D \text{ on }\Gm_D\}$ and the test space $V_0=\{v\in V:\ v=0\text{ on
}\Gm_D\}$. To derive the weak form, multiply \eqref{eq:bulk} by $v\in V_0$ and integrate by
parts on each subdomain separately:
\[
\int_{\Om^\pm}\Km\grad p\cdot\grad v\,d\mathbf{x}
=\int_{\Om^\pm}f v\,d\mathbf{x}
+\int_{\partial\Om^\pm\setminus\Gamma_{D}}(\Km\grad p\cdot\bn_{\rm out})\,v\,ds .
\]
On $\Gm$, the outward normal of $\Om^-$ is $\bn_{\Gamma}$ and that of $\Om^+$ is $-\bn_{\Gamma}$. 
Adding the two
identities, the interior boundary terms combine into a term supported on the barrier,
\[
\sum_{\pm}\int_{\Om^\pm}\Km\grad p\cdot\grad v\,d\mathbf{x}
=\int_{\Om}f v\,d\mathbf{x}-\int_{\Gm_N}q_N v\,ds
-\int_{\Gm}u_n\,\jump{v}\,ds ,
\]
where the flux continuity \eqref{eq:fluxcont} was used to write a single $u_n$. Substituting the
Robin condition \eqref{eq:robin}, $u_n=\jump{p}/R$, and moving the interface term to the left
gives the symmetric weak form \cite{angot2009asymptotic}: find $p\in V_D$ such that
\begin{equation}\label{eq:weakform}
a(p,v):=\sum_{\pm}\int_{\Om^\pm}\Km\grad p\cdot\grad v\,d\mathbf{x}
+\frac1R\int_{\Gm}\jump{p}\,\jump{v}\,ds
=\int_{\Om}fv\,d\mathbf{x}-\int_{\Gm_N}q_Nv\,ds=:\ell(v)
\qquad\forall v\in V_0 .
\end{equation}
The bilinear form $a$ is symmetric and, assuming the uniform ellipticity of $\Km$ and a Dirichlet boundary portion of positive measure, is coercive on $V_0$. 
The problem \eqref{eq:weakform} is therefore equivalent to the minimization of the \emph{broken energy functional}
\begin{equation}\label{eq:continuous-energy}
    J(q)
    =
    \frac12\sum_{\pm}\int_{\Om^\pm}\Km\grad q\cdot\grad q\,d\mathbf{x}
    +
    \frac1{2R}\int_{\Gm}\jump{q}^2\,ds
    -
    \ell(q),
    \qquad
    p=\mathop{\arg\min}_{q\in V_D} J(q).
\end{equation}
The first term is the matrix (bulk) energy on the two sides of the barrier. The second term is
the barrier energy: an interface conductance $1/R$ penalizes the square of the pressure jump. 
This functional is the starting point of the method.

\begin{rem}[Why a conforming space cannot see the barrier]\label{rem:invisible}
If the energy \eqref{eq:continuous-energy} is minimized over a conforming finite element space
$V_h\subset H^1(\Om)$, the term $\frac1{2R}\int_{\Gm}\jump{q}^2\,ds$ vanishes, because every conforming function has
$\jump{v_h}\equiv0$ on $\Gm$; the minimization degenerates to the barrier-free problem and the
barrier is invisible. A conforming method can therefore see the barrier only if the missing jump
is \emph{represented before the final conforming system is assembled}. 
This is precisely what the method of Section~\ref{sec:method} does. It introduces local auxiliary pressure-jump variables and then eliminates them by stationarity conditions.
\end{rem}

\section{The method of local stiffness matrix modification}\label{sec:method}

Let $\Th$ be a triangulation of $\Om$, \emph{not} required to be fitted to $\Gm$, and
let $V_h\subset H^1(\Om)$ be the finite element space of continuous piecewise linear functions on $\Th$. For a
cell $T\in\Th$, the local stiffness matrix contributed by the bulk matrix is
\begin{equation}\label{eq:elstiff}
    (K_T)_{\mu\nu}=\int_T \Km\grad\phi_\mu\cdot\grad\phi_\nu\,d\mathbf{x},
    \qquad \mu,\nu=1,2,3,
\end{equation}
with $\phi_\mu$ the nodal hat functions for local vertices $\mu=1,2,3$ on $T$. 

\subsection{Derivation by local auxiliary variables and elimination}\label{sec:derivation}

We call a cell $T$ \emph{cut} if $\Gm$ crosses its interior, and write $\ThG\subset\Th$ for the
set of cut cells. For a straight (or locally straight) barrier, $\Gm\cap T$ is a single chord
entering and leaving through two different edges, so exactly one vertex, which we call the lone vertex $\gamma=\gamma(T)$, is separated by the barrier from the other two vertices $\alpha$ and $\beta$. We denote the chord length by $\ell_T=|\Gm\cap T|$; see Figure~\ref{fig:cutcell}. 

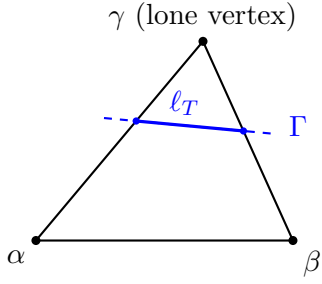
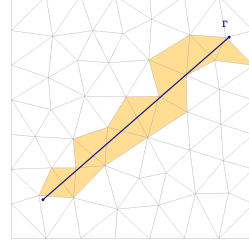
\begin{figure}[htbp!]\centering
\begin{subfigure}[b]{0.42\textwidth}\centering
\begin{tikzpicture}[scale=0.85,>=Latex]
  \coordinate (a) at (0,0); \coordinate (b) at (4,0); \coordinate (d) at (2.6,3.1);
  \coordinate (x1) at ($(a)!0.6!(d)$); \coordinate (x2) at ($(b)!0.55!(d)$);
  \draw[thick] (a)--(b)--(d)--cycle;
  \draw[blue,thick,dashed] ($(x1)!-0.3!(x2)$)--($(x2)!-0.3!(x1)$);
  \draw[blue,very thick] (x1)--(x2);
  \fill[blue] (x1) circle(1.6pt) (x2) circle(1.6pt);
  \foreach \pt in {a,b,d}\fill (\pt) circle(2pt);
  \node[below left] at (a){$\alpha$}; \node[below right] at (b){$\beta$};
  \node[above] at (d){$\gamma$ (lone vertex)};
  \node[blue,above] at ($(x1)!0.45!(x2)$){$\ell_T$};
  \node[blue] at ($(x2)!-0.32!(x1)+(0.32,0.12)$){$\Gm$};
\end{tikzpicture}
\caption{a cut cell $T=(\alpha,\beta,\gamma)$}
\label{fig:cutcell-a}
\end{subfigure}\hfill
\begin{subfigure}[b]{0.5\textwidth}\centering
\includegraphics[width=0.5\textwidth]{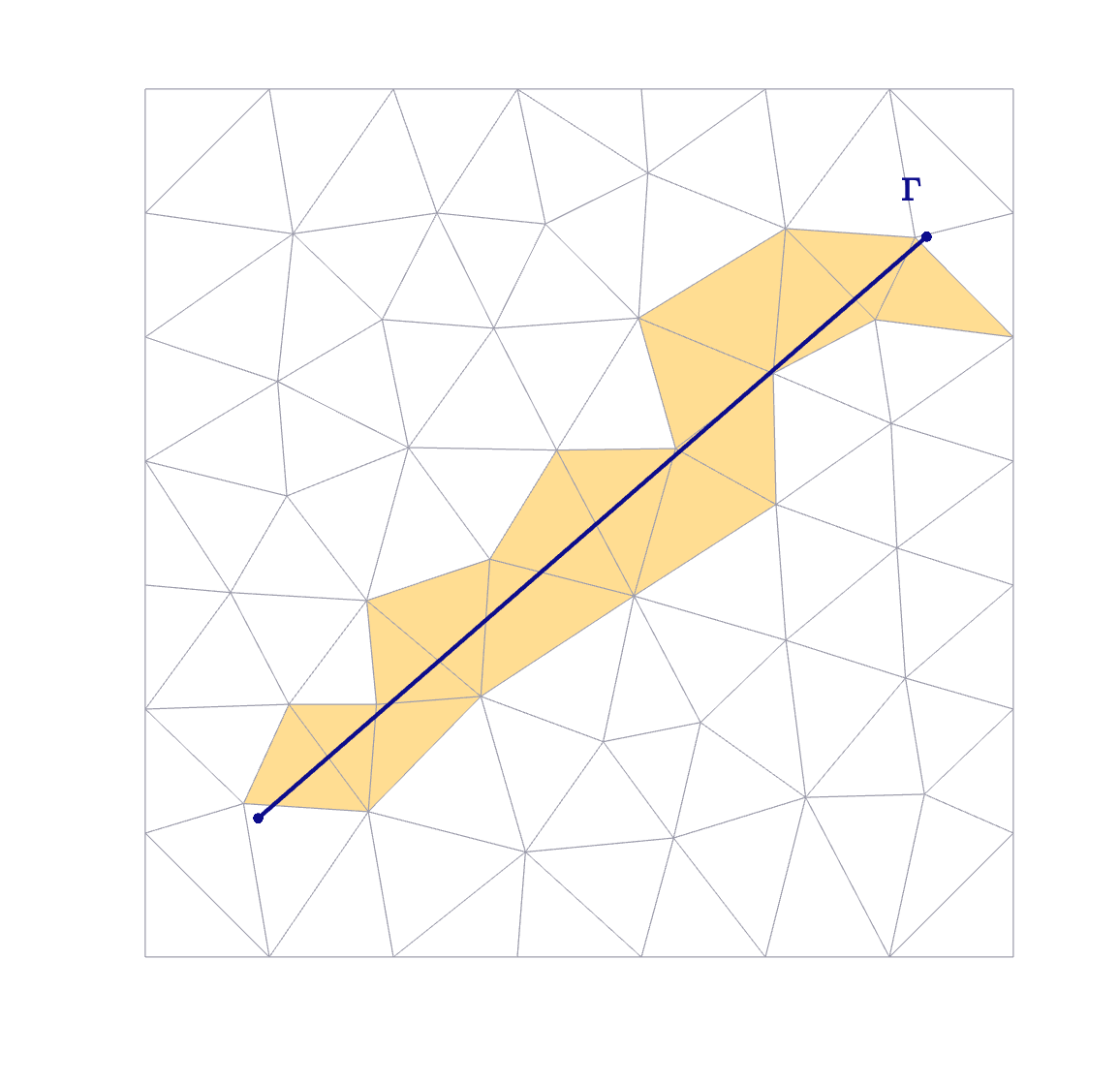}
\caption{a barrier on an unstructured, unfitted mesh}
\label{fig:cutcell-b}
\end{subfigure}
\caption{Geometry of the method. (a) In a cut cell the barrier chord (blue, length $\ell_T$)
separates the lone vertex $\gamma$ from $\alpha,\beta$. (b) An oblique barrier $\Gm$ crossing an unstructured
Delaunay mesh it is not fitted to. 
The cut cells are shaded, and only their element stiffness is modified.
The cells with only a free tip are left unchanged.}
\label{fig:cutcell}
\end{figure}

By Remark~\ref{rem:invisible}, the missing ingredient of the $H^1$-conforming finite element space is the jump, and we represent it locally. 
On each cut cell $T\in\ThG$, introduce one scalar auxiliary variable
$t_T$, interpreted as the pressure drop across the chord $\Gm\cap T$. The matrix (bulk) part of the energy \eqref{eq:continuous-energy} should act on the pressure after this drop has been removed (as if the barrier never existed). 
With the convention that the correction is applied at the lone vertex, the matrix (bulk) field represented in the finite element space is the jump-corrected nodal vector
\[
    \uT^{\rm bulk}=\uT-t_T\egamma,
    \qquad
    \uT=(u_1,u_2,u_3)^{\transp},
\]
where $\egamma$ is the $\gamma$-th coordinate vector in $\mathbb R^3$.

The local discrete energy functional on a cut cell $T\in\ThG$ is therefore
\begin{equation}\label{eq:ET}
    \mathcal E_T(\uT,t_T)
    =
    \frac12(\uT-t_T\egamma)^{\transp}K_T(\uT-t_T\egamma)
    +
    \frac{\ell_T}{2R}\,t_T^2 .
\end{equation}
The first term is the discrete counterpart of the bulk energy $\frac12\int_{T\setminus\Gamma}\Km\grad q\cdot\grad q\,d\mathbf{x}$. The second term
is the chord quadrature of the Robin jump energy: it approximates
$\frac1{2R}\int_{\Gm\cap T}\jump{q}^2ds$ with one sample of the jump per cell. 
The discrete energy of \eqref{eq:continuous-energy} is then
\begin{equation}\label{eq:Jh}
J_h\bigl(u_h,(t_T)_{T\in\ThG}\bigr)
=\sum_{T\notin\ThG}\frac12\,\uT^{\transp}K_T\uT
+\sum_{T\in\ThG}\mathcal E_T(\uT,t_T)
-\ell(u_h),
\qquad u_h\in V_{h,D}.
\end{equation}

The jump variable $t_T$ is not meant to be retained in the global system. We can eliminate it by minimizing $\mathcal E_T(\mathbf{u}_{T}, t_T)$ with
respect to $t_T$ at fixed nodal vector $\uT$,
\begin{equation}\label{eq:stationarity}
    0=
    \frac{\partial \mathcal E_T}{\partial t_T}
    =
    -\egamma^{\transp}K_T(\uT-t_T\egamma)
    +\frac{\ell_T}{R}\,t_T ,
\end{equation}
The solution of this stationarity equation is explicit:
\begin{equation}\label{eq:tstar}
    t_T^\ast=
    \frac{\egamma^{\transp}K_T\uT}{(K_T)_{\gamma\gamma}+\ell_T/R}.
\end{equation}
The elimination is legitimate because it is local: each $t_T$ appears in exactly one term of \eqref{eq:Jh}, so minimizing $J_h$ jointly over $\bigl(u_h,(t_T)_{T\in\ThG}\bigr)$ may be performed hierarchically, and the inner minimization over each $t_T$ is the cellwise problem \eqref{eq:stationarity}. 
Therefore, the minimizer of the local functional \eqref{eq:ET} is also the minimizer of the global functional \eqref{eq:Jh}. 
Substituting \eqref{eq:tstar} back into \eqref{eq:ET} condenses the discrete cell energy functional to a quadratic form in the retained unknowns,
\begin{equation}\label{eq:locally_condensed_energy}
\widetilde{\mathcal{E}}_T(\uT)=\min_{t_T}\mathcal{E}_T(\uT,t_T)
    =\frac12\,\uT^{\transp}\widetilde K_T\,\uT,
\end{equation}
with the \emph{cut-cell barrier stiffness}
\begin{equation}\label{eq:cellupdate}
    \widetilde K_T
    =
    K_T-
    \frac{(K_T\egamma)(K_T\egamma)^{\transp}}
    {(K_T)_{\gamma\gamma}+\ell_T/R},
\end{equation}
which is a rank-one modification of the local stiffness matrix of the matrix (bulk) field. Note that $K_T\egamma$ is simply the
$\gamma$-th column of $K_T$, so \eqref{eq:cellupdate} is assembled from data the
finite element code already has.

\begin{Def}[The method of local stiffness matrix modification]\label{def:method}
The method solves the $H^1$-conforming $P^1$-FEM system in which every cut cell $T\in\ThG$
assembles the modified local stiffness matrix $\widetilde K_T$ of \eqref{eq:cellupdate} in place of $K_T$;
uncut cells and the load vector are untouched. The discrete solution is
$p_h=\arg\min_{u_h\in V_{h,D}} \bigl[\tfrac12\sum_T\uT^{\transp}\widetilde K_T\uT-\ell(u_h)\bigr]$
(with $\widetilde K_T=K_T$ off the barrier).
\end{Def}

The final linear system has exactly the same nodal unknowns and the same sparsity pattern. In
an existing code the entire method is the following insertion in the assembly loop, executed for
the $O(1/h)$ cut cells:
\begin{center}
\texttt{g = K(:,gamma);\quad K = K - g*g'/(K(gamma,gamma) + ellT/R);}
\end{center}

A cell containing only an isolated free tip (the barrier ends in its interior) is left unmodified.
Cells crossed by several barrier pieces, or cells containing a junction, are handled by Section~\ref{sec:junctions}, where the auxiliary variables are local pressure offsets attached to the subregions of the cut cell.

\begin{prop}[Stability]\label{prop:spd}
For any symmetric positive definite $\Km$, any triangulation, and any
$R,\ell_T>0$, the modified local stiffness matrix $\widetilde K_T$ is symmetric and positive semidefinite with kernel the multiples of $\mathbf 1$. Consequently, the assembled global stiffness
matrix of Definition~\ref{def:method} is symmetric positive definite on the free nodes whenever
that of the standard $P^1$-FEM is.
\end{prop}
\begin{proof}
Symmetry is obvious from \eqref{eq:cellupdate}, and $K_T\mathbf 1=0$ implies $\widetilde K_T\mathbf 1=0$. For every $\uT$, $\tfrac12\uT^{\transp}\widetilde K_T\uT=\min_{t}\mathcal E_T(\uT,t)\ge0$, since $\mathcal E_T\ge0$. 
If $\uT^{\transp}\widetilde K_T\uT=0$, both
terms of $\mathcal E_T(\uT,t_T)$ vanish; $\ell_T/R>0$ forces $t_T=0$, and then $\uT$ lies in the
kernel of $K_T$, i.e., $\uT$ is a muitlple of $\mathbf{1}$. The global statement follows by assembling elementwise
semidefinite matrices with common kernel the multiples of one and applying the Dirichlet boundary
condition.
\end{proof}

An observation is immediate from \eqref{eq:cellupdate}: as $R\to0$ the update vanishes and the plain finite element method is recovered; as $R\to\infty$, $\widetilde K_T\to K_T-(K_T\egamma)(K_T\egamma)^{\transp}/(K_T)_{\gamma\gamma}$, whose $\gamma$-row and column vanish identically.
Therefore the lone vertex is sealed off from the cell, an exact impermeable limit. 

Since the retained pressure is a continuous,
single-valued field, its accuracy against a discontinuous exact solution is limited by
approximation theory: for a jump of size $O(1)$, no continuous piecewise polynomial does
better than $O(h)$ in $L^1(\Om)$ and $O(h^{1/2})$ in $L^2(\Om)$. 
As shown in the numerical tests, the method attains precisely these best-approximation rates.

\subsection{Barrier networks}\label{sec:junctions}
The rank-one modification above treats the elementary case in which a single barrier chord cuts a triangle into two subregions.
When barrier pieces meet inside an element, the interface model of Section~\ref{sec:model} needs no new physics.
The broken energy \eqref{eq:continuous-energy} restricted to $T$ is the sum of the bulk contributions in the subregions and the Robin-type pressure-jump penalties along the barrier pieces.
However, the local element topology becomes richer: a junction cell may be partitioned into more than two subregions.
The key question is: what replaces the single auxiliary pressure-drop variable $t_T$ when one element contains several barrier-induced subregions?

Let the barriers partition $T$ into subregions $T_0,T_1,\ldots,T_m$, and choose $T_0$ as a base subregion.  
The choice of $T_0$ is arbitrary.
We set $t_0=0$ and attach a local pressure offset $t_\rho$ to each non-base subregion $T_\rho$, $\rho=1,\ldots,m$.  These offsets represent the relative pressure levels of the subregions; see Figure \ref{fig:junction}.
The local bulk energy is evaluated after these offsets are removed from the nodal values. 

\begin{figure}[htbp!]\centering
\begin{subfigure}[b]{0.48\textwidth}\centering
\begin{tikzpicture}[scale=0.92,>=Latex]
  \coordinate (A) at (0,0); \coordinate (B) at (4.6,0); \coordinate (C) at (2.3,3.2);
  \coordinate (P) at (2.3,0.69);
  \coordinate (ca) at (1.35,0); \coordinate (cb) at (3.55,1.6);
  \coordinate (da) at (3.25,0); \coordinate (db) at (1.05,1.6);
  \fill[blue!7]    (A)--(ca)--(P)--(db)--cycle;
  \fill[red!7]     (B)--(cb)--(P)--(da)--cycle;
  \fill[green!9]   (C)--(cb)--(P)--(db)--cycle;
  \fill[orange!18] (ca)--(da)--(P)--cycle;
  \draw[thick] (A)--(B)--(C)--cycle;
  \draw[blue,very thick] (ca)--(cb); \draw[blue,very thick] (da)--(db);
  \fill[blue] (P) circle(1.6pt); 
  \foreach \pt in {A,B,C}\fill (\pt) circle(1.7pt);
  \node[below left] at (A){$\gamma$}; \node[below right] at (B){$\alpha$};
  \node[above] at (C){$\beta$};
  \node at (1.02,0.52){$t_2$}; \node at (3.58,0.52){$t_1$};
  \node at (2.3,1.75){$0$}; \node at (2.3,0.24){$t_3$};
\end{tikzpicture}
\caption{X-junction: four subregions}
\end{subfigure}\hfill
\begin{subfigure}[b]{0.48\textwidth}\centering
\begin{tikzpicture}[scale=0.92,>=Latex]
  \coordinate (A) at (0,0); \coordinate (B) at (4.6,0); \coordinate (C) at (2.3,3.2);
  \coordinate (da) at (1.05,1.6); \coordinate (db) at (3.55,1.6);
  \coordinate (Tp) at (2.3,1.6); \coordinate (tb) at (2.3,0);
  \fill[green!9] (C)--(da)--(db)--cycle;
  \fill[blue!7]  (A)--(tb)--(Tp)--(da)--cycle;
  \fill[red!7]   (B)--(db)--(Tp)--(tb)--cycle;
  \draw[thick] (A)--(B)--(C)--cycle;
  \draw[blue,very thick] (da)--(db); \draw[blue,very thick] (tb)--(Tp);
  \fill[blue] (Tp) circle(1.6pt); 
  \foreach \pt in {A,B,C}\fill (\pt) circle(1.7pt);
  \node[below left] at (A){$\gamma$}; \node[below right] at (B){$\alpha$};
  \node[above] at (C){$\beta$};
  \node at (2.3,2.2){$0$};
  \node at (1.25,0.5){$t_2$}; \node at (3.35,0.5){$t_1$};
\end{tikzpicture}
\caption{T-junction: three subregions}
\end{subfigure}
\caption{Junction cells and their local pressure offsets in each subregion.  
Barrier pieces are shown in blue; each shaded subregion carries a local pressure offset $t_\rho$, with the base subregion set to zero.  
(a) An X-junction: two crossing barriers cut the cell into four subregions.  
(b) A T-junction: a terminating arm meets another barrier and splits the cell into three subregions.}
\label{fig:junction}
\end{figure}
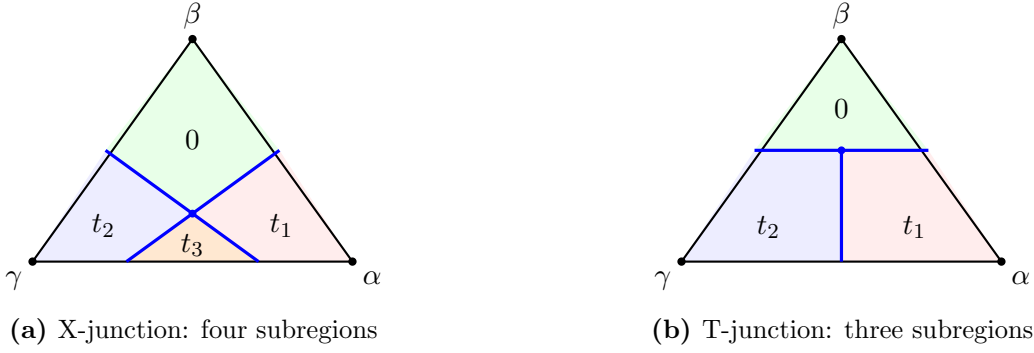

Let $\mathbf t_T=(t_1,\ldots,t_m)^{\transp}$.
For each vertex $i\in\{1,2,3\}$, let $T_{\rho_i}$ be the subregion containing it. We define the vertex-region incidence matrix $E_T\in\mathbb R^{3\times m}$ by
\[
(E_T)_{ir}=
\begin{cases}
1, & \rho_i=r,\\
0, & \text{otherwise},
\end{cases}
\qquad r=1,\ldots,m .
\]
Thus, with the convention $t_0=0$, $E_T\mathbf t_T=(t_{\rho_1},t_{\rho_2},t_{\rho_3})^{\transp}$ is the vector whose $i$-th entry is the pressure offset of the subregion containing vertex $i$.
The corrected nodal vector used in the bulk part of the local energy is therefore $\mathbf u_T-E_T\mathbf t_T$.

We next encode the pressure-jump penalty along the barriers.  
We split the barrier pieces inside $T$ at their junction points and call the resulting segments arms.
For each arm $a$, let $T_{\rho_a^-}$ and $T_{\rho_a^+}$ be the two adjacent subregions.
The pressure drop across this arm is the difference of the corresponding offsets, $t_{\rho_a^+}-t_{\rho_a^-}$.
Equivalently, let $\mathbf e_0=\mathbf{0}$ and let $\mathbf e_r$ be the $r$-th coordinate vector in $\mathbb R^m$ for $r=1,\ldots,m$, and define
\[
\mathbf d_a=\mathbf e_{\rho_a^+}-\mathbf e_{\rho_a^-}.
\]
Then the pressure drop across arm $a$ is $\mathbf d_a^{\transp}\mathbf t_T$, and its barrier contribution is
\[
\frac{\ell_a}{2R_a}(\mathbf d_a^{\transp}\mathbf t_T)^2,
\]
where $\ell_a$ is the length of the arm and $R_a$ is its resistance.  Summing over all arms gives
\[
\frac12\mathbf t_T^{\transp}\mathcal P_T\mathbf t_T,
\qquad
\mathcal P_T=\sum_a\frac{\ell_a}{R_a}\mathbf d_a\mathbf d_a^{\transp}.
\]
Thus $\mathcal P_T$ is the reduced weighted graph Laplacian of the local subregion-adjacency graph.
The subregions are the graph nodes, and the barrier arms are the weighted graph edges.  The orientation used in defining $\mathbf d_a$ is arbitrary because the contribution is quadratic.

With these definitions, the local energy on a cut cell $T\in\ThG$ is
\begin{equation}\label{eq:ET-junction}
\mathcal E_T(\mathbf u_T,\mathbf t_T)
=
\frac12(\mathbf u_T-E_T\mathbf t_T)^{\transp}K_T(\mathbf u_T-E_T\mathbf t_T)
+
\frac12\mathbf t_T^{\transp}\mathcal P_T\mathbf t_T .
\end{equation}
The first term is the bulk finite element energy evaluated after removing the local pressure offsets, and the second term is the sum of the pressure-jump penalties along the barrier arms.
Minimizing the quadratic energy $\mathcal{E}_{T}(\mathbf{u}_{T},\mathbf{t}_{T})$ with respect to the local auxiliary offsets $\mathbf t_T$ for fixed $\mathbf{u}_{T}$ gives
\begin{equation}\label{eq:tstar-junction}
\mathbf t_T^\ast
=
\bigl(E_T^{\transp}K_TE_T+\mathcal P_T\bigr)^{-1}E_T^{\transp}K_T\mathbf u_T .
\end{equation}
Since the local region graph is connected and the base region is grounded by $t_0=0$, the reduced graph Laplacian $\mathcal P_T$ is positive definite on the offset variables; hence the small matrix $E_T^{\transp}K_TE_T+\mathcal P_T$ is invertible.  Substituting $\mathbf t_T^\ast$ back into the energy gives the modified local stiffness
\begin{equation}\label{eq:cellupdate-junction}
\widetilde K_T
=
K_T
-
K_TE_T\bigl(E_T^{\transp}K_TE_T+\mathcal P_T\bigr)^{-1}E_T^{\transp}K_T .
\end{equation}
Positive semidefiniteness with kernel the multiples of $\mathbf{1}$ (Proposition~\ref{prop:spd}) holds exactly in the same way.
The formula \eqref{eq:cellupdate-junction} reduces to the rank-one update \eqref{eq:cellupdate} in the case of a single chord.  

Every barrier arm separating two adjacent subregions contributes its pressure-jump penalty to $\mathcal P_T$.  
A truly isolated free tip that does not create a separate subregion in the local partition gives no offset difference to penalize and is therefore ignored.

\begin{rem}[An alternative shortcut]
A simpler alternative is the independent-chord shortcut.  When multiple barriers cut a cell $T$, one may process them one at a time by repeatedly applying the single-chord rank-one update  \eqref{eq:cellupdate}.  Starting from $K_T^{(0)}=K_T$, the $j$-th chord gives
\[
K_T^{(j)}
=
K_T^{(j-1)}
-
\frac{
\bigl(K_T^{(j-1)}\mathbf e_{\gamma_j}\bigr)
\bigl(K_T^{(j-1)}\mathbf e_{\gamma_j}\bigr)^{\transp}
}{
\mathbf e_{\gamma_j}^{\transp}K_T^{(j-1)}\mathbf e_{\gamma_j}
+\ell_j/R_j
},
\]
where $\gamma_j$ is the lone vertex separated by the $j$-th chord, $\ell_j$ is the chord length, and $R_j$ is its resistance.  
This shortcut coincides with the modification for a single chord and gives nearly the same result in cells containing only full chords.  
Its limitation appears at T-junction cells: the terminating arm is an incomplete chord and is therefore ignored by the independent-chord procedure.  This omission weakens the local seal and produces a leakage effect.  
We therefore use the coupled-region condensation \eqref{eq:cellupdate-junction} for genuine junction cells, while the independent-chord shortcut is adequate for cells containing only full chords.
\end{rem}

All these modifications are performed element by element.  
Each element contributes a modified local stiffness matrix, and the global unknowns remain the standard $P^1$ nodal values.

\section{Recovery of broken pressures by local post-processing}
\label{sec:recovery}

The barrier discretization constructed in the previous section deliberately keeps the solution in the standard finite element space.  
Thus the computed pressure $p_h$ is single-valued, although the interface model \eqref{eq:reduced-model} itself allows pressure jumps across $\Gamma$.  
The jump information, however, has not been discarded. It was introduced locally through auxiliary cell variables and eliminated only by stationarity condition.  
The purpose of this section is to reverse this local elimination after the global solve and to reconstruct the corresponding cellwise broken pressure field.

We first consider a cut cell $T$ intersected by a single barrier segment $\Gamma_T=\Gamma\cap T$.  
Let $\gamma$ be the vertex separated from the other two vertices $\alpha,\beta$ by $\Gamma_T$, see Figure~\ref{fig:cutcell-a}, and write $\mathbf u_T=(u_1,u_2,u_3)^{\transp}$.  
Once the global system has been solved and $\mathbf u_T$ is known, the eliminated variable $t_T$ is recovered by the local stationarity condition \eqref{eq:stationarity}, namely
\[
t_T^{\ast}
=
\frac{\mathbf e_\gamma^{\transp}K_T\mathbf u_T}
{\mathbf e_\gamma^{\transp}K_T\mathbf e_\gamma+\ell_T/R}.
\]
This is a purely local back-substitution and requires no additional global solve.

Let $T_{\alpha\beta}$ denote the subregion containing the vertices
$\alpha,\beta$, and let $T_\gamma$ denote the subregion containing the lone
vertex $\gamma$.  The recovered pressure on $T$ is then defined by
\begin{equation}
\widehat p_h
=
\begin{cases}
p_h-t_T^\ast\phi_\gamma, & \text{in } T_{\alpha\beta},\\[1mm]
p_h+t_T^\ast(1-\phi_\gamma), & \text{in } T_\gamma,
\end{cases}
\end{equation}
where $\phi_\gamma$ is the local basis function associated with $\gamma$.
The reconstruction $\widehat p_h$ coincides with the continuous pressure $p_h$ at the three vertices, has the recovered jump $t_{T}^\ast$ across $\Gamma_{T}$, and has a common bulk gradient on both sides of the barrier.

This reconstruction gives a direct interpretation of the local energy \eqref{eq:ET}.
Indeed, we have
\[
\nabla \widehat p_h=
\nabla(p_h-t_T^\ast\phi_\gamma)
\qquad\text{in }T\setminus\Gamma .
\]
Moreover, $[\widehat p_h]_{\Gamma_T}=t_T^\ast$, with the orientation from $T_\gamma$ to $T_{\alpha\beta}$. Therefore
\begin{equation}\label{eq:broken-energy}
\widetilde{\mathcal{E}}_{T}(\mathbf{u}_T)=\mathcal E_T(\mathbf u_T,t_T^\ast)
=
\frac12\int_{T\setminus\Gamma}
\mathbf K_m\nabla \widehat p_h\cdot\nabla \widehat p_h
\,d\mathbf{x}
+
\frac{1}{2R}\int_{\Gamma_T}
[\widehat p_h]^2\,ds .
\end{equation}
Thus the post-processed field is precisely the broken finite element field whose
bulk gradient and barrier jump produce the discrete quadratic energy of \eqref{eq:continuous-energy} on a cut cell.

The same construction applies to the general junction cells considered in
Section~\ref{sec:junctions}.  Suppose that the local barrier network divides $T$
into subregions $T_0,T_1,\ldots,T_m$, see Figure \ref{fig:junction}.  
The eliminated regional offsets are recovered from
the local system
\[
\mathbf t_T^{\ast}
=
\bigl(E_T^{\transp}K_TE_T+\mathcal P_T\bigr)^{-1}
E_T^{\transp}K_T\mathbf u_T.
\]
Let $t_0=0$, and let
$t_r$ denote the recovered offset of $T_r$ for $r=1,\ldots,m$.  
For $x\in T$, let $\rho(x)\in\{0,\ldots,m\}$ be the index of the subregion containing $x$.  Also
let $\rho_i\in\{0,\ldots,m\}$ be the index of the subregion containing the
vertex $i$, $i=1,2,3$.  The recovered pressure is then defined by
\[
\widehat p_h(x)
=
\sum_{i=1}^3 \bigl(u_i-t_{\rho_i}\bigr)\phi_i(x)
+
t_{\rho(x)},\qquad x\in T_{\rho(x)} .
\]
Across a local barrier segment separating
$T_r$ and $T_s$, the recovered jump is $t_s-t_r$, with the
corresponding orientation.  

On uncut cells, we set $\widehat p_h=p_h$.

{
\begin{rem}[Equivalent reconstructed formulation]
The recovery introduces no additional approximation beyond the discrete model.
Let $\widehat{v}_h=\mathcal R_h(v_h,\mathbf t)$ denote the broken field obtained by applying the reconstruction above to a continuous nodal function $v_h$ and a collection of local jump variables $\mathbf t=(\mathbf{t}_T)_{T\in\mathcal T_h^\Gamma}$.
By the local energy identities established above, the uncondensed discrete energy functional can be written as
\[
J_h(v_h,\mathbf t)
=
\frac12
a_h\bigl(\mathcal R_h(v_h,\mathbf t),
         \mathcal R_h(v_h,\mathbf t)\bigr)
-\ell(v_h),
\]
where
\[
a_h(\widehat{w},\widehat{v})
:=
\sum_{T\in\mathcal T_h}
\int_{T\setminus\Gamma}
\mathbf K_m\nabla \widehat{w}\cdot\nabla \widehat{v}\,d\mathbf{x}
+
\frac1R
\sum_{T\in\mathcal T_h^\Gamma}
\int_{\Gamma_T}[\widehat{w}][\widehat{v}]\,ds .
\]
Since $(p_h,\mathbf t^\ast)$ is the joint minimizer of this functional, the recovered solution $\widehat p_h=\mathcal R_h(p_h,\mathbf t^\ast)$ satisfies
\[
a_h(\widehat p_h,\widehat v_h)
=
\ell(v_h)
\qquad
\forall\,
\widehat v_h=\mathcal R_h(v_h,\mathbf s),
\quad
v_h\in V_{h,0},
\quad
\mathbf s\ \text{arbitrary}.
\]
\end{rem}}

\section{Numerical results}\label{sec:num}

The numerical experiments in this section evaluate the performance of the proposed method on unfitted triangular meshes. 
Two quantities are obtained: the standard finite element solution $p_h$ in Section~\ref{sec:method} and its recovered broken solution $\widehat p_h$ in Section~\ref{sec:recovery}, and they are compared with the reference solutions of published benchmarks \cite{flemisch2018benchmarks, berre2021verification}.
The Darcy velocity $\mathbf{u}_h$ is computed from $\widehat{p}_{h}$ in Example 1 for convergence studies.
In all other examples, the source term is $f=0$ and the flow is driven by the boundary conditions.

\subsection{Example 1: convergence study}\label{sec:ex1}

We first study the convergence of the proposed method to manufactured solutions \cite{liu2026high}.
Two permeability fields are considered on the domain $\Om=[0,1]^2$: one with homogeneous isotropic permeability and one with heterogeneous anisotropic permeability.
Each setting is tested in two geometric configurations.
In all cases, the barrier resistance is fixed at $R=a/k_b=1$, and Dirichlet conditions are prescribed from the exact pressure. 
The meshes are obtained by dividing $\Omega$ into $N\times N$ squares
and cutting each square along the diagonal, with $N\in\{11,21,41,81,161,321\}$.
Representative mesh and barrier configurations are shown in Figure~\ref{fig:ex1mesh}.

\begin{figure}[htbp!]
\centering
\begin{subfigure}[b]{0.44\textwidth}
\includegraphics[width=\textwidth]{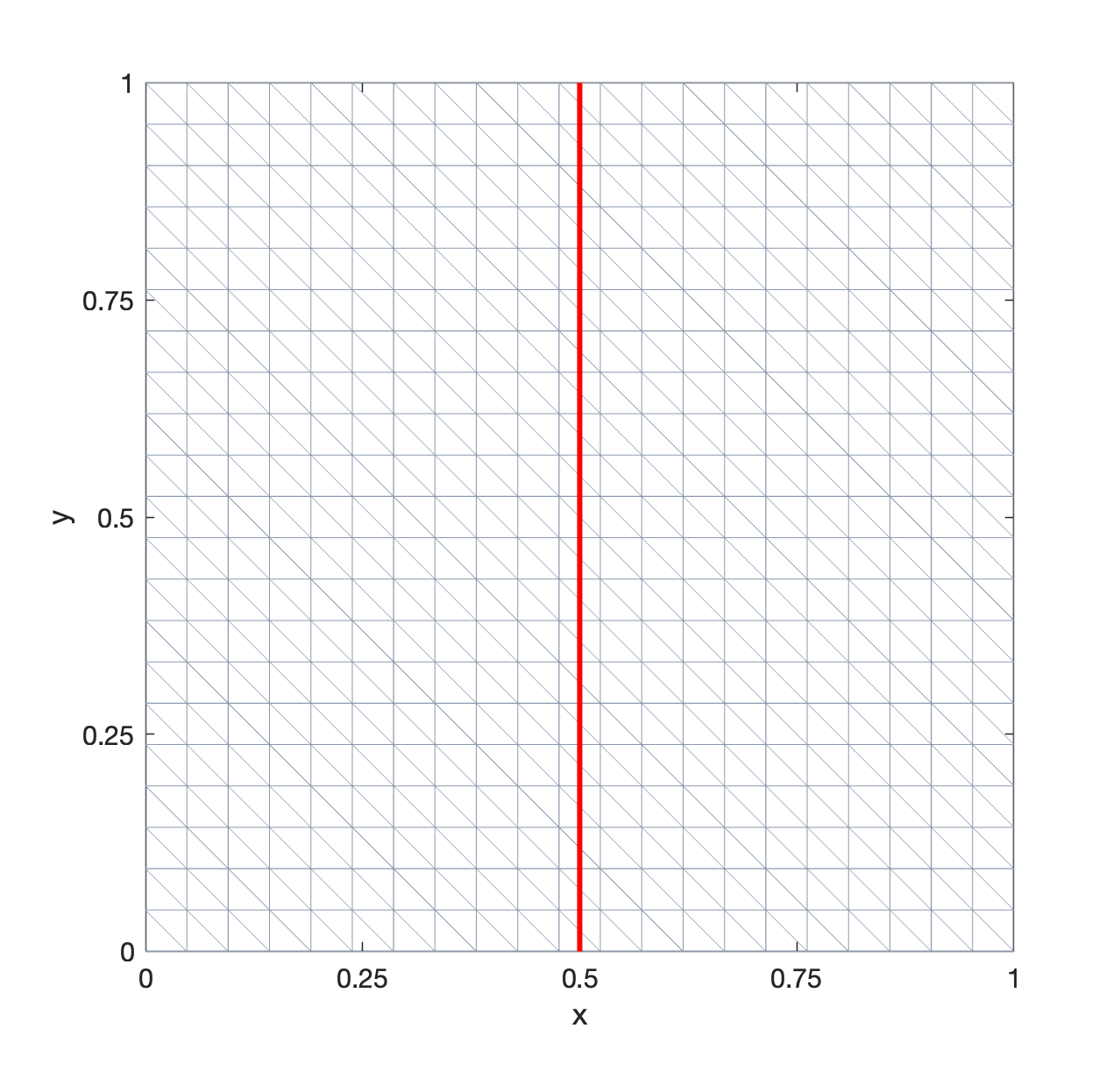}
  \caption{Vertical barrier used in cases (a) and (b).}
\end{subfigure}
\hfill
\begin{subfigure}[b]{0.44\textwidth}
\includegraphics[width=\textwidth]{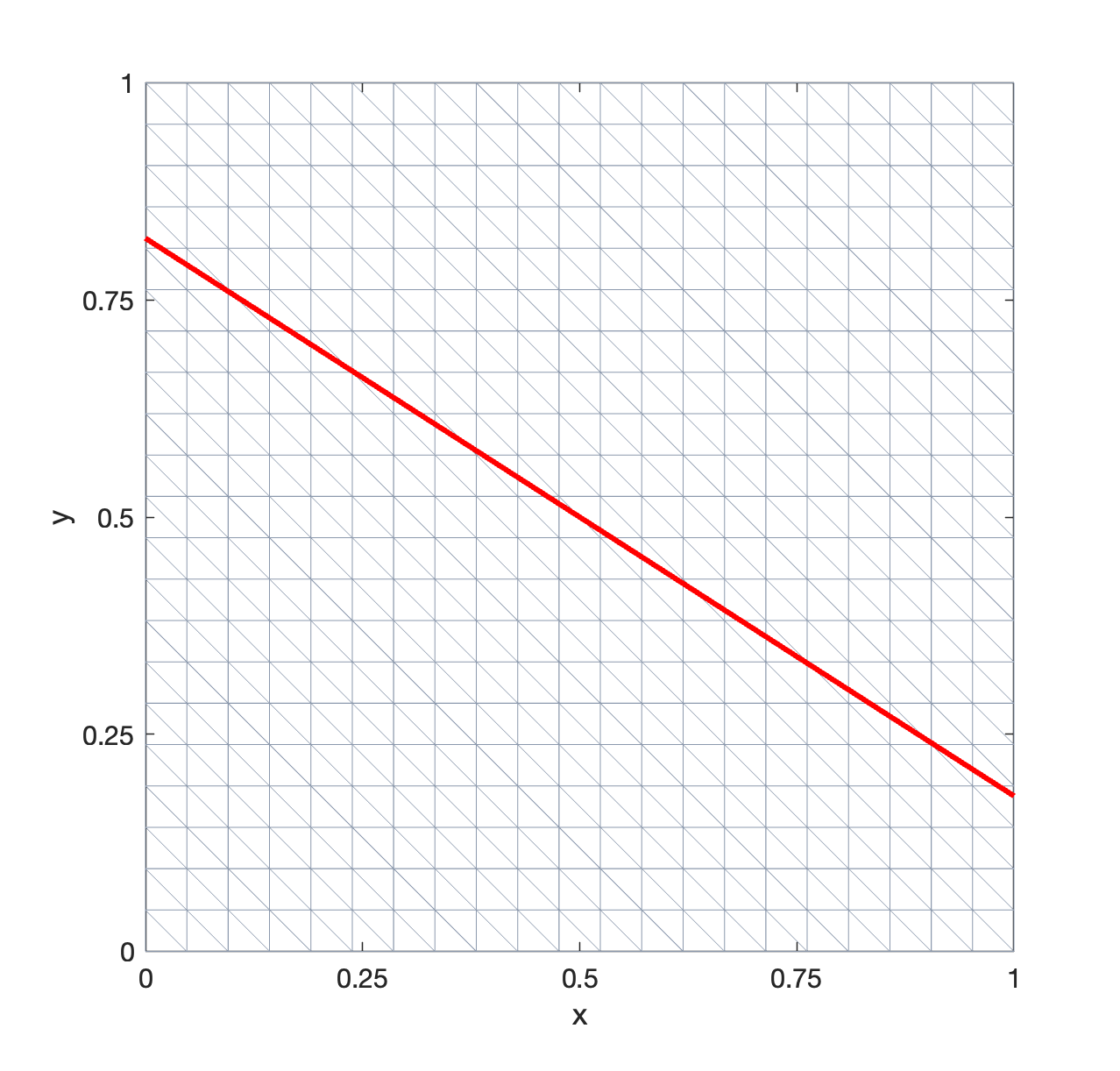}
  \caption{Slanted barrier used in cases (c) and (d).}
\end{subfigure}
\caption{Example~1: structured triangular mesh with $N=21$ and the barrier shown in red. In the axis-aligned configuration, the barrier is the vertical line $x=\tfrac12$ and crosses the interiors of the cells in the central column. 
The rotated configuration is obtained by rotating the barrier and the associated manufactured problem by one radian about $(\tfrac12,\tfrac12)$, while leaving the computational domain and mesh unchanged.}
\label{fig:ex1mesh}
\end{figure}

\emph{Case (a): homogeneous isotropic permeability.} 
With $\Km=\mathbf I$ and $\Gamma=\{(x,y)\in\Omega:x=\tfrac12\}$, we prescribe the pressure as 
\[p=\sin x\,\sin y+\chi(x>\frac12)\cos\tfrac12\,\sin y,\]
where $\chi$ is the indicator function.
The source term is 
\[f=2\sin x\,\sin y+\chi(x>\frac12)\cos\tfrac12\,\sin y.\]

\emph{Case (b): heterogeneous anisotropic permeability.}
We again take $\Gamma=\{(x,y)\in\Omega:x=\tfrac12\}$, but now use a full permeability tensor,
\begin{equation*}\label{eq:ex1bK}
\Km(x,y)=k(x,y)\,\mathbf A,
\qquad
k(x,y)=2+\rho\sin(\pi x)\sin(\pi y),
\qquad
\mathbf A=
\begin{pmatrix}
1-\tfrac\eta4 & \tfrac{\sqrt3}4\eta\\[2pt]
\tfrac{\sqrt3}4\eta & 1-\tfrac{3\eta}4
\end{pmatrix},
\end{equation*}
with $\rho=1$ and $\eta=0.9$. 
The scalar factor $k$ makes the medium
heterogeneous.
The matrix $\mathbf A$ has eigenvalues $1$ and
$1-\eta$, giving an anisotropy ratio of $10$, and its eigen-directions
are rotated by $30^\circ$ relative to the mesh axes. 
Writing $\alpha=1-\eta/4$, $\beta=\tfrac{\sqrt3}{4}\eta$, the exact pressure is
\[p=\sin x\,\sin y+\chi(x>\frac12)\left(J(y)-\frac{\beta}{\alpha}
\bigl(x-\tfrac12\bigr)J'(y)\right),\]
where $J(y)=k\bigl(\tfrac12,y\bigr)
\left[
\alpha\cos\tfrac12\,\sin y
+\beta\sin\tfrac12\,\cos y
\right].$
The source is evaluated based on $f=-\nabla\cdot(\Km\nabla p)$, which is omitted here to save space.

\emph{Cases (c) and (d): rotated counterparts of cases (a) and (b).}
The vertical barrier in cases (a) and (b) is parallel to the mesh columns, which could in principle produce artificially favorable behavior. To test whether the method depends on this alignment, we construct cases (c) and (d) by rotating the complete data of cases (a) and (b) by $\theta=1$ around the center $\mathbf{c}=(\tfrac12,\tfrac12)$, while leaving the computational domain and mesh unchanged.
In both cases, the barrier has no
special relation to the mesh directions.

In Table~\ref{tab:ex1}, we report the $L^2$ errors and observed convergence orders of the continuous pressure $p_h$, the recovered broken pressure $\widehat p_h$, and the Darcy velocity $\mathbf{u}_h=-\mathbf{K}_{m}\nabla\widehat{p}_{h}$.
The four cases exhibit the same qualitative behavior, indicating that the
observed convergence is not an artifact of barrier-mesh alignment. In every case, the single-valued pressure satisfies $\|p-p_h\|_{L^2}\sim h^{1/2}.$
Because $p_h$ is continuous, it must smear the $O(1)$ pressure jump across the one-cell-thick layer of cut elements. Its $L^2$ error is therefore limited to the half-order rate associated with approximating a discontinuous solution by continuous functions.
The local recovery removes this smearing by restoring the pressure drop at the barrier interface. 
Consequently, $\widehat p_h$ converges substantially faster than $p_h$. 
Over the tested mesh range, the observed rates are approximately $1.3$--$1.5$ in the axis-aligned cases and approximately $1.2$--$1.3$ in the rotated cases. 
The Darcy velocity remains approximately half-order accurate, $\|\mathbf u-\mathbf u_h\|_{L^2}\sim h^{1/2}$. The recovered pressure has only one bulk gradient on each cut element and therefore cannot fully represent the variation of the pressure within that element. The velocity error consequently remains limited by this unresolved subcell variation near the barrier.

\begin{table}[htbp!]
\centering
\small
\setlength{\tabcolsep}{5pt}
\renewcommand{\arraystretch}{1.05}
\caption{Example~1, $L^2$ errors and observed convergence orders of $p_h$, $\widehat p_h$,
and $\mathbf u_h$. Cases~(a) and~(b) use the axis-aligned barrier $\Gamma=\{x=\tfrac12\}$, whereas cases~(c) and~(d) are their counterparts rotated by $\theta=1$ about the center $(\tfrac12,\tfrac12)$.}
\label{tab:ex1}
\begin{tabular}{@{}r rr rr rr@{}}
\toprule
\multicolumn{7}{c}{
Case (a). Homogeneous isotropic matrix permeability.}\\
\toprule
$N$
& $\|p-p_h\|_{L^2}$ & order
& $\|p-\widehat p_h\|_{L^2}$ & order
& $\|\mathbf u-\mathbf u_h\|_{L^2}$ & order \\
\midrule
11  & 4.00e-2 & --   & 2.61e-3 & --   & 1.30e-1 & --\\
21  & 2.89e-2 & 0.50 & 9.25e-4 & 1.60 & 8.84e-2 & 0.60\\
41  & 2.07e-2 & 0.50 & 3.31e-4 & 1.53 & 6.10e-2 & 0.56\\
81  & 1.47e-2 & 0.50 & 1.22e-4 & 1.47 & 4.25e-2 & 0.53\\
161 & 1.04e-2 & 0.50 & 4.67e-5 & 1.40 & 2.98e-2 & 0.52\\
321 & 7.38e-3 & 0.50 & 1.89e-5 & 1.31 & 2.10e-2 & 0.51\\
\toprule
\multicolumn{7}{c}{
Case (b). Heterogeneous anisotropic matrix permeability.}\\
\toprule
$N$
& $\|p-p_h\|_{L^2}$ & order
& $\|p-\widehat p_h\|_{L^2}$ & order
& $\|\mathbf u-\mathbf u_h\|_{L^2}$ & order \\
\midrule
11  & 1.15e-1 & --   & 1.26e-2 & --   & 3.13e-1 & --\\
21  & 8.23e-2 & 0.51 & 4.19e-3 & 1.70 & 1.76e-1 & 0.89\\
41  & 5.88e-2 & 0.50 & 1.38e-3 & 1.67 & 9.66e-2 & 0.89\\
81  & 4.18e-2 & 0.50 & 4.66e-4 & 1.59 & 5.42e-2 & 0.85\\
161 & 2.97e-2 & 0.50 & 1.66e-4 & 1.50 & 3.17e-2 & 0.78\\
321 & 2.10e-2 & 0.50 & 6.25e-5 & 1.41 & 1.96e-2 & 0.70\\
\toprule
\multicolumn{7}{c}{
Case (c). Homogeneous isotropic matrix permeability.}\\
\toprule
$N$
& $\|p-p_h\|_{L^2}$ & order
& $\|p-\widehat p_h\|_{L^2}$ & order
& $\|\mathbf u-\mathbf u_h\|_{L^2}$ & order \\
\midrule
11  & 5.26e-2 & --   & 4.80e-3 & --   & 1.45e-1 & --\\
21  & 3.73e-2 & 0.53 & 2.29e-3 & 1.14 & 1.14e-1 & 0.37\\
41  & 2.67e-2 & 0.50 & 8.70e-4 & 1.45 & 8.18e-2 & 0.50\\
81  & 1.88e-2 & 0.51 & 3.23e-4 & 1.46 & 5.73e-2 & 0.52\\
161 & 1.33e-2 & 0.50 & 1.36e-4 & 1.26 & 4.08e-2 & 0.49\\
321 & 9.45e-3 & 0.50 & 5.56e-5 & 1.29 & 2.89e-2 & 0.50\\
\toprule
\multicolumn{7}{c}{
Case (d). Heterogeneous anisotropic matrix permeability.}\\
\toprule
$N$
& $\|p-p_h\|_{L^2}$ & order
& $\|p-\widehat p_h\|_{L^2}$ & order
& $\|\mathbf u-\mathbf u_h\|_{L^2}$ & order \\
\midrule
11  & 1.36e-1 & --   & 1.26e-2 & --   & 2.96e-1 & --\\
21  & 9.92e-2 & 0.48 & 7.11e-3 & 0.88 & 2.51e-1 & 0.26\\
41  & 7.05e-2 & 0.51 & 2.70e-3 & 1.45 & 1.76e-1 & 0.52\\
81  & 4.99e-2 & 0.51 & 1.09e-3 & 1.34 & 1.25e-1 & 0.50\\
161 & 3.54e-2 & 0.50 & 4.68e-4 & 1.23 & 8.95e-2 & 0.49\\
321 & 2.50e-2 & 0.50 & 2.09e-4 & 1.17 & 6.33e-2 & 0.50\\
\bottomrule
\end{tabular}
\end{table}

\subsection{Example 2: a single strong barrier}\label{sec:ex2}
We now turn to a single, nearly impermeable barrier that terminates at a free tip \cite{angot2009asymptotic,xu2024box}.
On $\Om=[0,1]^2$ with $\Km=\mathbf I$, the left and right boundaries are Dirichlet with $p=0$ and $p=1$, respectively, while the top and bottom are impermeable. The barrier is nearly sealing with $R=a/k_b=10^5$. Two
configurations are solved on the same unstructured mesh of size $h\approx0.06$: a vertical barrier
from $(\tfrac12,\tfrac12)$ to $(\tfrac12,1)$ and a slanted barrier from $(\tfrac14,\tfrac34)$ to
$(\tfrac34,\tfrac14)$.

Figures~\ref{fig:ex2-vertical} and~\ref{fig:ex2-slanted} show the computational results together with the background mesh. 
The flow bends around the free tip, where both pressure fields remain continuous, while the nearly sealing barrier
supports a large pressure drop. 
The sampling lines are visualized in the pressure fields, and slices of continuous $p_{h}$ and 
recovered broken $\widehat{p}_{h}$ along the lines are compared with fitted fine-grid reference solutions of the Box-DFM \cite{xu2024box} with roughly $2.3\times10^{4}$ cells. 
From the slices we can observe that $p_{h}$ and $\widehat{p}_{h}$ agree on uncut cells and $\widehat{p}_{h}$ recovers the pressure jump smeared by the finite element solution $p_{h}$ at the barrier interface.
The pronounced smearing of $p_h$ in Figure \ref{fig:ex2-slanted}(c) is within expectation because the barrier happens to span three cells in the horizontal direction along the sampling line; see \ref{fig:ex2-slanted}(a).

\begin{figure}[htbp!]\centering
\begin{subfigure}[b]{0.32\textwidth}\includegraphics[width=\textwidth]{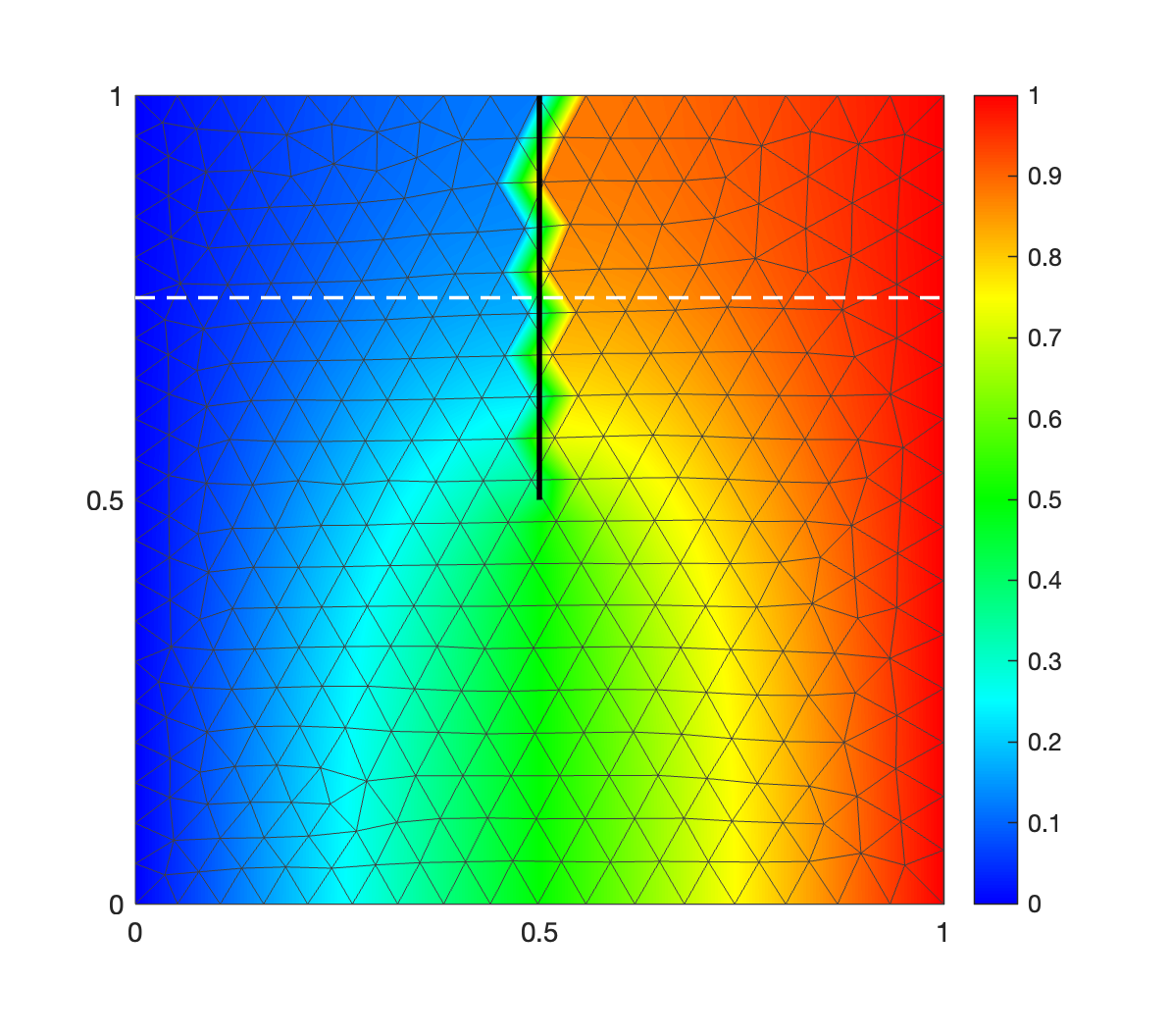}
  \caption{continuous $p_h$}\end{subfigure}\hfill
\begin{subfigure}[b]{0.32\textwidth}\includegraphics[width=\textwidth]{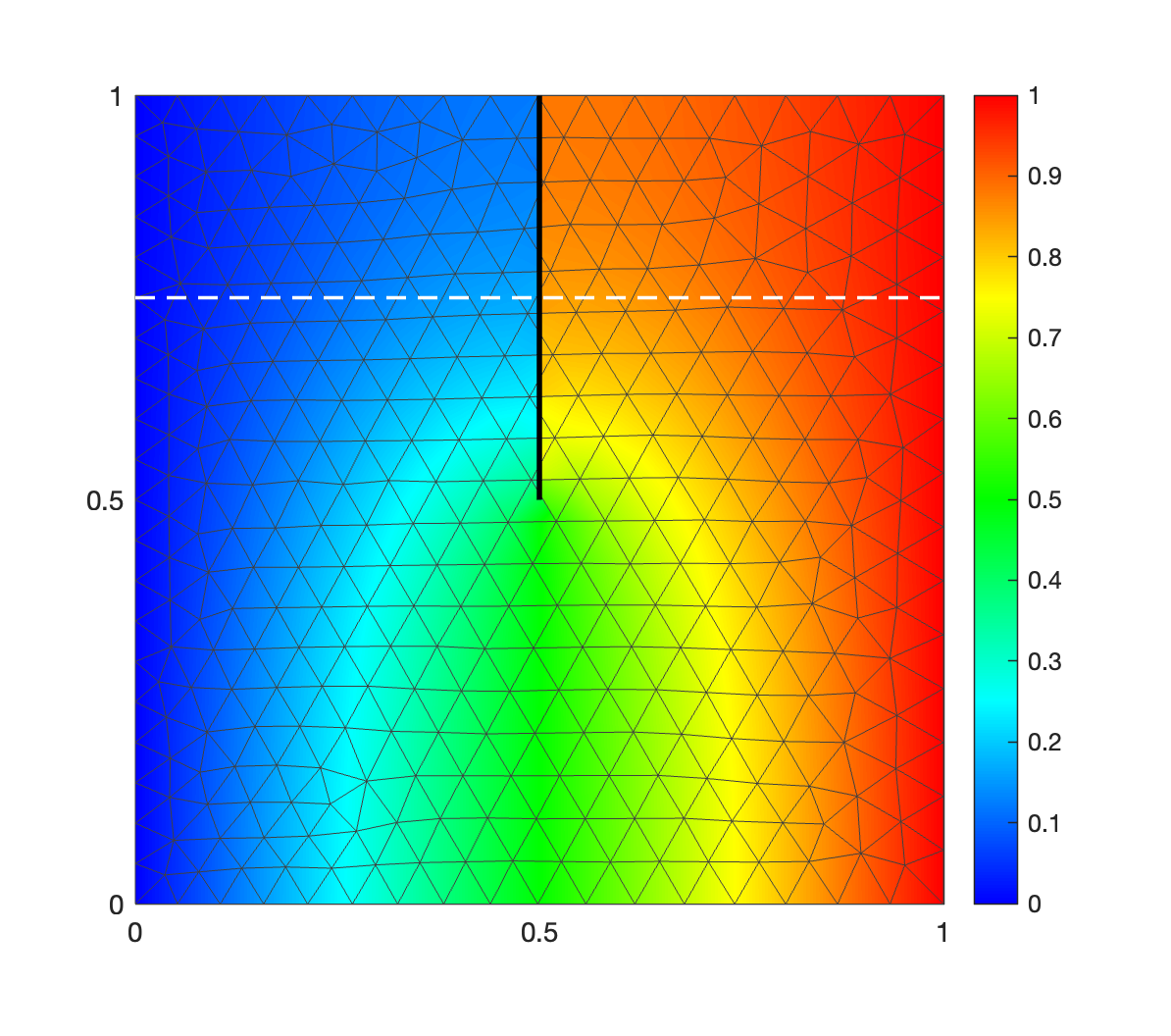}
  \caption{recovered broken $\widehat p_h$}\end{subfigure}\hfill
\begin{subfigure}[b]{0.34\textwidth}\includegraphics[width=\textwidth]{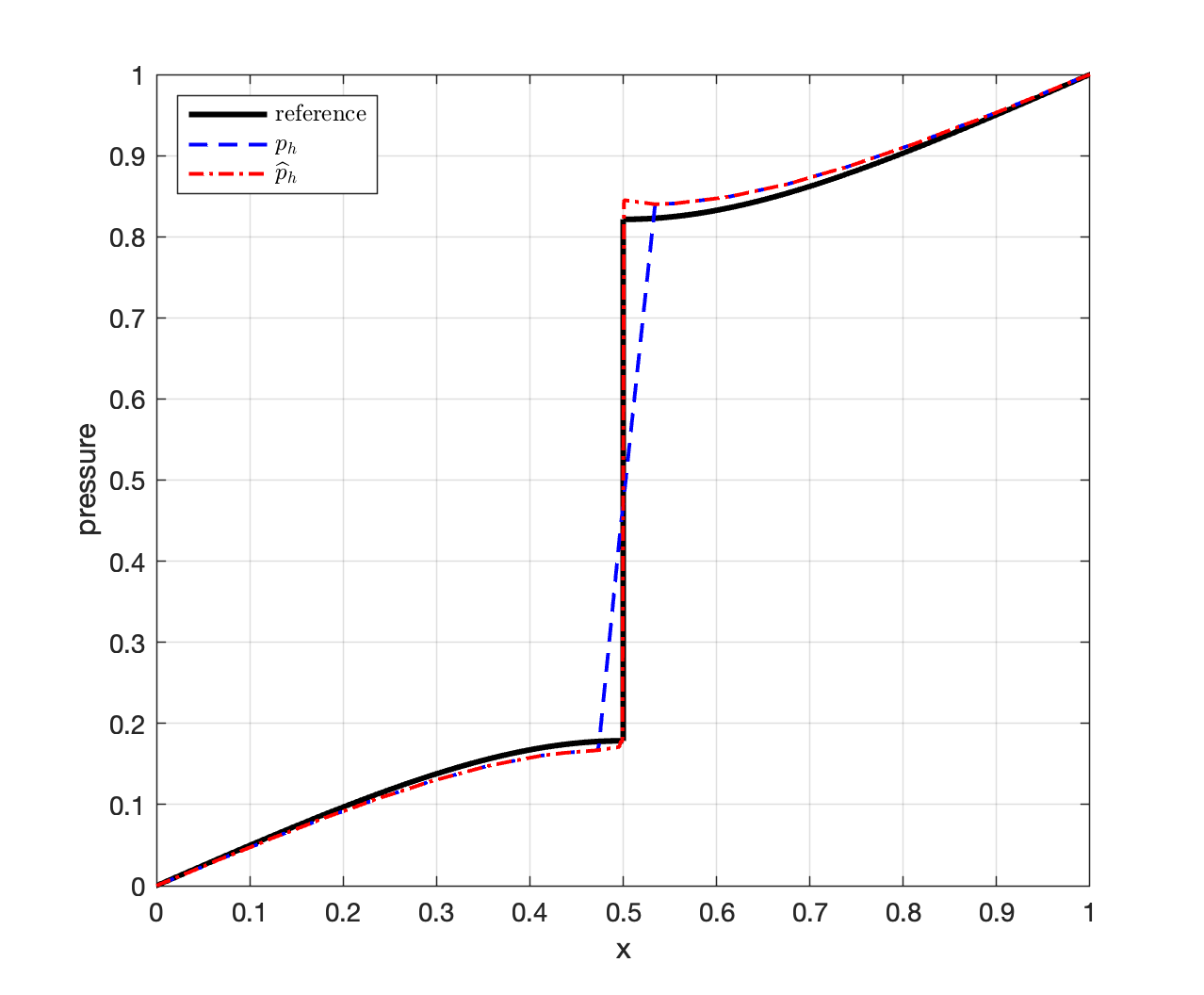}
  \caption{slice along $y=0.75$}\end{subfigure}
\caption{Example~2, vertical barrier $(\tfrac12,\tfrac12)$--$(\tfrac12,1)$, $R=10^5$. (a) continuous
pressure, (b) recovered broken pressure, and (c) the profile along $y=0.75$ against the Box-DFM \cite{xu2024box} reference solution obtained from a fitted fine grid containing $23,306$ cells.}
\label{fig:ex2-vertical}
\end{figure}

\begin{figure}[htbp!]\centering
\begin{subfigure}[b]{0.32\textwidth}\includegraphics[width=\textwidth]{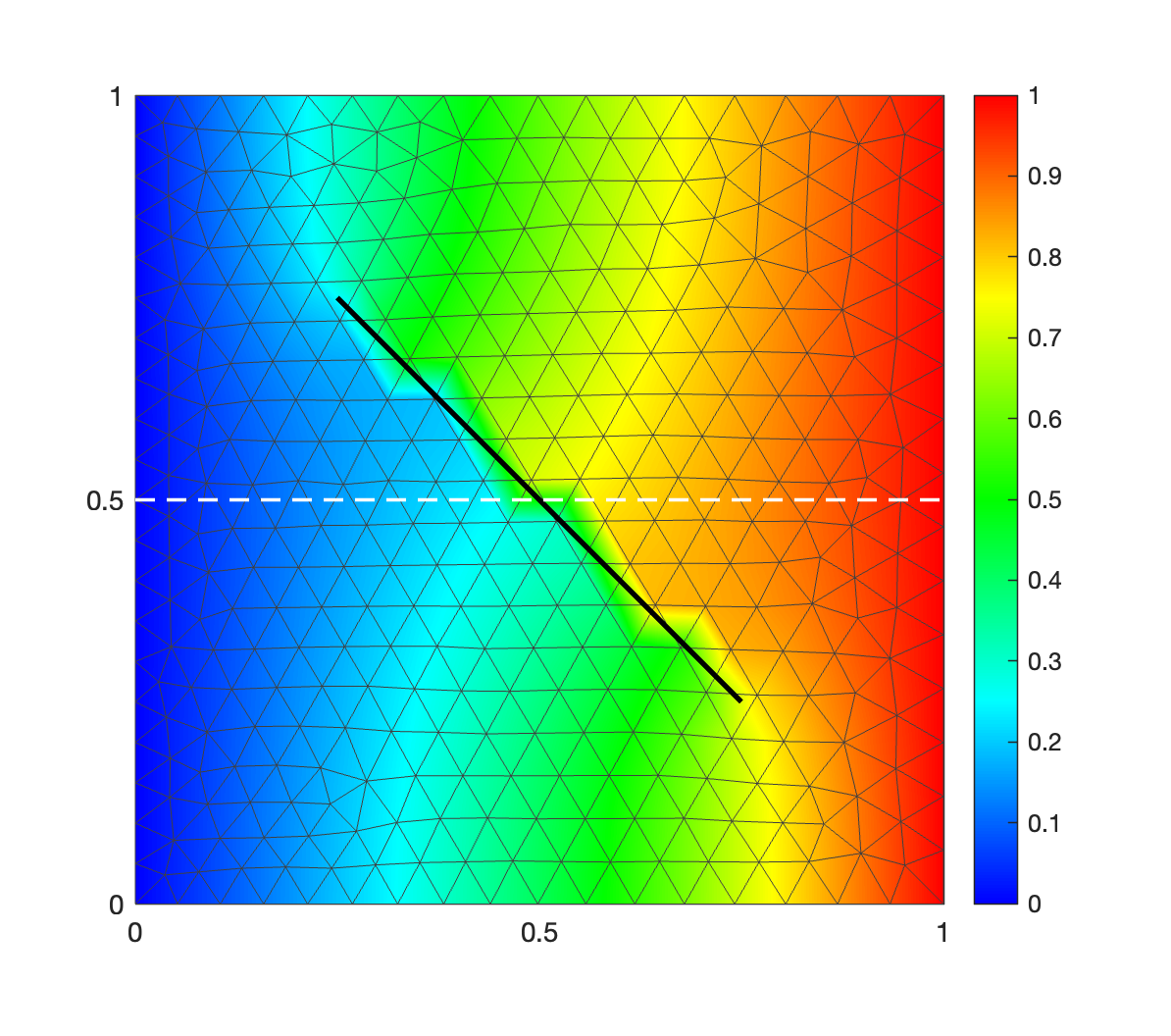}
  \caption{continuous $p_h$}\end{subfigure}\hfill
\begin{subfigure}[b]{0.32\textwidth}\includegraphics[width=\textwidth]{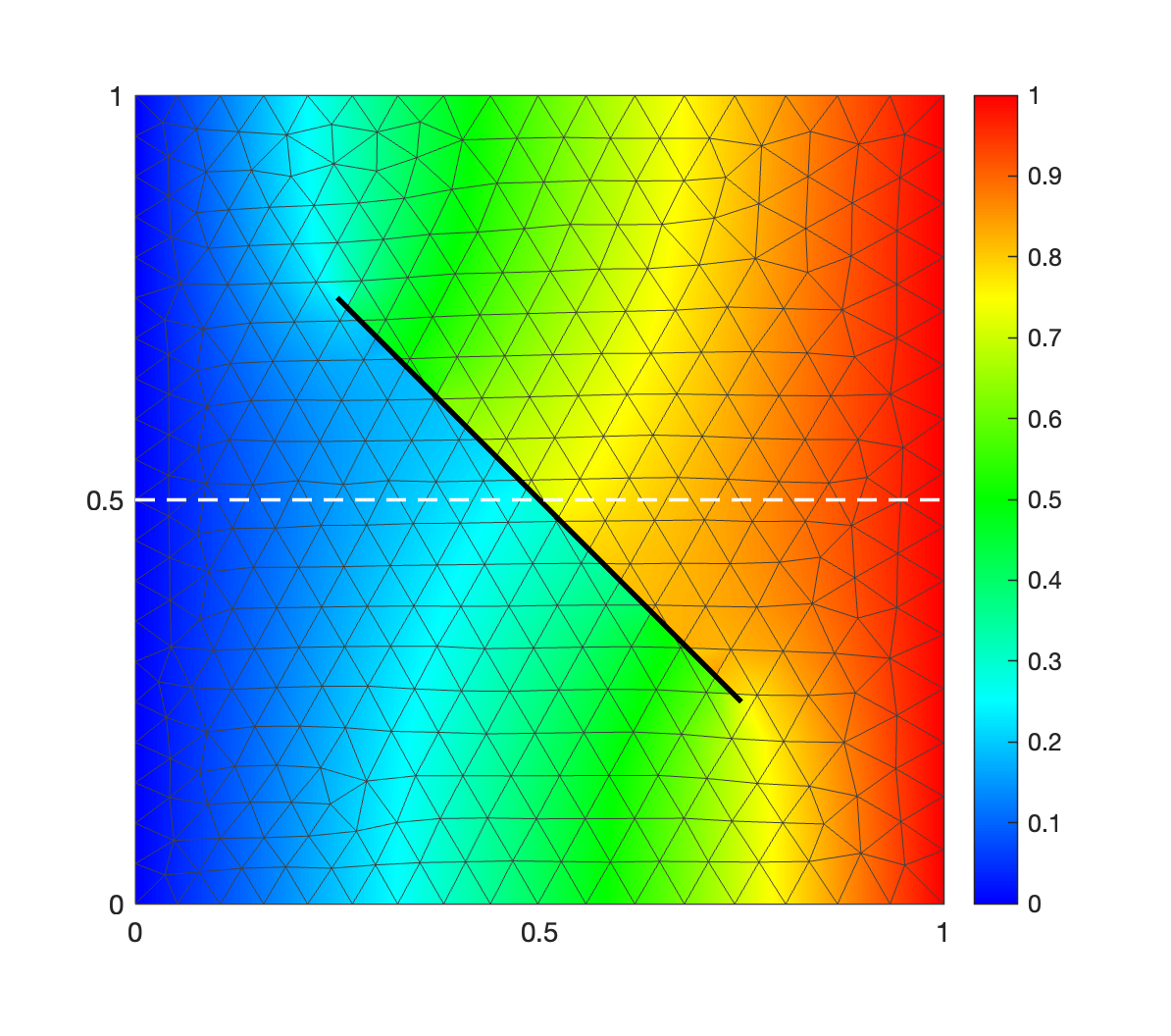}
  \caption{recovered broken $\widehat p_h$}\end{subfigure}\hfill
\begin{subfigure}[b]{0.34\textwidth}\includegraphics[width=\textwidth]{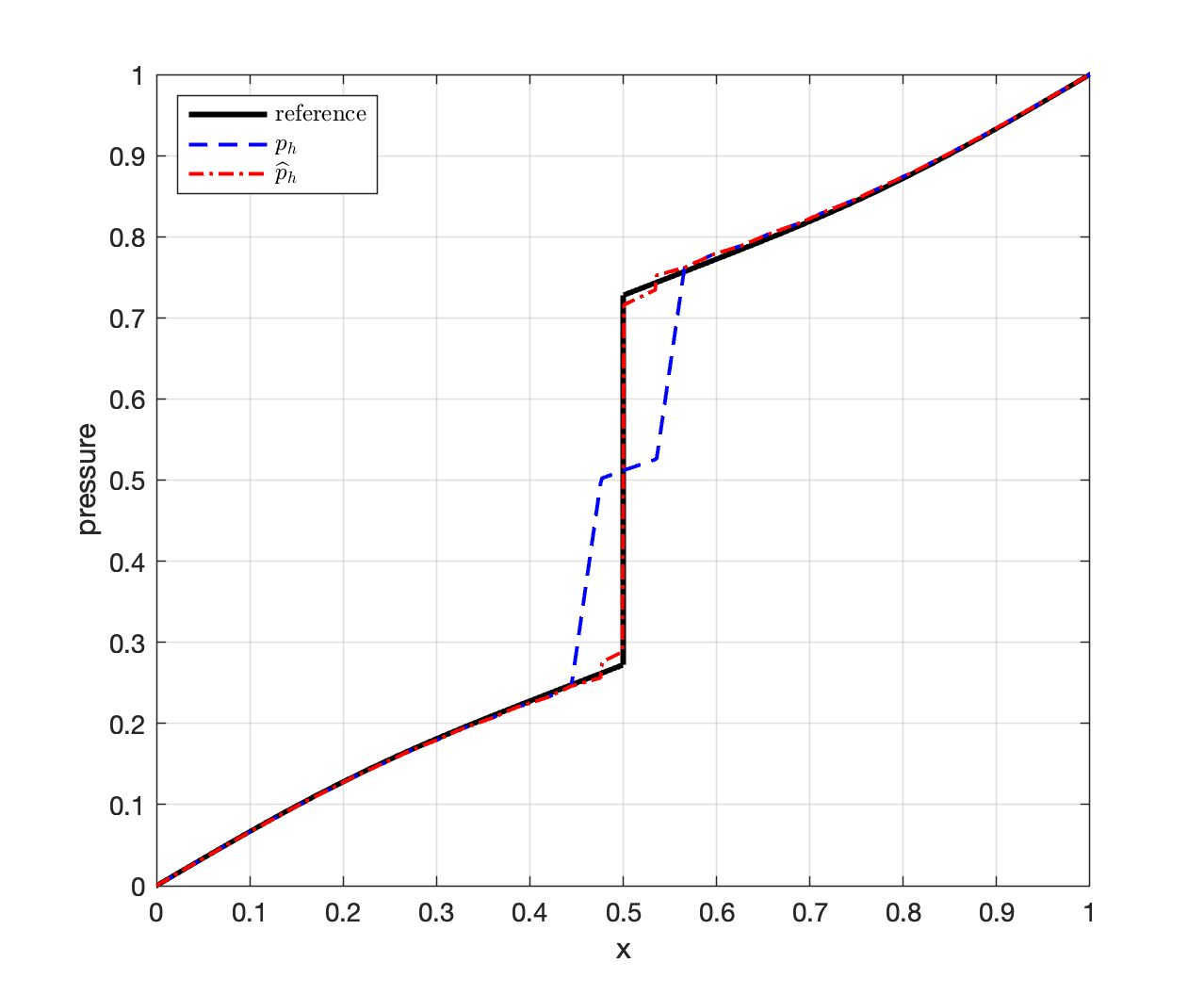}
  \caption{slice along $y=0.5$}\end{subfigure}
\caption{Example~2, slanted barrier $(\tfrac14,\tfrac34)$--$(\tfrac34,\tfrac14)$ on the same mesh, $R=10^{5}$. 
(a) continuous pressure, (b) recovered broken pressure, and (c) the profile along $y=0.5$ against the Box-DFM \cite{xu2024box} reference solution obtained from a fitted fine grid containing $23,455$ cells.}
\label{fig:ex2-slanted}
\end{figure}

\subsection{Example 3: a regular barrier network}\label{sec:ex3}

We next solve the regular network of six intersecting barriers of
\cite{flemisch2018benchmarks} on $\Om=[0,1]^2$ with $\Km=\mathbf I$. 
These barriers lie along the lines $x=0.5$ and $y=0.5$, each with length $1$; along the lines $x=0.75$ and $y=0.75$, each with length $0.5$; and along the lines $x=0.625$ and $y=0.625$, each with length $0.25$.
Each barrier has aperture and permeability $a=k_b=10^{-4}$, so $R=a/k_b=1$. 
The right boundary is Dirichlet with $p=1$, the left boundary carries a unit inflow ($\mathbf u\cdot\mathbf n=-1$), and the top and bottom are
impermeable. 
The barriers meet in X- and T-junctions, which are treated by the method of Section~\ref{sec:junctions}. 
An unstructured mesh of size $h\approx0.03$ is used.
The barrier coordinates are displaced by a small perturbation
$(7\times10^{-4},\,7.4\times10^{-3})$, so that no mesh entity degenerates onto a barrier.

Figure~\ref{fig:ex3} shows the pressure fields $p_h$ and $\widehat{p}_h$ on the background mesh, and the profile along the slice $(0,0.1)$--$(0.9,1)$.
A reference solution of the equidimensional mimetic finite-difference method  \cite{flemisch2018benchmarks} with about $1.2\times10^6$ cells joints the comparison.
We can see that the barriers separate regions of slowly varying pressure by finite jumps.
The recovered $\widehat p_h$ reproduces the position and the height of every step, while $p_h$ rounds each step over the cut
layer.

\begin{figure}[htbp!]\centering
\begin{subfigure}[b]{0.32\textwidth}\includegraphics[width=\textwidth]{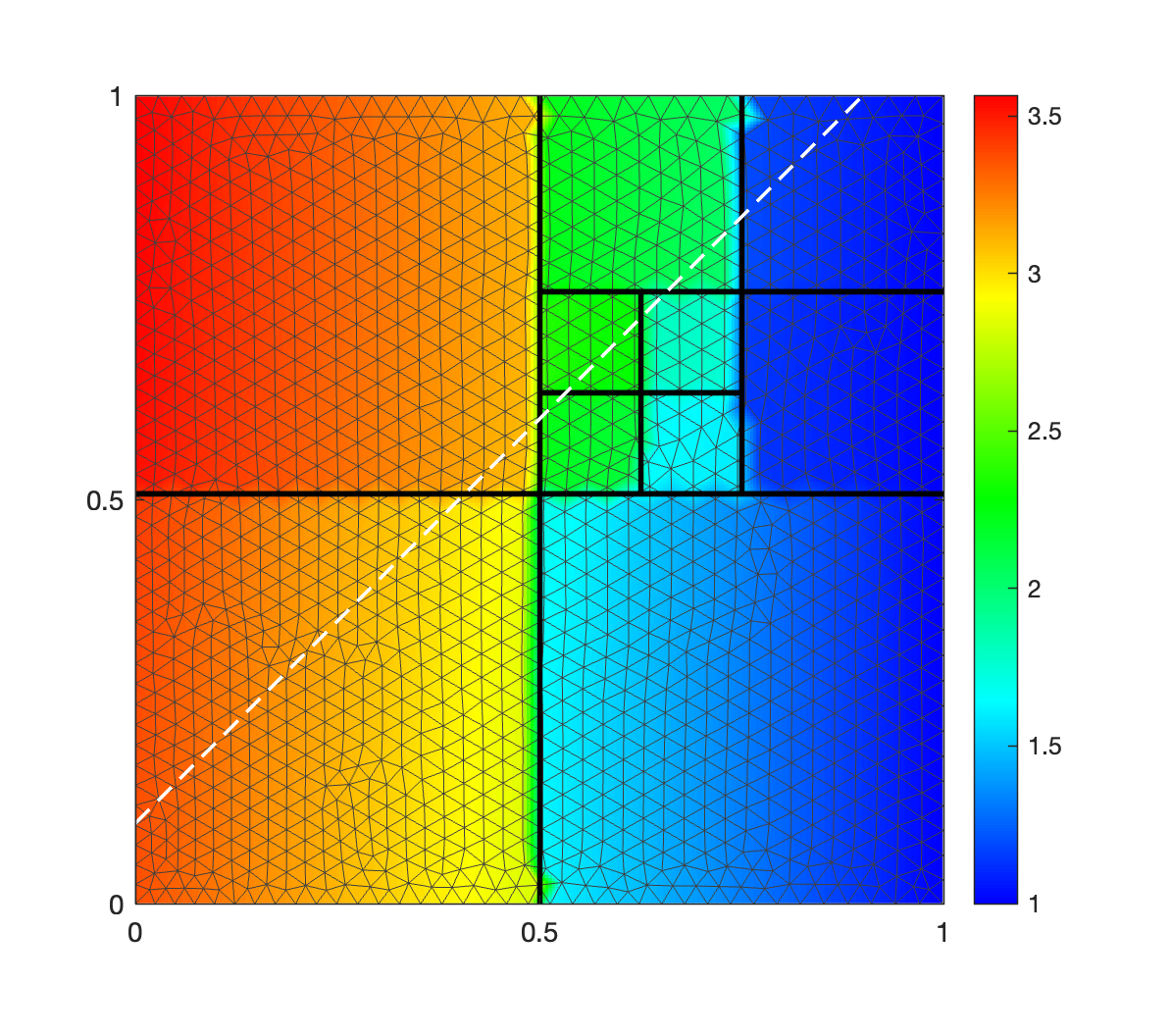}
  \caption{continuous $p_h$}\end{subfigure}\hfill
\begin{subfigure}[b]{0.32\textwidth}\includegraphics[width=\textwidth]{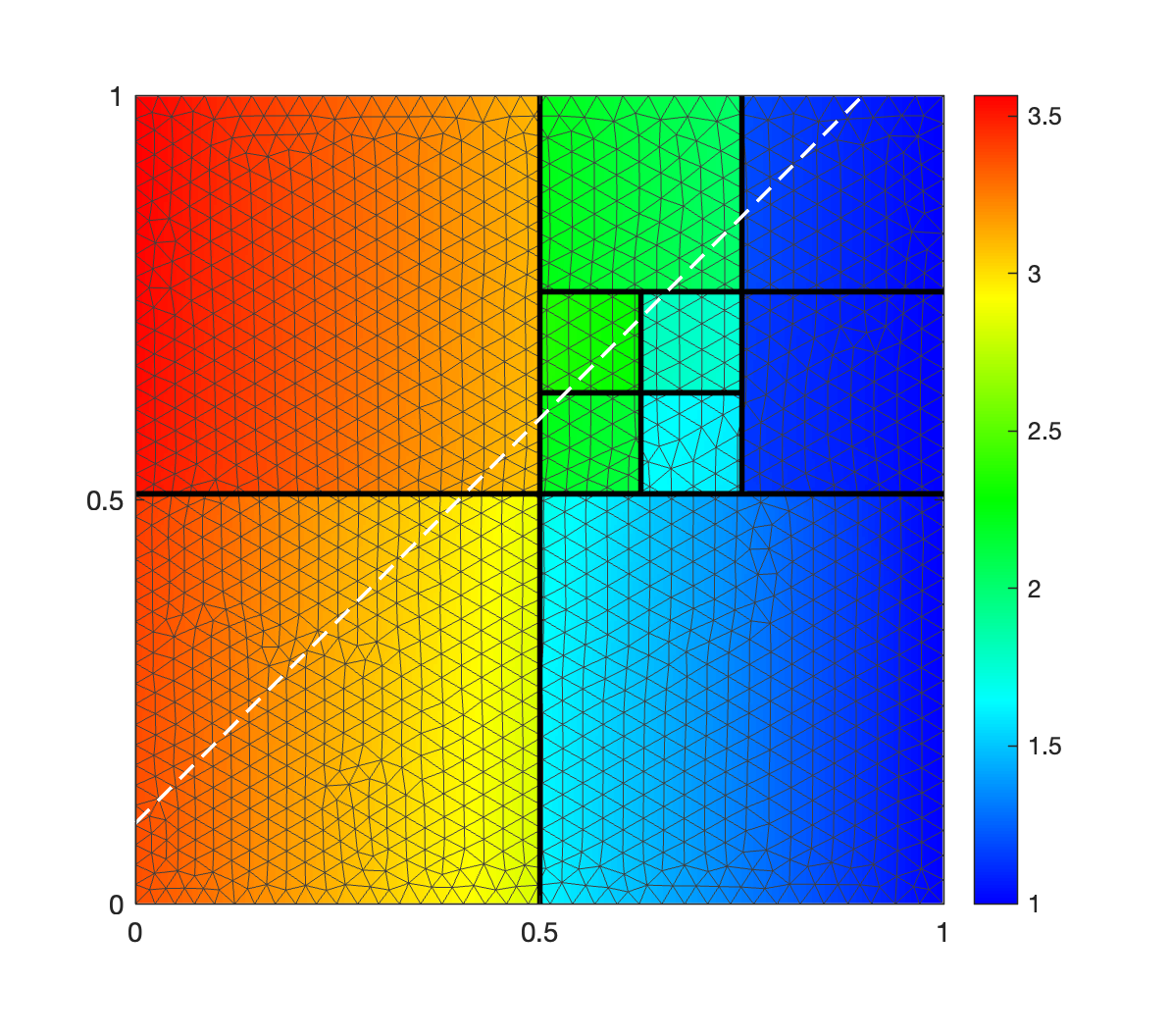}
  \caption{recovered broken $\widehat p_h$}\end{subfigure}\hfill
\begin{subfigure}[b]{0.34\textwidth}\includegraphics[width=\textwidth]{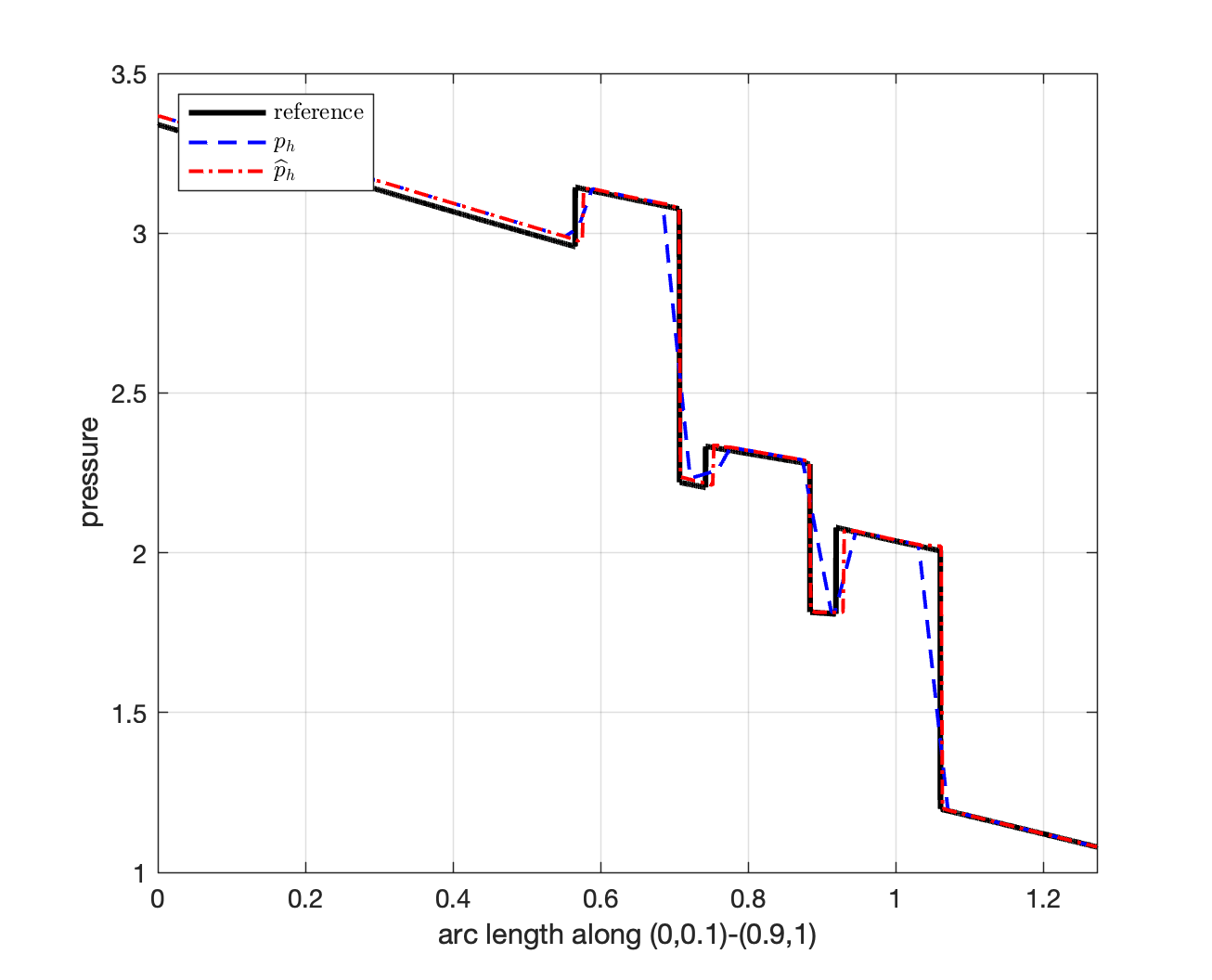}
  \caption{slice along $(0,0.1)$--$(0.9,1)$}\end{subfigure}
\caption{Example~3, regular network of six intersecting barriers, $R=1$: (a) continuous pressure,
(b) recovered broken pressure, and (c) the profile along $(0,0.1)$--$(0.9,1)$ against the mimetic
finite-difference reference of \cite{flemisch2018benchmarks}.}
\label{fig:ex3}
\end{figure}

\subsection{Example 4: a complex fracture-barrier network}\label{sec:ex4}

This example combines conductive fractures with blocking barriers, in the complex network of the third benchmark case of \cite{flemisch2018benchmarks}, where the coordinates of all fractures and barriers can be found in Appendix C. 
On $\Om=[0,1]^2$ with $\Km=\mathbf I$, eight fractures of permeability $k_f=10^4$ and two barriers of permeability $k_b=10^{-4}$ share the aperture $a=10^{-4}$, so the fractures conduct with $ak_f=1$ and the barriers block with $R=a/k_b=1$.
The conductive fractures are treated using the FEM-DFM proposed by the author and a collaborator in \cite{xu2020hybrid} for unfitted meshes. 
The method adds the tangential $1$D stiffness matrices to the local stiffness matrix of each cell the fracture crosses. 
The pressure remains continuous across each fracture.
The blocking barriers are treated using the method established in Section \ref{sec:method}.
Two treatment are fully compatible and can be combined naturally within the same finite element formulation.
Where a fracture meets a barrier, we adopt a barrier-dominant rule: the fracture contribution is dropped in that cell and only the barrier will be considered, so the pressure is discontinuous at the intersection.
Two flow regimes are computed on the same mesh: in the first, the top boundary is held at $p=4$ and the bottom at $p=1$; in the second, the left at $p=4$ and the right at $p=1$; the two remaining sides are impermeable in each regime. 
The unstructured mesh has size $h\approx0.03$,
and the two barriers are displaced by the small perturbation
$(6\times10^{-4},\,4\times10^{-4})$ to avoid degenerate incidences with the mesh.

Figures~\ref{fig:ex4-t2b} and~\ref{fig:ex4-l2r} show both regimes together with the background mesh, and the profile along the slice $(0,0.5)$--$(1,0.9)$, which crosses both types of fractures. In the profile, the two types are easy to tell apart: where the slice meets a conductive fracture the pressure stays continuous and only its slope changes, while at each of the two barriers it falls by a finite
jump, which $\widehat p_h$ resolves and $p_h$ rounds off. 
The reference is again the equidimensional mimetic finite-difference computation of \cite{flemisch2018benchmarks}, with about $2.3\times10^6$ cells for this geometry. 
The recovered pressure reproduces the position and the size of both jumps. 
The small remaining discrepancy near the fracture-barrier intersections may arise from the different treatments for them. Our method only considers the barrier effect, whereas the methods in \cite{flemisch2018benchmarks} use the harmonic mean of the fracture and barrier permeabilities.

\begin{figure}[htbp!]\centering
\begin{subfigure}[b]{0.32\textwidth}\includegraphics[width=\textwidth]{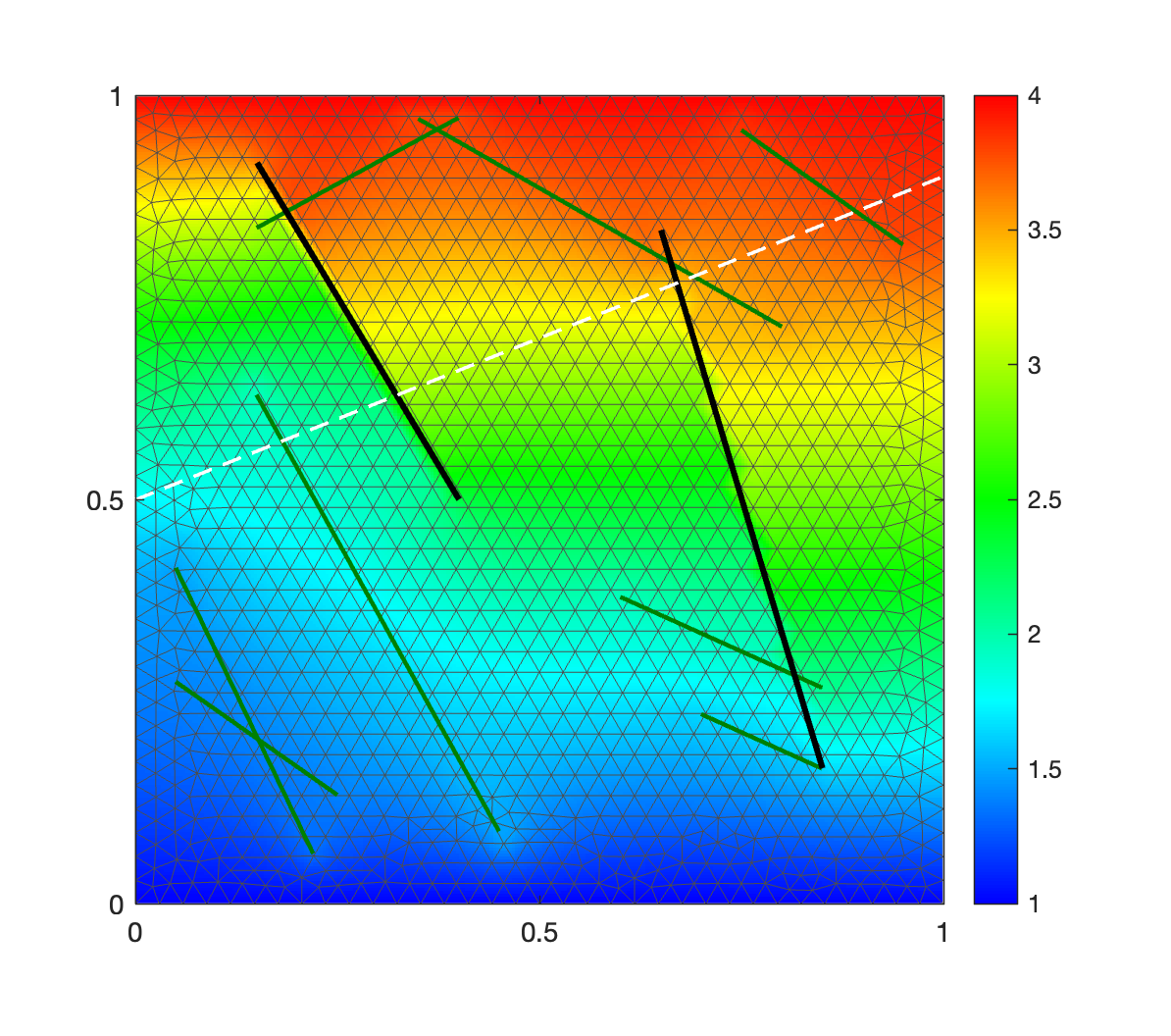}
  \caption{continuous $p_h$}\end{subfigure}\hfill
\begin{subfigure}[b]{0.32\textwidth}\includegraphics[width=\textwidth]{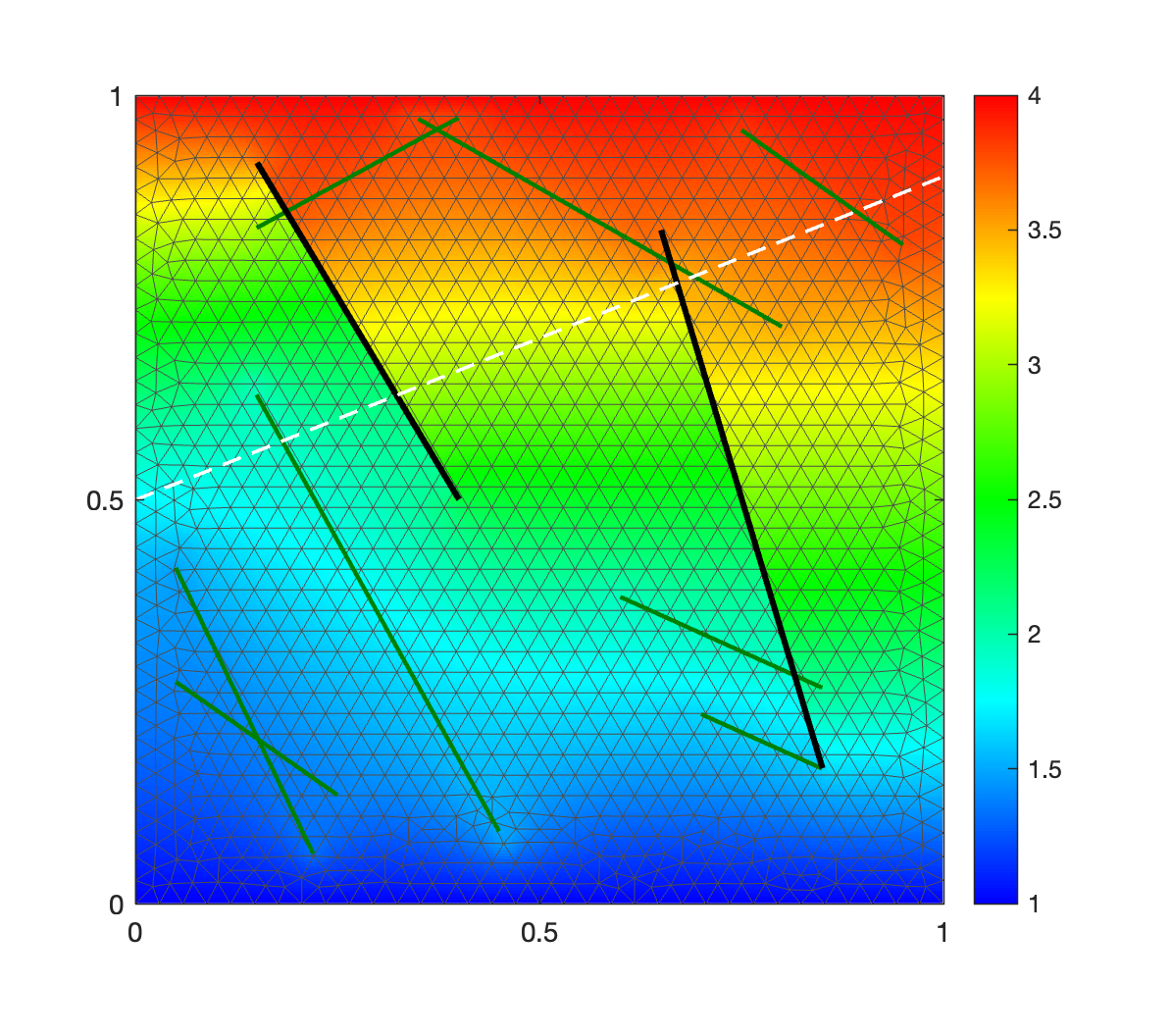}
  \caption{recovered broken $\widehat p_h$}\end{subfigure}\hfill
\begin{subfigure}[b]{0.34\textwidth}\includegraphics[width=\textwidth]{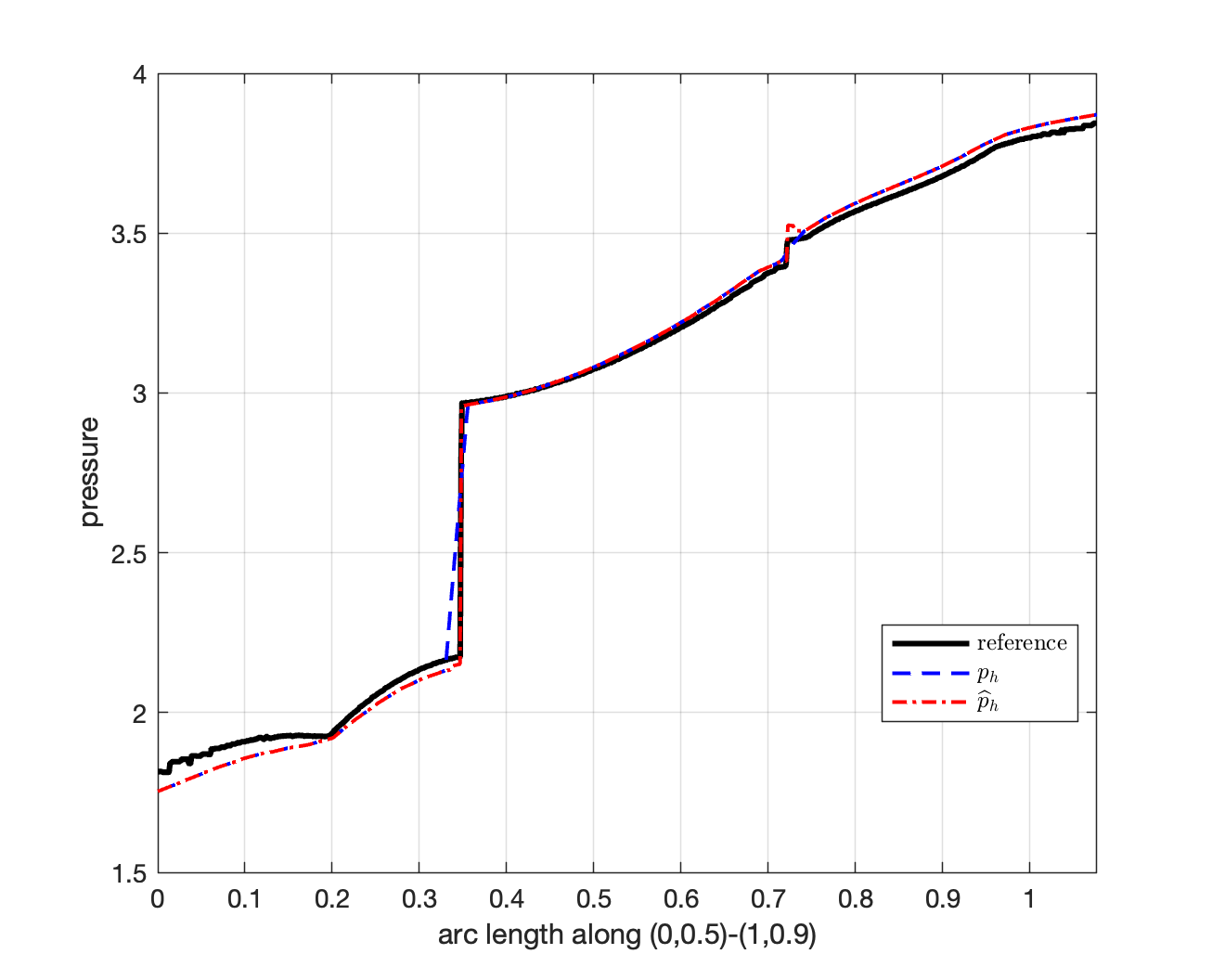}
  \caption{slice along $(0,0.5)$--$(1,0.9)$}\end{subfigure}
\caption{Example~4, top-to-bottom flow: (a) continuous pressure, (b) recovered broken pressure, and (c) the profile along $(0,0.5)$--$(1,0.9)$ against the mimetic finite-difference reference of
\cite{flemisch2018benchmarks}.
Fractures are shown in in green and barriers are in black.}
\label{fig:ex4-t2b}
\end{figure}

\begin{figure}[htbp!]\centering
\begin{subfigure}[b]{0.32\textwidth}\includegraphics[width=\textwidth]{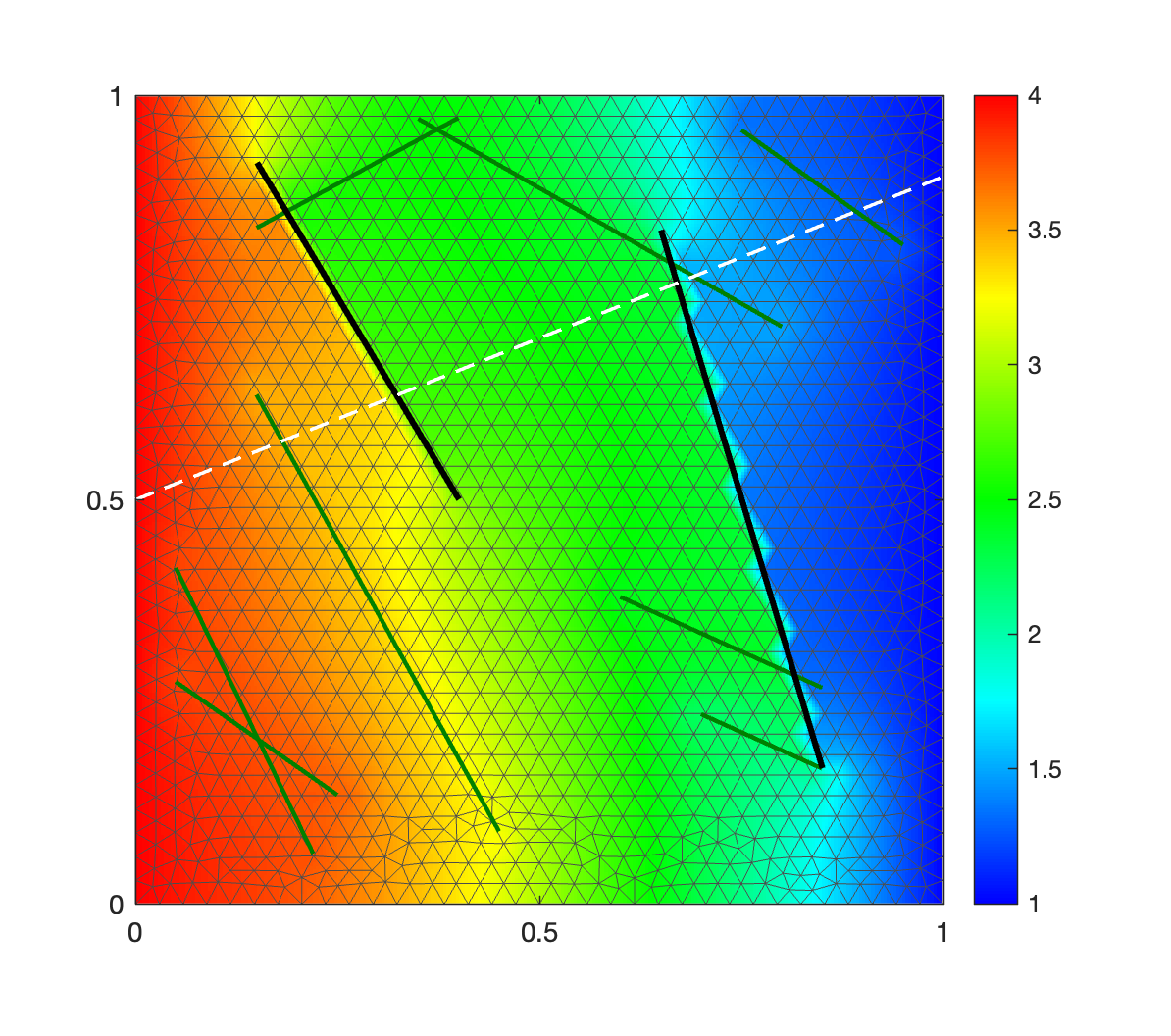}
  \caption{continuous $p_h$}\end{subfigure}\hfill
\begin{subfigure}[b]{0.32\textwidth}\includegraphics[width=\textwidth]{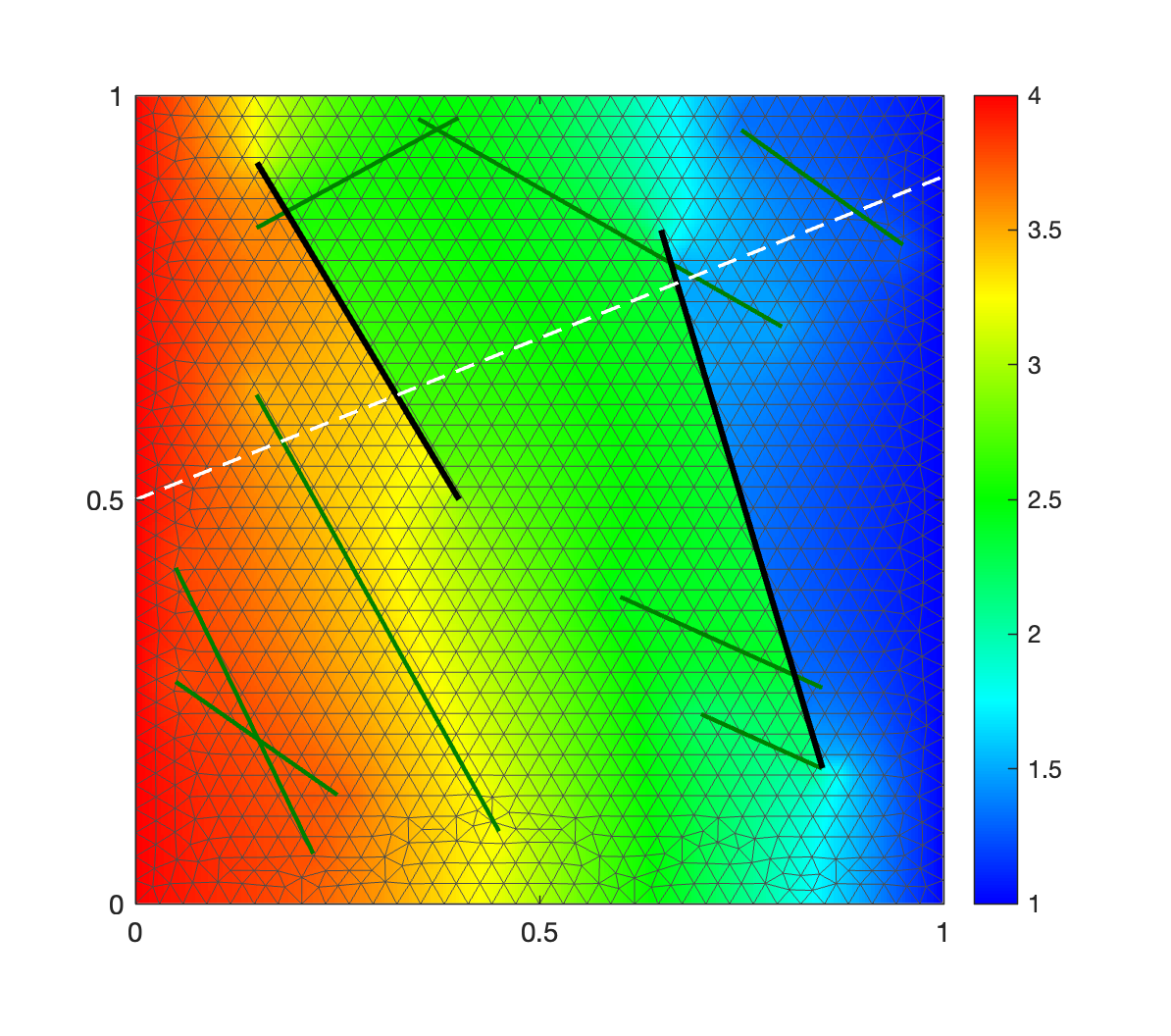}
  \caption{recovered broken $\widehat p_h$}\end{subfigure}\hfill
\begin{subfigure}[b]{0.34\textwidth}\includegraphics[width=\textwidth]{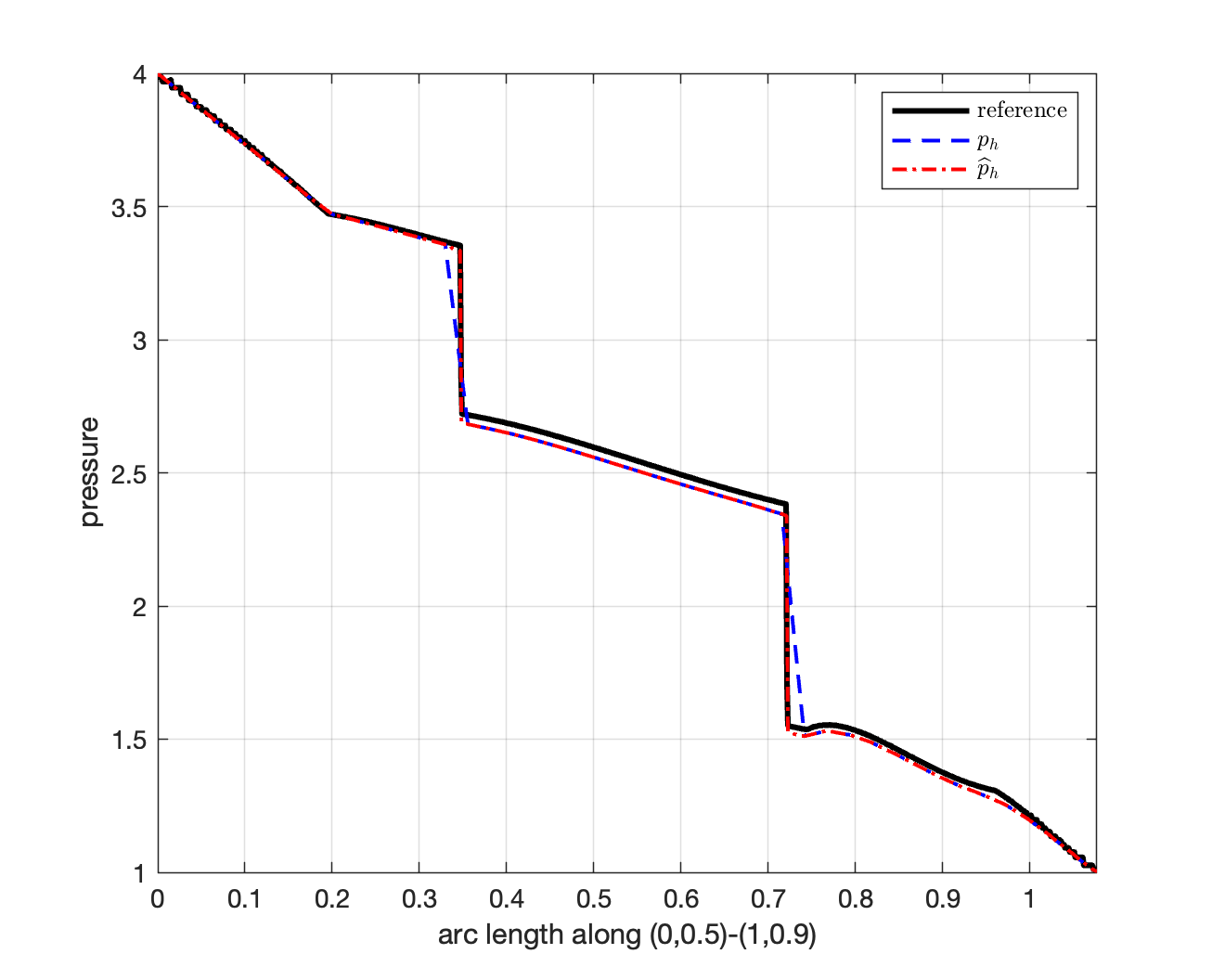}
  \caption{slice along $(0,0.5)$--$(1,0.9)$}\end{subfigure}
\caption{Example~4, left-to-right flow, on the same mesh: (a) continuous pressure, (b) recovered broken pressure, and (c) the profile along $(0,0.5)$--$(1,0.9)$ against the mimetic finite-difference reference of \cite{flemisch2018benchmarks}.
Fractures are shown in in green and barriers are in black.}
\label{fig:ex4-l2r}
\end{figure}

\subsection{Example 5: a realistic barrier network}\label{sec:ex5}

Our final $2$D example is a realistic network of $63$ low-permeability barriers. 
It was obtained in \cite{glaser2022comparison} by turning the conductive fracture network of \cite{flemisch2018benchmarks} into blocking barriers. 
On the rectangle $\Om=[0,700]\times[0,600]$ (in meters), a left-to-right flow is driven by the Dirichlet heads $p=1{,}013{,}250$ at $x=0$ and $p=0$ at $x=700$ (the heads are normalized to $1$ and $0$ below), with no-flow top and bottom. 
All barriers share the aperture $a=10^{-2}$ and permeability $k_b=10^{-18}$ in a matrix
of $k_m=10^{-14}$. 
Factoring out $k_m$ leaves an identity matrix permeability and a barrier resistance
$R=k_m a/k_b=100$. 
The barriers form a dense, repeatedly intersecting network, treated by the method of Section~\ref{sec:method}. 
The unstructured mesh has about
$10{,}150$ vertices and is fitted to none of the barriers.

Figure~\ref{fig:ex5} shows the two fields $p_h$ and $\widehat{p}_h$, and the profiles along the slices $(0,0)$--$(700,600)$ and $(625,0)$--$(625,600)$. 
The reference is a solution of the fitted Box-DFM of \cite{xu2024box} on a barrier-fitted grid of about $3\times10^5$ cells. 
On both slices the raw $p_h$ follows the overall descent of the head but flattens the individual jumps, while the recovery $\widehat{p}_h$ restores them.
The narrow peak that the vertical slice crosses near the outlet, clipped by $p_h$, is also recovered. 
A realistic blocking network is thus captured with no mesh adaptation to the barriers at all.

\begin{figure}[htbp!]\centering
\begin{subfigure}[b]{0.46\textwidth}\includegraphics[width=\textwidth]{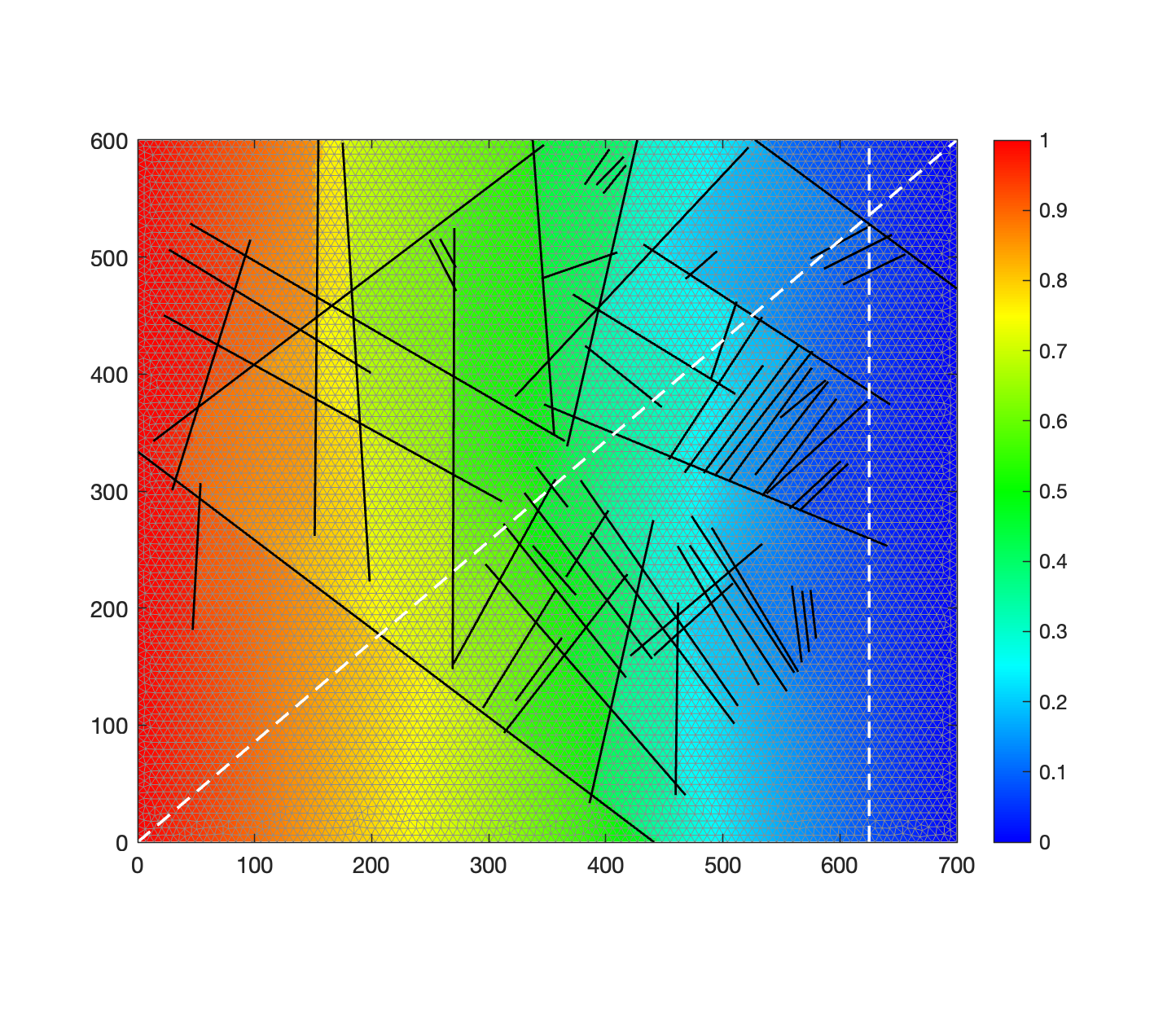}
  \caption{continuous $p_h$}\end{subfigure}\hfill
\begin{subfigure}[b]{0.46\textwidth}\includegraphics[width=\textwidth]{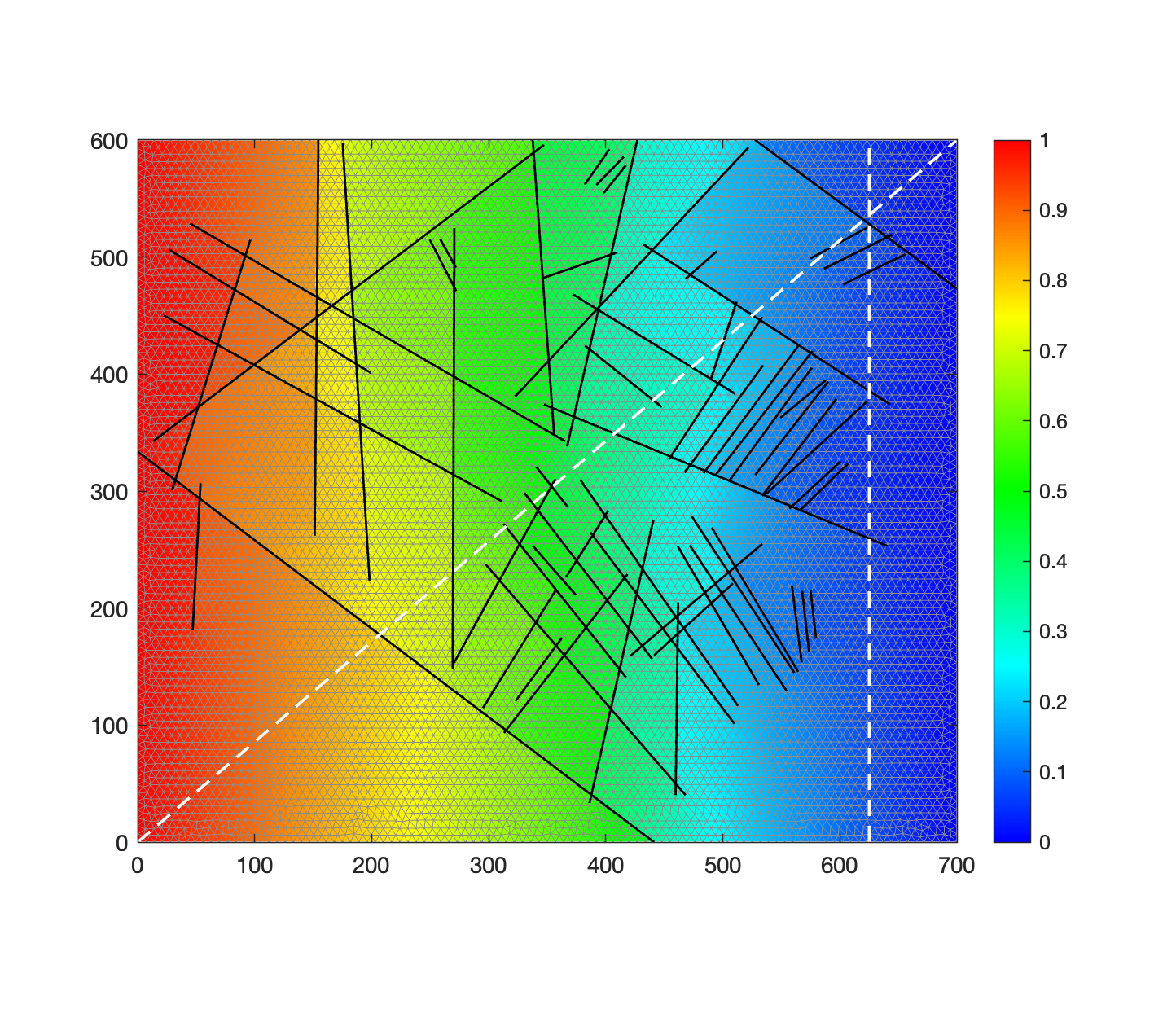}
  \caption{recovered broken $\widehat p_h$}\end{subfigure}

\begin{subfigure}[b]{0.46\textwidth}\includegraphics[width=\textwidth]{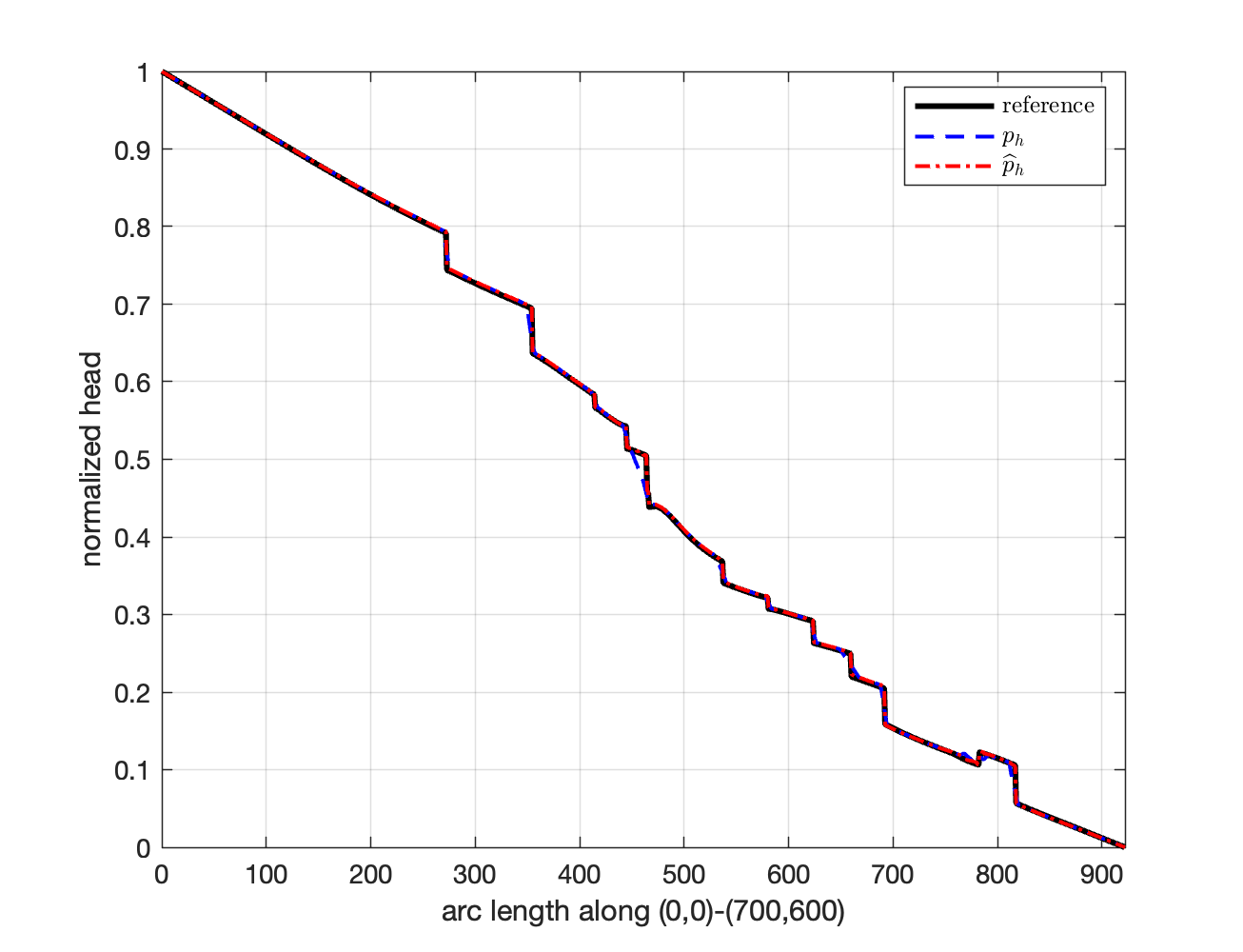}
  \caption{slice along $(0,0)$--$(700,600)$}\end{subfigure}\hfill
\begin{subfigure}[b]{0.46\textwidth}\includegraphics[width=\textwidth]{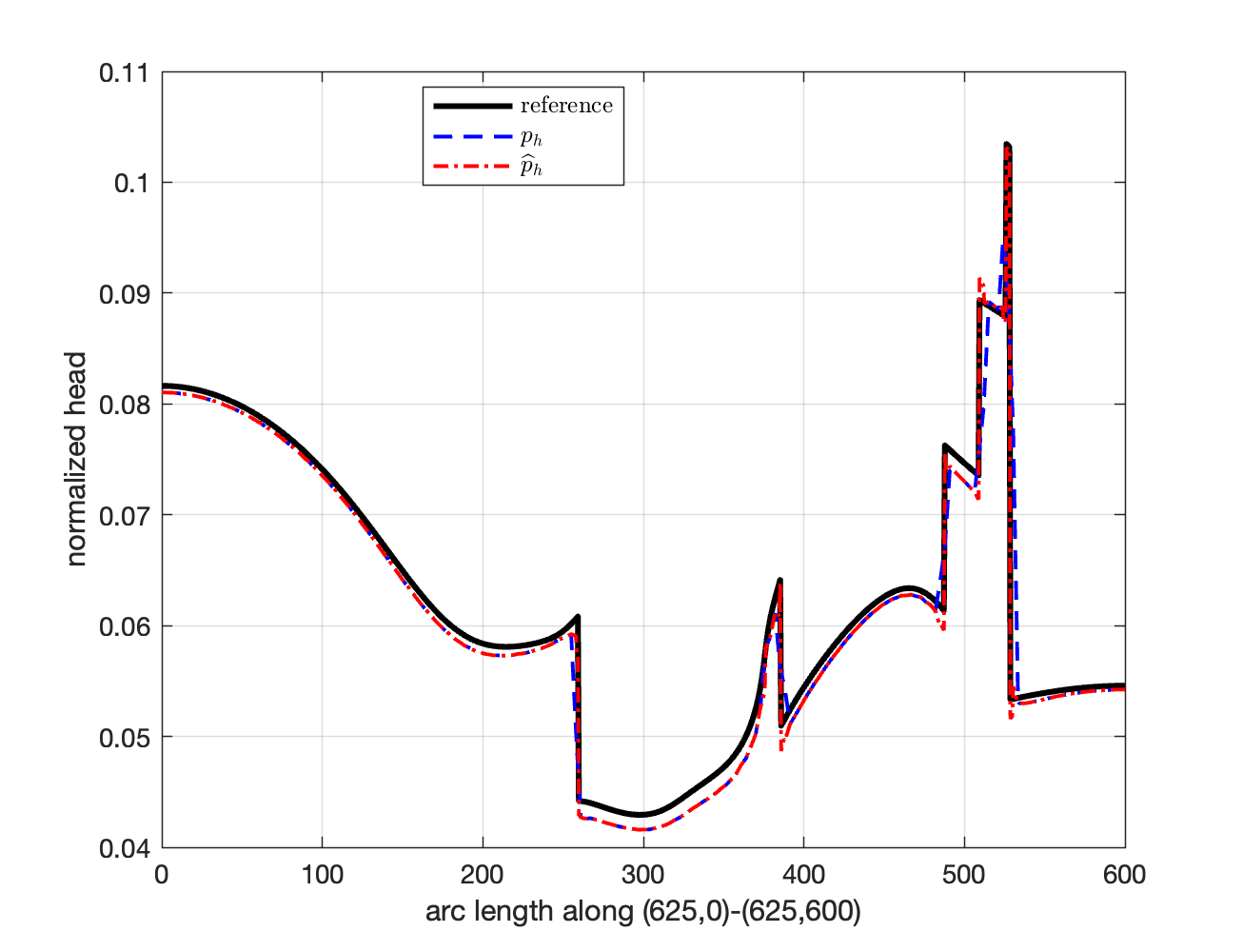}
  \caption{slice along $(625,0)$--$(625,600)$}\end{subfigure}
\caption{Example~5, realistic network of $63$ barriers: (a) continuous pressure, (b) recovered broken pressure, and the profiles along $(0,0)$--$(700,600)$ in (c) and $(625,0)$--$(625,600)$ in (d) against the fine Box-DFM reference of \cite{xu2024box}.}
\label{fig:ex5}
\end{figure}

\subsection{Example 6: a three-dimensional barrier network}\label{sec:ex6}

Our last example is three-dimensional. 
On the unit cube $\Om=[0,1]^3$, we take the regular blocking network of Case~2.2 of the benchmark study of Berre et al.~\cite{berre2021verification}: 
nine axis-oriented barrier planes, namely $x,y,z=\tfrac12$ (full planes), $x,y,z=\tfrac34$ (restricted to $[\tfrac12,1]^2$), and $x,y,z=\tfrac58$ (restricted to $[\tfrac12,\tfrac34]^2$), arranged so that triples of planes meet along the cube diagonal. 
Every barrier has aperture $a=10^{-4}$ and permeability $k_b=10^{-4}$, so its resistance is $R=a/k_b=1$. 
The matrix permeability is $\mathbf{K}_m=0.1\times\mathbf{I}$ in the subregion 
\begin{equation*}
\begin{split}
\Omega_{1}=&\{(x,y,z)\in\Omega: x>0.5, y<0.5\}\cup
\{(x,y,z)\in\Omega: x>0.75, 0.5<y<0.75, z>0.5\}\\
&\cup
\{(x,y,z)\in\Omega: 0.625<x<0.75, 0.5<y<0.625, 0.5<z<0.75\},
\end{split}
\end{equation*}
and is $\mathbf{K}_m=\mathbf{I}$ in the subregion $\Omega\setminus\Omega_{1}$. 
An inflow $\bu\cdot\bn=-1$ enters on the corner $\{x,y,z<\tfrac14\}$, the head is fixed to $p=1$ on $\{x,y,z>\tfrac78\}$, and the rest of the boundary is impermeable. 
We solve on a single unstructured tetrahedral mesh unfitted to any barrier. 
The barrier treatment follows Appendix~\ref{app:threed}, which is a direct extension of the methods in Section~\ref{sec:method}.

Figure~\ref{fig:ex6} shows pressures $p_h$ and $\widehat{p}_h$ on the cube surface and the head along the diagonal $(0,0,0)$--$(1,1,1)$. 
On the surface, the traces of the barrier planes are invisible in $p_h$ and
appear as sharp color breaks in $\widehat p_h$. 
Along the sampling line, the recovered head $\widehat{p}_h$ places all three jumps at the correct arc lengths with the correct magnitudes and stays close to the reference \cite{berre2021verification}, while the raw $p_h$ stretches each jump over a cell.

\begin{figure}[htbp!]\centering
\begin{subfigure}[b]{0.32\textwidth}\includegraphics[width=\textwidth]{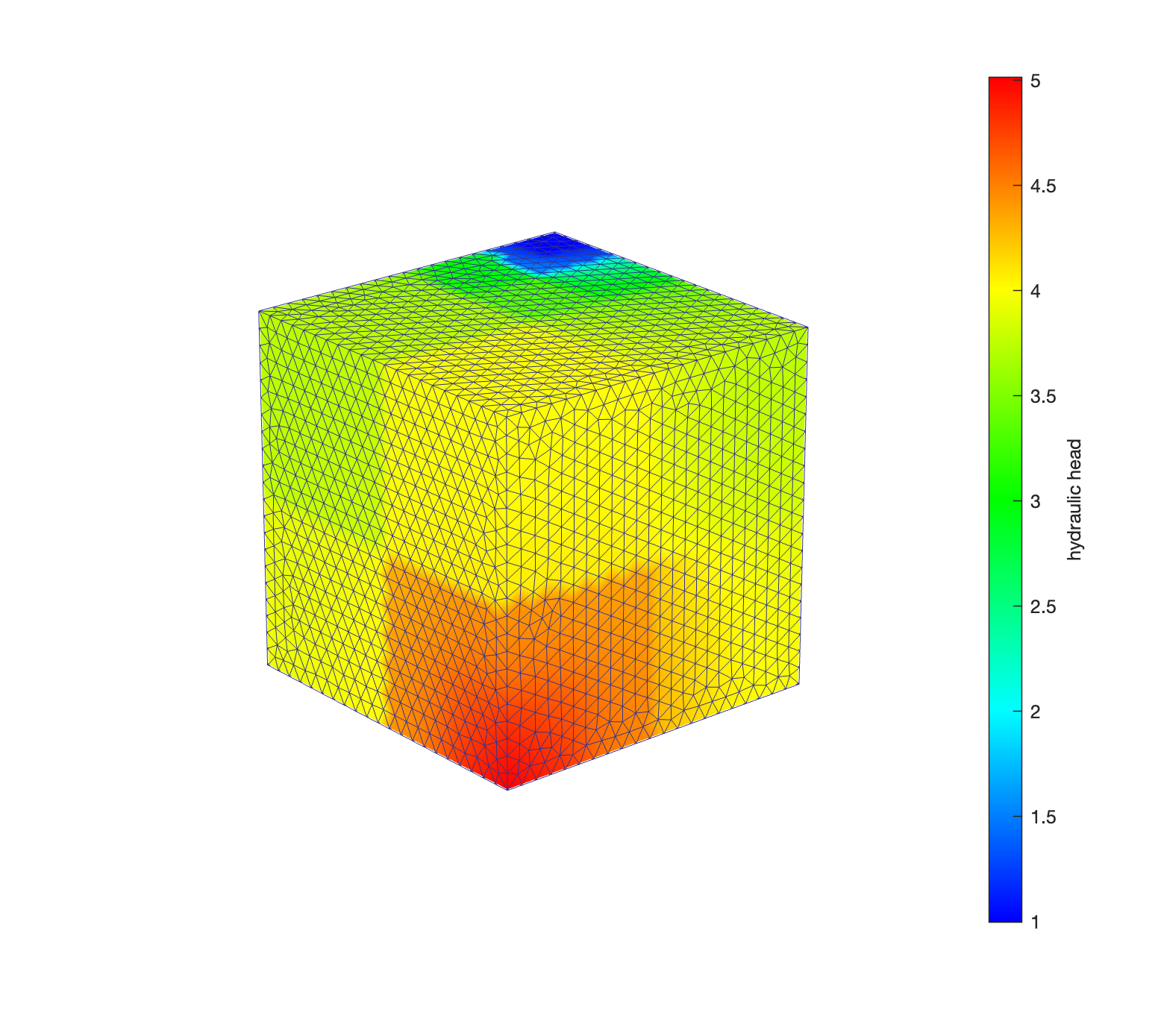}
  \caption{continuous $p_h$}\end{subfigure}\hfill
\begin{subfigure}[b]{0.32\textwidth}\includegraphics[width=\textwidth]{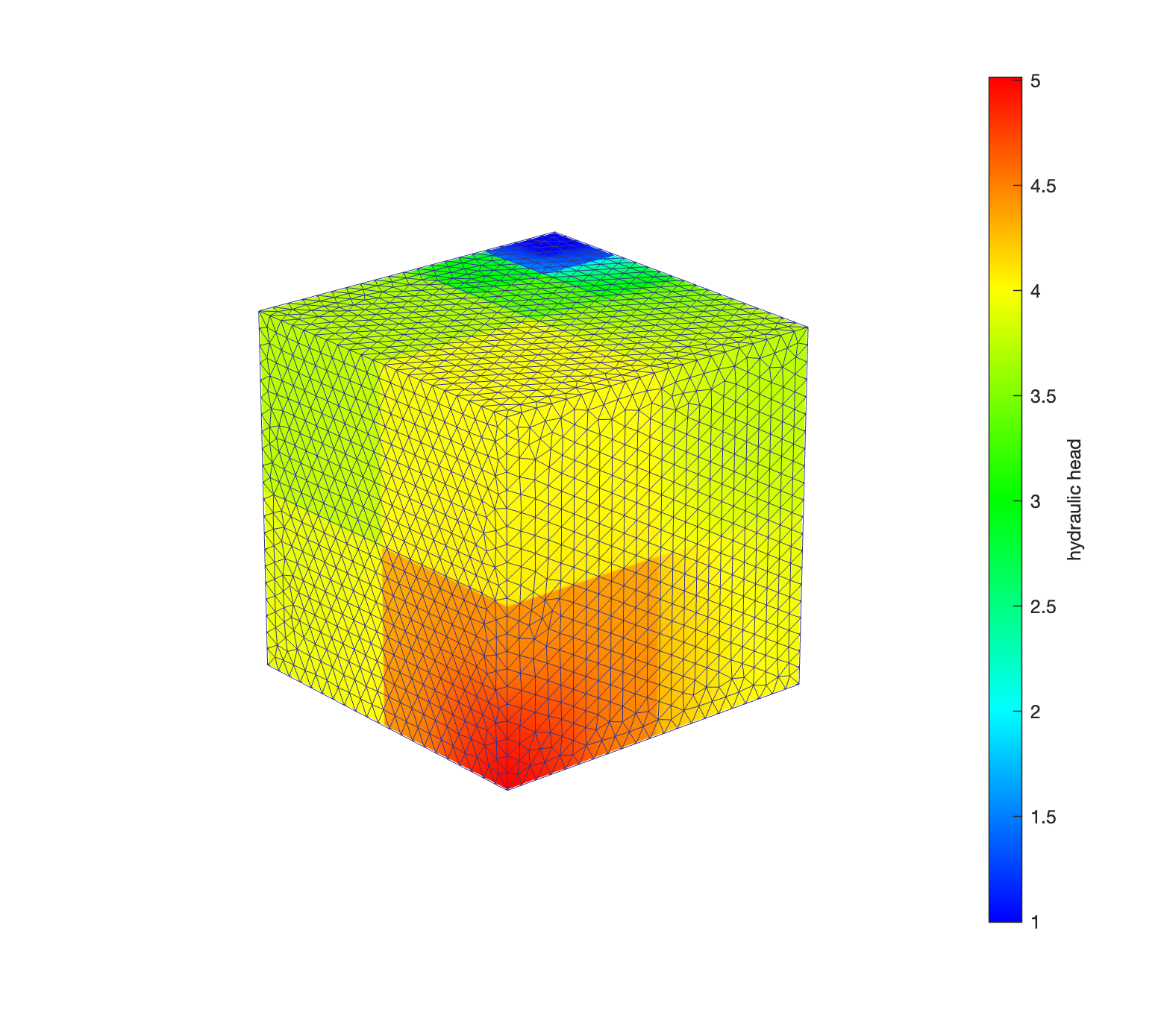}
  \caption{recovered broken $\widehat p_h$}\end{subfigure}\hfill
\begin{subfigure}[b]{0.34\textwidth}\includegraphics[width=\textwidth]{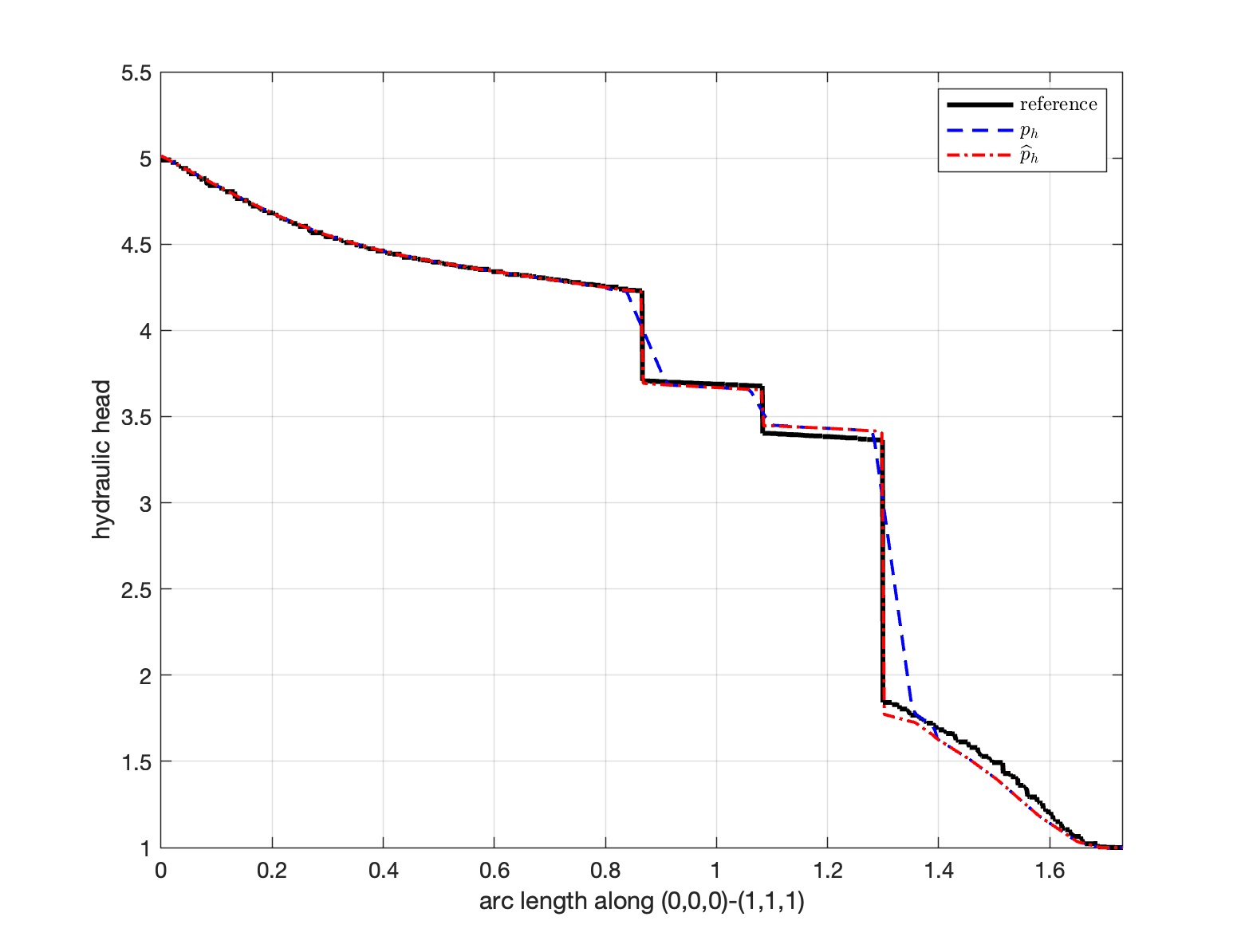}
  \caption{slice along $(0,0,0)$--$(1,1,1)$}\end{subfigure}
\caption{Example~6, nine-barrier network on an unfitted tetrahedral mesh: (a) continuous pressure 
and (b) recovered broken pressure on the cube surface, and (c) the head along the diagonal $(0,0,0)$--$(1,1,1)$ against the benchmark reference solution.}
\label{fig:ex6}
\end{figure}

\section{Conclusions}\label{sec:conclusion}

We have presented an extension of the linear finite element discrete fracture model to low-permeability barriers.
The method minimizes the broken energy of the interface model over the $H^1$-conforming space augmented, on each cut cell, by a local jump variable, which is then eliminated in closed form. 
What results is a local modification of the element stiffness matrices: the rank-one update \eqref{eq:cellupdate} on cells cut by a single chord, the modification \eqref{eq:cellupdate-junction} on cells containing multiple barriers, and the three-dimensional counterparts \eqref{eq:cellupdate3d} and \eqref{eq:cellupdate-junction3d}. 
The modification keeps the assembled system symmetric positive definite on any triangulation. 
The pressure jumps eliminated are recovered by local post-processing: the broken reconstruction removes the jump-smearing and raises the accuracy well above the jump-limited rate. 
Behind the formulas lies a simple physical picture, developed in Appendix~\ref{app:network}: the $P^1$ system is the nodal current balance of a resistor network, and a barrier enters by splitting a hidden local node and inserting a barrier resistor between the two copies.

A forthcoming work extends the same energy-condensation principle to a finite difference discretization on Cartesian grids \cite{xu2026fdm}.
Several directions remain open. 
We are pursuing an a priori error analysis for the recovered broken pressure $\widehat{p}_h$. 
A richer interface representation, with more than one jump unknown per cut element, may bring the recovered accuracy for general unfitted barriers to the optimal order. 
The same idea of local discrete operator modification will be extended to TPFA and MPFA discrete fracture models on unfitted meshes. 
Moreover, an analogous complementary-energy formulation for mixed finite element discretizations on unfitted meshes will also be investigated.

\begin{appendices}

\section{A resistor-network interpretation}\label{app:network}

The derivation in Section~\ref{sec:method} is variational.  In this appendix we
give a complementary circuit interpretation, whose purpose is to make the local
modification \eqref{eq:cellupdate} physically transparent. 
This was, in fact, the original starting point of the present work, and the method proposed here was first derived from this perspective.

The basic observation is that the $P^1$-finite element equations can be read as
Kirchhoff current balances on a resistor network: nodal values are voltages,
non-diagonal entries of the stiffness matrix are edge conductances, and the finite element left-hand side
at a node is the net current flowing out of that node. 
The perspective of equivalent circuit representations is classical \cite{kron1945numerical, macneal1953asymmetrical, duffin1959distributed}.
In particular, Duffin \cite{duffin1959distributed} identified
the $P^1$-finite element discretization of the Laplacian as a resistor network carrying the cotangent conductances.
From this point of view, a blocking barrier is naturally represented by inserting an additional barrier resistor.  
The cut-cell update \eqref{eq:cellupdate} is then exactly what results when the
vertex separated by the barrier is split into two copies, a barrier resistor is
inserted between the copies, and the cell-local copy is eliminated by Kirchhoff's
law.

\subsection*{A.1 \; The $P^1$-finite element equations as Kirchhoff balances}

To keep the discussion focused on the mechanism, we consider a typical interior node.  
Suppose first that a resistor network is built on the mesh graph; see Figure \ref{fig:mesh-circuit}.  
If $u_i$ denotes the voltage at node $i$, and if the edge $(i,j)$ has conductance $c_{ij}(=c_{ji})$, then Ohm's law gives the current from $i$ to $j$ as
\[
I_{ij}=c_{ij}(u_i-u_j).
\]
Kirchhoff's current balance at node $i$ is therefore
\[
\sum_{j\sim i} c_{ij}(u_i-u_j)=b_i,
\]
where $b_i$ denotes the current source assigned to the node $i$.

The $P^1$-finite element method for $-\nabla\cdot(\mathbf K_m\nabla p)=f$ has exactly this form.  Let $\{\phi_i\}$ be the nodal Lagrange basis and define
\[
K_{ij}
=
\sum_{T\in\mathcal T_h}
\int_T \mathbf K_m\nabla\phi_j\cdot\nabla\phi_i\,d\mathbf{x}.
\]
Since $\sum_j\phi_j=1$, the stiffness matrix has zero row sum,
\[
\sum_j K_{ij}=0.
\]
Thus, if we set $c_{ij}:=-K_{ij}$ for $i\ne j$, then
\[
K_{ii}=-\sum_{j\ne i}K_{ij}=\sum_{j\ne i}c_{ij}.
\]
Consequently,
\[
(Ku)_i
=
K_{ii}u_i+\sum_{j\ne i}K_{ij}u_j
=
\sum_{j\ne i}c_{ij}(u_i-u_j).
\]
The finite element equation $Ku=b$, with $b_i
=
\int_\Omega f\phi_i\,d\mathbf{x}$, is therefore exactly Kirchhoff's current balance at each node.  In this
identification, the finite element nodal values are voltages, the off-diagonal
stiffness entries define edge conductances $c_{ij}=-K_{ij}$, and the finite
element left-hand side $(Ku)_i$ is the net current flowing out of node $i$.

\begin{figure}[htbp!]\centering
\begin{tikzpicture}[scale=0.82,>=Latex]
\begin{scope}
  \draw[thick] (0,0)--(0.6,2); \draw[thick] (0,0)--(2.2,0.6); \draw[thick] (0,0)--(1.7,-1.5);
  \draw[thick] (0,0)--(-0.5,-1.8); \draw[thick] (0,0)--(-2.1,-0.8); \draw[thick] (0,0)--(-1.5,1.5);
  \draw[thick] (0.6,2)--(2.2,0.6); \draw[thick] (2.2,0.6)--(1.7,-1.5); \draw[thick] (1.7,-1.5)--(-0.5,-1.8);
  \draw[thick] (-0.5,-1.8)--(-2.1,-0.8); \draw[thick] (-2.1,-0.8)--(-1.5,1.5); \draw[thick] (-1.5,1.5)--(0.6,2);
  \draw[thick,path fading=fade down] (0.6,2)--(-0.1,2.7); \draw[thick,path fading=fade down] (0.6,2)--(0.9,2.8); \draw[thick,path fading=fade down] (0.6,2)--(1.5,2.2);
  \draw[thick,path fading=fade right] (2.2,0.6)--(2.6,1.5); \draw[thick,path fading=fade right] (2.2,0.6)--(3.3,-0.2);
  \draw[thick,path fading=fade right] (1.7,-1.5)--(3.0,-1.0); \draw[thick,path fading=fade right] (1.7,-1.5)--(2.9,-2.5); \draw[thick,path fading=fade up] (1.7,-1.5)--(1.3,-2.6);
  \draw[thick,path fading=fade up] (-0.5,-1.8)--(0.2,-2.7); \draw[thick,path fading=fade up] (-0.5,-1.8)--(-1.8,-2.5);
  \draw[thick,path fading=fade up] (-2.1,-0.8)--(-2.2,-2.0); \draw[thick,path fading=fade up] (-2.1,-0.8)--(-3.2,-1.0); \draw[thick,path fading=fade down] (-2.1,-0.8)--(-3.3,0.5);
  \draw[thick,path fading=fade left] (-1.5,1.5)--(-2.8,1.0); \draw[thick,path fading=fade left] (-1.5,1.5)--(-2.8,2.0); \draw[thick,path fading=fade down] (-1.5,1.5)--(-1.0,2.4);
  \foreach \P in {(0,0),(0.6,2),(2.2,0.6),(1.7,-1.5),(-0.5,-1.8),(-2.1,-0.8),(-1.5,1.5)}\fill \P circle(1.9pt);
  \node[font=\small] at (0.0,-3.5) {(a) FEM mesh};
\end{scope}
\begin{scope}[shift={(8,0)}]
  \draw[thick] (0,0) to[R] (0.6,2); \draw[thick] (0,0) to[R] (2.2,0.6); \draw[thick] (0,0) to[R] (1.7,-1.5);
  \draw[thick] (0,0) to[R] (-0.5,-1.8); \draw[thick] (0,0) to[R] (-2.1,-0.8); \draw[thick] (0,0) to[R] (-1.5,1.5);
  \draw[thick] (0.6,2) to[R] (2.2,0.6); \draw[thick] (2.2,0.6) to[R] (1.7,-1.5); \draw[thick] (1.7,-1.5) to[R] (-0.5,-1.8);
  \draw[thick] (-0.5,-1.8) to[R] (-2.1,-0.8); \draw[thick] (-2.1,-0.8) to[R] (-1.5,1.5); \draw[thick] (-1.5,1.5) to[R] (0.6,2);
  \draw[thick,path fading=fade down] (0.6,2) to[R] (-0.1,2.7); \draw[thick,path fading=fade down] (0.6,2) to[R] (0.9,2.8); \draw[thick,path fading=fade down] (0.6,2) to[R] (1.5,2.2);
  \draw[thick,path fading=fade right] (2.2,0.6) to[R] (2.6,1.5); \draw[thick,path fading=fade right] (2.2,0.6) to[R] (3.3,-0.2);
  \draw[thick,path fading=fade right] (1.7,-1.5) to[R] (3.0,-1.0); \draw[thick,path fading=fade right] (1.7,-1.5) to[R] (2.9,-2.5); \draw[thick,path fading=fade up] (1.7,-1.5) to[R] (1.3,-2.6);
  \draw[thick,path fading=fade up] (-0.5,-1.8) to[R] (0.2,-2.7); \draw[thick,path fading=fade up] (-0.5,-1.8) to[R] (-1.8,-2.5);
  \draw[thick,path fading=fade up] (-2.1,-0.8) to[R] (-2.2,-2.0); \draw[thick,path fading=fade up] (-2.1,-0.8) to[R] (-3.2,-1.0); \draw[thick,path fading=fade down] (-2.1,-0.8) to[R] (-3.3,0.5);
  \draw[thick,path fading=fade left] (-1.5,1.5) to[R] (-2.8,1.0); \draw[thick,path fading=fade left] (-1.5,1.5) to[R] (-2.8,2.0); \draw[thick,path fading=fade down] (-1.5,1.5) to[R] (-1.0,2.4);
  \foreach \P in {(0,0),(0.6,2),(2.2,0.6),(1.7,-1.5),(-0.5,-1.8),(-2.1,-0.8),(-1.5,1.5)}\fill \P circle(1.9pt);
  \node[font=\small] at (0.0,-3.5) {(b) equivalent circuit network};
\end{scope}
\end{tikzpicture}
\caption{The $P^1$-finite element system as a resistor network. 
The nodal values are the node voltages, and the finite element equation at a node is Kirchhoff's current balance.}
\label{fig:mesh-circuit}
\end{figure}
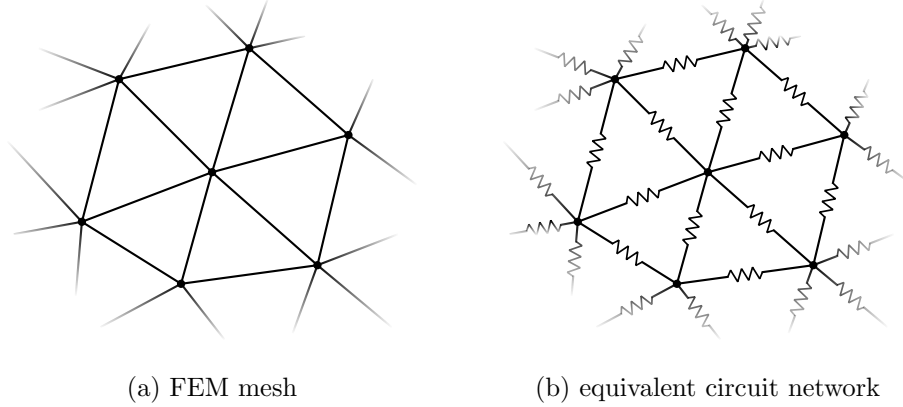

It is useful to distinguish the assembled edge conductance from its elemental
contributions.  If an interior edge $e=(i,j)$ is shared by two elements
$T^+$ and $T^-$, then the assembled stiffness entry is the sum of the two
local contributions,
\[
K_{ij}
=
(K_{T^+})_{ij}+(K_{T^-})_{ij}.
\]
Equivalently, if
\[
c_{ij}^T:=-(K_T)_{ij},
\]
denotes the conductance contribution supplied by element $T$, then the assembled
edge conductance is
\[
c_{ij}
=
c_{ij}^{T^+}+c_{ij}^{T^-}.
\]
The same additivity holds for the current:
\[
I_{ij}
=
c_{ij}(u_i-u_j)
=
c_{ij}^{T^+}(u_i-u_j)+c_{ij}^{T^-}(u_i-u_j).
\]
Thus a local stiffness matrix may be viewed as a local subnetwork whose branch
currents are assembled into the global circuit.  In particular, $(K_Tu_T)_i$
is the current contribution flowing out of node $i$ through the local branches
of element $T$, whereas $(Ku)_i$ is the total assembled current flowing out of
the global node $i$.  This local viewpoint is the one used below for cut cells.

\subsection*{A.2 \; The barrier as an inserted resistor}

We now apply the same circuit viewpoint at the level of a single cut element.
Let $T=(\alpha,\beta,\gamma)$ be a cut cell, and suppose that the barrier
separates the lone vertex $\gamma$ from the two vertices $\alpha$ and
$\beta$.  
The local stiffness matrix $K_T$ represents the original element subnetwork before
the barrier is inserted.  
In the notation of the previous subsection,
$(K_Tu_T)_i$ is the current contribution flowing out of node $i$ through the local branches
of element $T$.

The physical effect of a blocking barrier is that a flux crossing the barrier
requires an additional pressure drop.  In the interface law, this drop is
proportional to the cross-barrier flux, with resistance parameter $R$.  This is
precisely the role of a resistor in a circuit.  If the barrier cuts the element
$T$ along a chord of length $\ell_T$, then the barrier segment inside this
cell is represented by a barrier resistor with conductance
\[
C_T=\frac{\ell_T}{R},
\]
or equivalently with resistance $R/\ell_T$.

The corresponding circuit operation is a change of local topology.  The lone vertex
$\gamma$ is split into two copies,
\[
\gamma_{\rm in}
\qquad\text{and}\qquad
\gamma_{\rm out}.
\]
The copy $\gamma_{\rm out}$ is the retained global node, while
$\gamma_{\rm in}$ is a cell-local hidden copy.  The original element subnetwork
$K_T$ is attached to $\gamma_{\rm in}$, not directly to
$\gamma_{\rm out}$.  Thus the local connections formerly associated with
$\gamma$ are rewired as
\[
\alpha \leftrightarrow \gamma_{\rm in},
\qquad
\beta \leftrightarrow \gamma_{\rm in},
\]
and the barrier is represented by the additional resistor
\[
\gamma_{\rm in}
\leftrightarrow
\gamma_{\rm out}
\]
with conductance $C_T=\ell_T/R$.  In this enlarged local circuit,
$\gamma_{\rm out}$ is connected to the element subnetwork only through the
barrier resistor; see Figure~\ref{fig:barrier-circuit}.

\begin{figure}[htbp!]
\centering
\begin{tikzpicture}[scale=1.0,>=Latex]

\ctikzset{
  bipoles/length=0.55cm,
  resistors/zigs=4,
  resistors/thickness=0.8
}

\begin{scope}[shift={(0,0)}]

  \coordinate (a) at (0,0);
  \coordinate (b) at (4.2,0);
  \coordinate (g) at (2.2,3.5);

  \coordinate (x1) at ($(a)!0.62!(g)$);
  \coordinate (x2) at ($(b)!0.62!(g)$);

  \draw[gray!60,thick] (a)--(g)--(b)--cycle;

  \draw[blue,thick,dashed] ($(x1)!-0.28!(x2)$)--(x1);
  \draw[blue,very thick] (x1)--(x2);
  \draw[blue,thick,dashed] (x2)--($(x2)!-0.28!(x1)$);
  \node[blue,right] at ($(x2)!-0.28!(x1)+(0.08,0.05)$) {$\Gamma$};

  \draw[thick] (a) to[R] (b);
  \draw[thick] (a) to[R] (g);
  \draw[thick] (g) to[R] (b);

  \fill (a) circle(2.4pt);
  \fill (b) circle(2.4pt);
  \fill (g) circle(2.4pt);

  \node[below left] at (a) {$\alpha$};
  \node[below right] at (b) {$\beta$};
  \node[above] at (g) {$\gamma$};

  \node[font=\small] at (2.1,-0.78) {(a) original element subnetwork};

\end{scope}

\draw[->,very thick] (5.3,1.8) -- (7.0,1.8);
\node[align=center] at (6.15,2.45)
  {\small split $\gamma$ into $\gamma_{\rm in}$ and $\gamma_{\rm out}$\\[-1pt]
   \small insert barrier resistor};

\begin{scope}[shift={(8.2,0)}]

  \coordinate (a) at (0,0);
  \coordinate (b) at (4.2,0);
  \coordinate (gout) at (2.2,3.5);

  \coordinate (gin) at (2.2,1.7);

  \coordinate (x1) at ($(a)!0.62!(gout)$);
  \coordinate (x2) at ($(b)!0.62!(gout)$);

  \draw[gray!20,thick] (a)--(gout)--(b)--cycle;

  \draw[blue,thick,dashed] ($(x1)!-0.28!(x2)$)--(x1);
  \draw[blue,very thick] (x1)--(x2);
  \draw[blue,thick,dashed] (x2)--($(x2)!-0.28!(x1)$);
  \node[blue,right] at ($(x2)!-0.28!(x1)+(0.08,0.05)$) {$\Gamma$};

  \draw[thick] (a) to[R] (b);
  \draw[thick] (a) to[R] (gin);
  \draw[thick] (gin) to[R] (b);

  \draw[red,line width=1.1pt] (gin) to[R] (gout);

  \fill (a) circle(2.4pt);
  \fill (b) circle(2.4pt);
  \fill[red] (gin) circle(2.6pt);
  \fill (gout) circle(2.4pt);

  \node[below left] at (a) {$\alpha$};
  \node[below right] at (b) {$\beta$};
  \node[above] at (gout) {$\gamma_{\rm out}$};
  \node[red,left] at ($(gin)+(-0.18,-0.05)$) {$\gamma_{\rm in}$};

  \node[red,right] at ($(gin)!0.58!(gout)+(0.22,0.05)$) {$\ell_T/R$};

  \node[font=\small] at (2.15,-0.78) {(b) rewired cut-cell subnetwork};

\end{scope}

\end{tikzpicture}
\caption{Local circuit interpretation of a cut cell. 
(a) The original element subnetwork associated with the local stiffness matrix
$K_T$ on $(\alpha,\beta,\gamma)$. 
(b) After the barrier is introduced, the lone vertex $\gamma$ is split into a
retained copy $\gamma_{\rm out}$ and a cell-local hidden copy
$\gamma_{\rm in}$. The original element subnetwork is attached to
$\gamma_{\rm in}$, and the barrier segment inside the cell is represented by a
barrier resistor of conductance $\ell_T/R$ connecting $\gamma_{\rm in}$ to
$\gamma_{\rm out}$.}
\label{fig:barrier-circuit}
\end{figure}
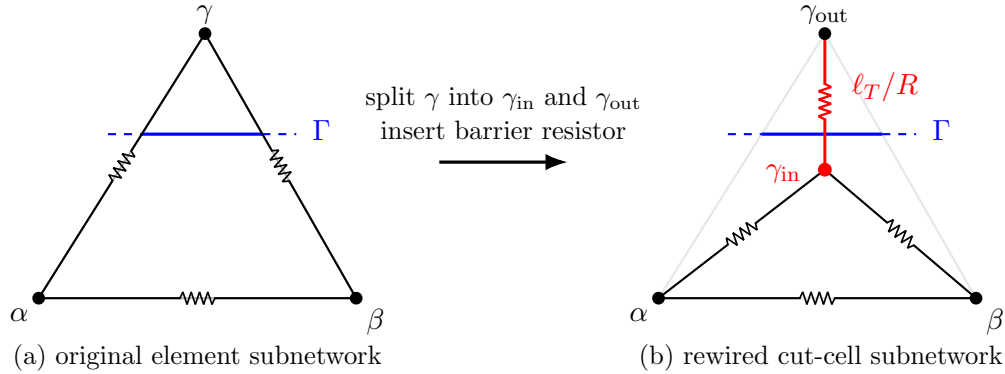

Since $\gamma_{\rm in}$ is an internal hidden node and carries no source, the
net current flowing out of this node must vanish.  The current flowing from
$\gamma_{\rm in}$ through the original element subnetwork is given by the
$\gamma$-row of $K_T$, evaluated with the hidden value
$u_{\gamma_{\rm in}}$:
\[
\bigl(K_T
(u_\alpha,u_\beta,u_{\gamma_{\rm in}})^{\transp}\bigr)_\gamma .
\]
The current flowing from $\gamma_{\rm in}$ through the barrier resistor to
$\gamma_{\rm out}$ is, by Ohm's law,
\[
\frac{\ell_T}{R}\,
(u_{\gamma_{\rm in}}-u_{\gamma_{\rm out}}).
\]
Kirchhoff's current balance at the hidden node therefore gives
\[
\bigl(K_T
(u_\alpha,u_\beta,u_{\gamma_{\rm in}})^{\transp}\bigr)_\gamma
+
\frac{\ell_T}{R}\,
(u_{\gamma_{\rm in}}-u_{\gamma_{\rm out}})
=0.
\]
This is the circuit form of the local stationarity condition used in the
variational derivation.

The hidden copy $\gamma_{\rm in}$ is eliminated locally by solving this Kirchhoff balance and substituting the result back into the retained-node equations.  
In circuit terminology, this elimination is called Kron reduction.
The result is an effective local circuit, or equivalently an
effective local stiffness matrix, on the retained nodes
\[
(u_\alpha,u_\beta,u_{\gamma_{\rm out}}).
\]

To identify this effective stiffness matrix from the circuit, solve for
$u_{\gamma_{\rm in}}$ from the Kirchhoff current balance at the hidden node:
\[
u_{\gamma_{\rm in}}
=
\frac{
\dfrac{\ell_T}{R}u_{\gamma_{\rm out}}
-
(K_T)_{\gamma\alpha}u_\alpha
-
(K_T)_{\gamma\beta}u_\beta
}
{
(K_T)_{\gamma\gamma}+\dfrac{\ell_T}{R}
}.
\]
For retained voltages
$(u_\alpha,u_\beta,u_{\gamma_{\rm out}})$, let
$I_\alpha$, $I_\beta$, and $I_{\gamma_{\rm out}}$ denote the net currents
flowing out of the retained nodes through this modified local subnetwork.  These
currents are the components of the effective local stiffness action.

At $\alpha$ and $\beta$, the currents are supplied by the original element
subnetwork:
\[
I_\alpha
=
(K_T)_{\alpha\alpha}u_\alpha
+
(K_T)_{\alpha\beta}u_\beta
+
(K_T)_{\alpha\gamma}u_{\gamma_{\rm in}},
\]
\[
I_\beta
=
(K_T)_{\beta\alpha}u_\alpha
+
(K_T)_{\beta\beta}u_\beta
+
(K_T)_{\beta\gamma}u_{\gamma_{\rm in}}.
\]
At the retained node $\gamma_{\rm out}$, the only local connection is the
barrier resistor, so Ohm's law gives
\[
I_{\gamma_{\rm out}}
=
\frac{\ell_T}{R}
\left(
u_{\gamma_{\rm out}}-u_{\gamma_{\rm in}}
\right).
\]
Substituting the expression for $u_{\gamma_{\rm in}}$ expresses these three
outgoing currents linearly in the retained voltages:
\[
\begin{pmatrix}
I_\alpha\\
I_\beta\\
I_{\gamma_{\rm out}}
\end{pmatrix}
=
\widetilde{K}_T
\begin{pmatrix}
u_\alpha\\
u_\beta\\
u_{\gamma_{\rm out}}
\end{pmatrix}.
\]
Reading off the coefficients gives
\[
\widetilde{K}_T
=
K_T
-
\frac{(K_T \mathbf{e}_\gamma) (K_T \mathbf{e}_\gamma)^{\transp}}
{(K_T)_{\gamma\gamma}+\ell_T/R}.
\]
This is exactly the cut-cell update
\eqref{eq:cellupdate}.

The operation is local in the same sense as finite element assembly.  
Each element contributes its own current branches before these contributions are assembled into global edge conductances.  
Thus each cut element introduces its own cell-local hidden node and its own barrier resistor.
Since these hidden nodes are not shared between elements, eliminating them in the global circuit is the same as eliminating them locally.

The circuit picture also gives the correct limiting behavior.  
If $R\to0$, then $\ell_T/R\to\infty$, so the barrier resistor becomes a short circuit and
$u_{\gamma_{\rm in}}$ is forced to coincide with $u_{\gamma_{\rm out}}$.  
The ordinary $P^1$ element is recovered.  If $R\to\infty$, then $\ell_T/R\to0$, so the barrier resistor becomes an open circuit and
$\gamma_{\rm out}$ is disconnected from the element subnetwork across the barrier.  
This is the fully blocking limit.

\section{Extension to three dimensions}\label{app:threed}

The construction of Section~\ref{sec:method} was carried out in two dimensions, but an extension to $\mathbb{R}^3$ is straightforward.
The barrier is now a codimension-one
\emph{surface} $\Gm\subset\Om\subset\mathbb R^3$ and the interface model of Section~\ref{sec:model} and its broken energy functional are unchanged.
Consider an unfitted tetrahedral triangulation
$\Th$ of $\Omega$. 
We show below that the chord of a cut triangle
is simply replaced by the cut face of a cut tetrahedron, and the chord length $\ell_T$
by the cut-face area $A_T$, after which the enriched element energy, its condensation to a rank-one
(or, at junctions, low-rank) stiffness update, the stability, and the broken-pressure recovery all
repeat line for line.

\subsection*{B.1 \; Cut geometry of a tetrahedron}

Call a tetrahedron $T$ \emph{cut} if the barrier plane crosses its interior. 
There are two cases (Figure~\ref{fig:cut3d}). In a $1$--$3$ split, one vertex lies on one side and the other
three on the other and $\Gm\cap T$
is a triangle. 
In a $2$--$2$ split, the vertices divide two against two and $\Gm\cap T$ is a quadrilateral. 
In either case $\Gm\cap T$ is a single planar polygon, whose area we denote $A_T=|\Gm\cap T|$.

To record which side of the plane a vertex lies on, let $\mathbf v\in\{0,1\}^4$ be the indicator of
the vertices on one chosen side. For a $1$--$3$ split $\mathbf v$ has a single unit entry, at the
lone vertex, and is then exactly the lone-vertex indicator $\egamma$ of Section~\ref{sec:method}. 
For a $2$--$2$ split $\mathbf v$ has two unit entries. The choice of side is arbitrary and is fixed once per cell.

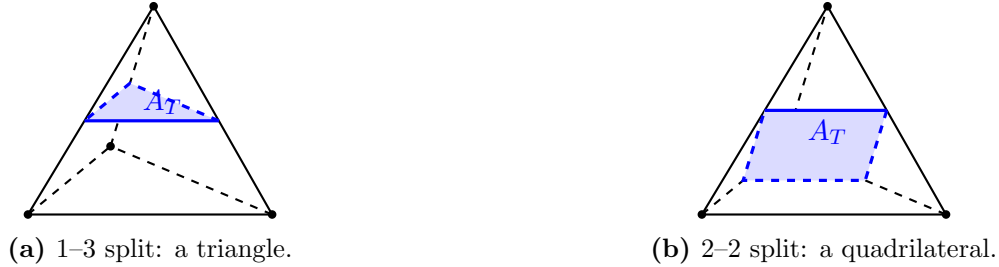
\begin{figure}[htbp!]\centering
\begin{subfigure}[b]{0.46\textwidth}\centering
\begin{tikzpicture}[scale=0.95,>=Latex]
  \coordinate (p1) at (0,0); \coordinate (p2) at (3.4,0);
  \coordinate (p3) at (1.15,0.95); \coordinate (p4) at (1.75,2.9);
  \draw[thick,dashed] (p1)--(p3); \draw[thick,dashed] (p2)--(p3); \draw[thick,dashed] (p3)--(p4);
  \draw[thick] (p1)--(p2)--(p4)--cycle;
  \coordinate (q1) at ($(p4)!0.55!(p1)$);
  \coordinate (q2) at ($(p4)!0.55!(p2)$);
  \coordinate (q3) at ($(p4)!0.55!(p3)$);
  \fill[blue!14] (q1)--(q2)--(q3)--cycle;
  \draw[blue,very thick] (q1)--(q2); \draw[blue,very thick,dashed] (q2)--(q3)--(q1);
  \foreach \p in {p1,p2,p3,p4}\fill (\p) circle(1.7pt);
  \node[blue] at ($(q1)!0.5!(q2)+(0.15,0.26)$){$A_T$};
\end{tikzpicture}
\caption{$1$--$3$ split: a triangle.}
\end{subfigure}\hfill
\begin{subfigure}[b]{0.46\textwidth}\centering
\begin{tikzpicture}[scale=0.95,>=Latex]
  \coordinate (p1) at (0,0); \coordinate (p2) at (3.4,0);
  \coordinate (p3) at (1.15,0.95); \coordinate (p4) at (1.75,2.9);
  \draw[thick,dashed] (p1)--(p3); \draw[thick,dashed] (p2)--(p3); \draw[thick,dashed] (p3)--(p4);
  \draw[thick] (p1)--(p2)--(p4)--cycle;
  \coordinate (r1) at ($(p4)!0.5!(p1)$);
  \coordinate (r2) at ($(p4)!0.5!(p2)$);
  \coordinate (r3) at ($(p3)!0.5!(p2)$);
  \coordinate (r4) at ($(p3)!0.5!(p1)$);
  \fill[blue!14] (r1)--(r2)--(r3)--(r4)--cycle;
  \draw[blue,very thick] (r1)--(r2); \draw[blue,very thick,dashed] (r2)--(r3)--(r4)--(r1);
  \foreach \p in {p1,p2,p4}\fill (\p) circle(1.7pt);
  \node[blue] at (1.75,1.15){$A_T$};
\end{tikzpicture}
\caption{$2$--$2$ split: a quadrilateral.}
\end{subfigure}
\caption{The two ways a plane cuts a tetrahedron. The cut face $\Gm\cap T$ (blue area $A_T$) replaces the chord of the two-dimensional cut cell.}
\label{fig:cut3d}
\end{figure}

\subsection*{B.2 \; The local stiffness matrix modification and its properties}

On a cut cell $T$ we introduce one scalar auxiliary unknown $t_T$, which is the pressure drop across the cut face $\Gm\cap T$. 
The bulk part of the energy must
act on the field after this drop has been removed from the chosen side, so the jump-corrected
nodal vector representing the bulk field is
\[
\uT^{\rm bulk}=\uT-t_T\mathbf v,\qquad \uT=(u_1,u_2,u_3,u_4)^{\transp},
\]
and the local energy functional of the cut cell is
\begin{equation}\label{eq:ET3d}
\mathcal E_T(\uT,t_T)=\tfrac12(\uT-t_T\mathbf v)^{\transp}K_T(\uT-t_T\mathbf v)+\frac{A_T}{2R}\,t_T^2 .
\end{equation}
The first term is the tetrahedral bulk energy of the jump-removed field and the second term is the one-point face quadrature of the Robin jump energy $\frac1{2R}\int_{\Gm\cap T}\jump{q}^2\,dA$.

As in the plane, $t_T$ is local to one cell and is eliminated by minimizing $\mathcal E_T(\mathbf{u}_{T}, t_{T})$ over it at fixed $\uT$. 
Stationarity reads
\begin{equation*}\label{eq:stationarity3d}
0=\frac{\partial\mathcal E_T}{\partial t_T}=-\mathbf v^{\transp}K_T(\uT-t_T\mathbf v)+\frac{A_T}{R}\,t_T ,
\end{equation*}
whose solution is
\begin{equation}\label{eq:tstar3d}
t_T^\star=\frac{\mathbf v^{\transp}K_T\uT}{\mathbf v^{\transp}K_T\mathbf v+A_T/R} .
\end{equation}
Substituting \eqref{eq:tstar3d} into \eqref{eq:ET3d} condenses the cell energy to a quadratic form in
the retained nodal unknowns,
\[
\min_{t_T}\mathcal E_T(\uT,t_T)=\tfrac12\,\uT^{\transp}\widetilde K_T\,\uT,
\]
with the modified local stiffness matrix on cut-cell,
\begin{equation}\label{eq:cellupdate3d}
\widetilde K_T=K_T-\frac{(K_T\mathbf v)(K_T\mathbf v)^{\transp}}{\mathbf v^{\transp}K_T\mathbf v+A_T/R} .
\end{equation}
This is the direct analog to the two-dimensional update \eqref{eq:cellupdate}.

The properties established in two dimensions carry over unchanged. 
The matrix $\widetilde K_T$ is symmetric; it has zero row sums, since $K_T\mathbf
1=0$ gives $\widetilde K_T\mathbf 1=0$; and it is positive semidefinite with kernel exactly the multiples of $\mathbf{1}$. 
The chosen side can be arbitrary as replacing $\mathbf v$ by $\mathbf
1-\mathbf v$ leaves $\widetilde K_T$ unchanged.

\subsection*{B.3 \; Junctions and recovery}

We now consider the case where several barrier surfaces meet inside one tetrahedron. 
Let the planes cut $T$ into subregions $\mathsf R_0,\dots,\mathsf R_m$, with $\mathsf R_0$ a chosen base, and attach a scalar offset $t_\rho$ to each non-base region, collected in $\mathbf t_T\in\mathbb R^m$. 
Moreover, index the pieces of cut face inside $T$ by $f$, each separating regions $\mathsf R_{\rho_f^-}$ and $\mathsf R_{\rho_f^+}$, and set $\mathbf d_f=\mathbf e_{\rho_f^+}-\mathbf e_{\rho_f^-}$ (with $\mathbf e_0=\mathbf 0$ and $\mathbf e_r$ the $r$-th coordinate vector of $\mathbb R^m$).

The bulk field is the pressure field after the offsets are removed, i.e., $\uT-E_T\mathbf t_T$, where the region incidence matrix $E_T\in\mathbb R^{4\times m}$ has a unit in row $v$, column $\rho$ when vertex $v$ lies in region $\mathsf R_\rho$. 
Two subregions that share a piece $f$ of cut face
are coupled by $\frac1{2R}\int\jump{u}^2\,dA$ over that piece.
Summing these contributions gives $\tfrac12\mathbf t_T^{\transp}\mathcal P_T\mathbf t_T$, with $\mathcal P_T=\sum_f\frac{A_f}{R_f}\,\mathbf d_f\mathbf d_f^{\transp}$, where $A_f$ is the area of piece $f$ and $R_f$ the resistance of the plane.

The cell energy is thus
\begin{equation}\label{eq:ETjunction3d}
\mathcal E_T(\uT,\mathbf t_T)=\tfrac12(\uT-E_T\mathbf t_T)^{\transp}K_T(\uT-E_T\mathbf t_T)
+\tfrac12\mathbf t_T^{\transp}\mathcal P_T\mathbf t_T ,
\end{equation}
and minimizing over $\mathbf t_T$ gives $\mathbf t_T^{\ast}=(E_T^{\transp}K_TE_T+\mathcal P_T)^{-1}E_T^{\transp}K_T\uT$ and the modified local stiffness matrix
\begin{equation}\label{eq:cellupdate-junction3d}
\widetilde K_T=K_T-K_TE_T\bigl(E_T^{\transp}K_TE_T+\mathcal P_T\bigr)^{-1}E_T^{\transp}K_T.
\end{equation} 
A single plane ($m=1$) has $E_T=\mathbf v$ and $\mathcal P_T=A_T/R$,
and \eqref{eq:cellupdate-junction3d} collapses to the rank-one modification \eqref{eq:cellupdate3d}.

The broken pressure of Section~\ref{sec:recovery} is recovered cellwise exactly as in the two dimensions.
Having solved the global system for the nodal field, the offsets are restored by back-substitution, and the broken value at a point $x$ in region $\mathsf R_\rho$ is the jump-corrected pressure field plus that region's offset.

\end{appendices}

\bibliographystyle{plain}
\bibliography{ref}

\begin{thebibliography}{10}

\bibitem{angot2009asymptotic}
Philippe Angot, Franck Boyer, and Florence Hubert.
\newblock Asymptotic and numerical modelling of flows in fractured porous
  media.
\newblock {\em ESAIM: Mathematical Modelling and Numerical Analysis},
  43(2):239--275, 2009.

\bibitem{baca1984modelling}
RG~Baca, RC~Arnett, and DW~Langford.
\newblock Modelling fluid flow in fractured-porous rock masses by
  finite-element techniques.
\newblock {\em International Journal for Numerical Methods in Fluids},
  4(4):337--348, 1984.

\bibitem{berre2021verification}
Inga Berre, Wietse~M Boon, Bernd Flemisch, Alessio Fumagalli, Dennis
  Gl{\"a}ser, Eirik Keilegavlen, Anna Scotti, Ivar Stefansson, Alexandru
  Tatomir, Konstantin Brenner, et~al.
\newblock Verification benchmarks for single-phase flow in three-dimensional
  fractured porous media.
\newblock {\em Advances in Water Resources}, 147:103759, 2021.

\bibitem{boon2018robust}
Wietse~M Boon, Jan~M Nordbotten, and Ivan Yotov.
\newblock Robust discretization of flow in fractured porous media.
\newblock {\em SIAM Journal on Numerical Analysis}, 56(4):2203--2233, 2018.

\bibitem{cervera2022comparative}
Miguel Cervera, GB~Barbat, Michele Chiumenti, and J-Y Wu.
\newblock A comparative review of {XFEM}, mixed {FEM} and phase-field models
  for quasi-brittle cracking.
\newblock {\em Archives of Computational Methods in Engineering},
  29(2):1009--1083, 2022.

\bibitem{duffin1959distributed}
Richard~J Duffin.
\newblock Distributed and lumped networks.
\newblock {\em Journal of Mathematics and Mechanics}, 8(5):793--826, 1959.

\bibitem{d2012mixed}
Carlo D’Angelo and Anna Scotti.
\newblock A mixed finite element method for {D}arcy flow in fractured porous
  media with non-matching grids.
\newblock {\em ESAIM: Mathematical Modelling and Numerical Analysis},
  46(2):465--489, 2012.

\bibitem{favino2020fully}
Marco Favino, J{\"u}rg Hunziker, Eva Caspari, Beatriz Quintal, Klaus Holliger,
  and Rolf Krause.
\newblock Fully-automated adaptive mesh refinement for media embedding complex
  heterogeneities: application to poroelastic fluid pressure diffusion.
\newblock {\em Computational Geosciences}, 24(3):1101--1120, 2020.

\bibitem{flemisch2018benchmarks}
Bernd Flemisch, Inga Berre, Wietse Boon, Alessio Fumagalli, Nicolas Schwenck,
  Anna Scotti, Ivar Stefansson, and Alexandru Tatomir.
\newblock Benchmarks for single-phase flow in fractured porous media.
\newblock {\em Advances in Water Resources}, 111:239--258, 2018.

\bibitem{frih2012modeling}
Najla Frih, Vincent Martin, Jean~Elizabeth Roberts, and Ali Sa{\^a}da.
\newblock Modeling fractures as interfaces with nonmatching grids.
\newblock {\em Computational Geosciences}, 16(4):1043--1060, 2012.

\bibitem{fu2023hybridizable}
Guosheng Fu and Yang Yang.
\newblock A hybridizable discontinuous {G}alerkin method on unfitted meshes for
  single-phase {Darcy} flow in fractured porous media.
\newblock {\em Advances in Water Resources}, 173:104390, 2023.

\bibitem{fumagalli2013numerical}
Alessio Fumagalli and Anna Scotti.
\newblock A numerical method for two-phase flow in fractured porous media with
  non-matching grids.
\newblock {\em Advances in Water Resources}, 62:454--464, 2013.

\bibitem{glaser2022comparison}
Dennis Gl{\"a}ser, Martin Schneider, Bernd Flemisch, and Rainer Helmig.
\newblock Comparison of cell-and vertex-centered finite-volume schemes for flow
  in fractured porous media.
\newblock {\em Journal of Computational Physics}, 448:110715, 2022.

\bibitem{jiao2024enriched}
Kaituo Jiao, Dongxu Han, Yujie Chen, Bofeng Bai, Bo~Yu, and Shurong Wang.
\newblock The enriched-embedded discrete fracture model ({nEDFM}) for fluid
  flow in fractured porous media.
\newblock {\em Advances in Water Resources}, 184:104610, 2024.

\bibitem{karimi2004efficient}
Mohammad Karimi-Fard, Luis~J Durlofsky, and Khalid Aziz.
\newblock An efficient discrete-fracture model applicable for general-purpose
  reservoir simulators.
\newblock {\em SPE Journal}, 9(02):227--236, 2004.

\bibitem{karimi2003numerical}
Mohammad Karimi-Fard and Abbas Firoozabadi.
\newblock Numerical simulation of water injection in fractured media using the
  discrete-fracture model and the {G}alerkin method.
\newblock {\em SPE Reservoir Evaluation \& Engineering}, 6(02):117--126, 2003.

\bibitem{kim2000finite}
Jong-Gyun Kim and Milind~D Deo.
\newblock Finite element, discrete-fracture model for multiphase flow in porous
  media.
\newblock {\em AIChE Journal}, 46(6):1120--1130, 2000.

\bibitem{koppel2019lagrange}
Markus K{\"o}ppel, Vincent Martin, J{\'e}r{\^o}me Jaffr{\'e}, and Jean~E
  Roberts.
\newblock A {L}agrange multiplier method for a discrete fracture model for flow
  in porous media.
\newblock {\em Computational Geosciences}, 23(2):239--253, 2019.

\bibitem{kron1945numerical}
Gabriel Kron.
\newblock Numerical solution of ordinary and partial differential equations by
  means of equivalent circuits.
\newblock {\em Journal of Applied Physics}, 16(3):172--186, 1945.

\bibitem{li2008efficient}
Liyong Li and Seong~H Lee.
\newblock Efficient field-scale simulation of black oil in a naturally
  fractured reservoir through discrete fracture networks and homogenized media.
\newblock {\em SPE Reservoir evaluation \& engineering}, 11(04):750--758, 2008.

\bibitem{li2022high}
Xujing Li, Xiaodi Zhang, and Xinxin Zhou.
\newblock High order interface-penalty finite element methods for elliptic
  interface problems with {Robin} jump conditions.
\newblock {\em Computer Methods in Applied Mechanics and Engineering},
  390:114505, 2022.

\bibitem{liu2026high}
Jingyao Liu, Hui Guo, Zhuozheng Chen, and Ziyao Xu.
\newblock A high-order finite volume discrete fracture model for single-phase
  flow in fractured porous media.
\newblock {\em Advances in Water Resources}, page 105425, 2026.

\bibitem{liu2026interior}
Yong Liu and Ziyao Xu.
\newblock An interior penalty discontinuous {G}alerkin method for an interface
  model of flow in fractured porous media.
\newblock {\em Journal of Scientific Computing}, 107(3):83, 2026.

\bibitem{losapio2023local}
Davide Losapio and Anna Scotti.
\newblock Local embedded discrete fracture model ({LEDFM}).
\newblock {\em Advances in Water Resources}, 171:104361, 2023.

\bibitem{macneal1953asymmetrical}
Richard~H Macneal.
\newblock An asymmetrical finite difference network.
\newblock {\em Quarterly of Applied Mathematics}, 11(3):295--310, 1953.

\bibitem{martin2005modeling}
Vincent Martin, J{\'e}r{\^o}me Jaffr{\'e}, and Jean~E Roberts.
\newblock Modeling fractures and barriers as interfaces for flow in porous
  media.
\newblock {\em SIAM Journal on Scientific Computing}, 26(5):1667--1691, 2005.

\bibitem{moinfar2014development}
Ali Moinfar, Abdoljalil Varavei, Kamy Sepehrnoori, and Russell~T Johns.
\newblock Development of an efficient embedded discrete fracture model for {3D}
  compositional reservoir simulation in fractured reservoirs.
\newblock {\em SPE Journal}, 19(02):289--303, 2014.

\bibitem{noorishad1982upstream}
Jahan Noorishad and Mohsen Mehran.
\newblock An upstream finite element method for solution of transient transport
  equation in fractured porous media.
\newblock {\em Water Resources Research}, 18(3):588--596, 1982.

\bibitem{rashid2024continuous}
Harun~U Rashid and Olufemi Olorode.
\newblock A continuous projection-based {EDFM} model for flow in fractured
  reservoirs.
\newblock {\em SPE Journal}, 29(01):476--492, 2024.

\bibitem{schadle20193d}
Philipp Sch{\"a}dle, Patrick Zulian, Daniel Vogler, Sthavishtha~R Bhopalam,
  Maria~GC Nestola, Anozie Ebigbo, Rolf Krause, and Martin~O Saar.
\newblock 3{D} non-conforming mesh model for flow in fractured porous media
  using {L}agrange multipliers.
\newblock {\em Computers \& Geosciences}, 132:42--55, 2019.

\bibitem{schwenck2015dimensionally}
Nicolas Schwenck, Bernd Flemisch, Rainer Helmig, and Barbara~I Wohlmuth.
\newblock Dimensionally reduced flow models in fractured porous media:
  crossings and boundaries.
\newblock {\em Computational Geosciences}, 19(6):1219--1230, 2015.

\bibitem{tene2017projection}
Matei {\c{T}}ene, Sebastian~BM Bosma, Mohammed~Saad Al~Kobaisi, and Hadi
  Hajibeygi.
\newblock Projection-based embedded discrete fracture model ({pEDFM}).
\newblock {\em Advances in Water Resources}, 105:205--216, 2017.

\bibitem{xu2026fdm}
Ziyao Xu.
\newblock A non-conforming finite-difference discrete fracture model based on
  an energy principle.
\newblock Preprint, 2026.

\bibitem{xu2024box}
Ziyao Xu and Dennis Gl{\"a}ser.
\newblock An extension of the box method discrete fracture model ({Box-DFM}) to
  include low-permeable barriers with minimal additional degrees of freedom.
\newblock {\em Advances in Water Resources}, 195:104869, 2025.

\bibitem{xu2023hybrid}
Ziyao Xu, Zhaoqin Huang, and Yang Yang.
\newblock The hybrid-dimensional {D}arcy's law: a non-conforming reinterpreted
  discrete fracture model ({RDFM}) for single-phase flow in fractured media.
\newblock {\em Journal of Computational Physics}, 473:111749, 2023.

\bibitem{xu2020hybrid}
Ziyao Xu and Yang Yang.
\newblock The hybrid dimensional representation of permeability tensor: a
  reinterpretation of the discrete fracture model and its extension on
  nonconforming meshes.
\newblock {\em Journal of Computational Physics}, 415:109523, 2020.

\bibitem{zhang2013accurate}
Na~Zhang, Jun Yao, Zhaoqin Huang, and Yueying Wang.
\newblock Accurate multiscale finite element method for numerical simulation of
  two-phase flow in fractured media using discrete-fracture model.
\newblock {\em Journal of Computational Physics}, 242:420--438, 2013.

\bibitem{zhao2024discrete}
Jijing Zhao and Hongxing Rui.
\newblock A discrete fracture-matrix approach based on {P}etrov-{G}alerkin
  immersed finite element for fractured porous media flow on nonconforming
  mesh.
\newblock {\em Journal of Computational Physics}, 499:112718, 2024.

\bibitem{zhao2026petrov}
Jijing Zhao and Shuyu Sun.
\newblock Petrov-{G}alerkin immersed finite element method for
  {Darcy}/{Darcy}-{F}orchheimer flow in fractured porous media.
\newblock In {\em International Conference on Computational Science}, pages
  123--137. Springer, 2026.

\bibitem{zhao2018modeling}
Jinzhou Zhao, Youshi Jiang, Yongming Li, Xu~Zhou, and Ruisi Wang.
\newblock Modeling fractures and barriers as interfaces for porous flow with
  extended finite-element method.
\newblock {\em Journal of Hydrologic Engineering}, 23(7):04018024, 2018.

\end{thebibliography}

\end{document}